\documentclass[leqno,11pt]{amsart}

\usepackage{amssymb}
\usepackage{amsmath}
\usepackage[usenames, dvipsnames]{color}
\usepackage{enumitem}
\usepackage{amsthm}
\usepackage{mathrsfs}
\usepackage{amsrefs}
\numberwithin{equation}{section}

\usepackage{marginnote}
\usepackage{todonotes}
\usepackage{hyperref}
\usepackage{esint}
\usepackage{verbatim}
\usepackage{tikz}

\usepackage{hyperref}

\usepackage{geometry} 

\usepackage{etoolbox}

\theoremstyle{plain}
\newtheorem{theorem}{Theorem}[section]
\newtheorem{lemma}[theorem]{Lemma}
\newtheorem{corollary}[theorem]{Corollary}
\newtheorem{proposition}[theorem]{Proposition}

\newtheorem{definition}[theorem]{Definition}

\newtheorem{remark}[theorem]{Remark}

\newcommand{\R}{\mathbb{R}}
\newcommand{\N}{\mathbb{N}}

\newcommand{\bA}{\mathbf{A}}
\newcommand{\bB}{\mathbf{B}}
\newcommand{\A}{\mathcal{A}}

\newcommand{\M}{\mathcal{M}}

\newcommand{\W}{\mathcal{W}}
\newcommand{\hW}{\hat{\mathcal{W}}}

\makeatletter
\@namedef{subjclassname@2020}{\textup{2020} Mathematics Subject Classification}
\makeatother

\newcommand{\wei}[1]{\langle #1 \rangle}

\def\XXint#1#2#3{{\setbox0=\hbox{$#1{#2#3}{\int}$ }
\vcenter{\hbox{$#2#3$ }}\kern-.6\wd0}}

\title[Equations with singular-degenerate coefficients]{Well-posedness for a class of parabolic equations with singular-degenerate coefficients}

\author[J. Fang]{Junyuan Fang}
\address[J. Fang]{Department of Mathematics, University of Tennessee, 227 Ayres Hall,
1403 Circle Drive, Knoxville, TN 37996-1320 }
\email{jfang9@vols.utk.edu}

\author[T. Phan]{Tuoc Phan}
\address[T. Phan]{Department of Mathematics, University of Tennessee, 227 Ayres Hall,
1403 Circle Drive, Knoxville, TN 37996-1320}
\email{tphan2@utk.edu}


\begin{document}

\begin{abstract} This paper studies a class of linear parabolic equations with measurable coefficients in divergence form whose volumetric heat capacity coefficients are assumed to be in some Muckenhoupt class of weights.
As such, the coefficients can be degenerate, singular, or both degenerate and singular. A class of weighted parabolic cylinders with a non-homogeneous quasi-distance function, and a class of weighted parabolic Sobolev spaces intrinsically suitable for the class of equations, are introduced. Under some smallness assumptions on the mean oscillations of the coefficients, regularity estimates, existence, and uniqueness of weak solutions in the weighted Sobolev spaces are proved. To achieve the results, we apply the level-set argument introduced by Caffarelli and Peral.  Several weighted inequalities and a weighted Aubin-Lions compactness theorem for sequences in weighted parabolic Sobolev spaces are established.
\end{abstract}
\maketitle
\section{Introduction} 
\subsection{Problem settings and motivations} \label{mov-sec} This work investigates the existence, uniqueness, and regularity estimates of solutions to the following class of parabolic equations  
\begin{equation} \label{main-eqn}
\left\{
\begin{array}{cccl}
\beta(x) u_t - \textup{div}(\bA(x,t) \nabla u) & = & \textup{div}(F)  & \quad  \text{in} \quad \Omega_T, \\[4pt]
u & =  & 0 & \quad \text{on} \quad \partial_p \Omega_T.
\end{array} \right.
\end{equation}
Here in \eqref{main-eqn}, $T \in (0,\infty)$ is fixed, $\Omega \subset \mathbb{R}^n$ is an open bounded set with boundary $\partial \Omega$, 
\begin{equation*} 
\Omega_T = \Omega \times (0, T], \quad \text{and} \quad \partial_p \Omega_T = (\overline{\Omega} \times \{0\}) \cup (\partial \Omega \times (0, T]).
\end{equation*}
Also, $F: \Omega_T \rightarrow \mathbb{R}^n$ is a given forcing term which is assumed to be in suitable Lebesgue spaces, and $\bA: \Omega_T \rightarrow \mathbb{R}^{n\times n}$ is a measurable matrix-valued function satisfying the following ellipticity and boundedness conditions: there is a constant $\nu \in (0,1)$ such that
\begin{equation} \label{ellip-cond}
\nu |\xi|^2 \leq \wei{\bA(x,t) \xi, \xi} \quad \text{and} \quad |\bA(x,t)| \leq \nu^{-1}, 
\end{equation}
for all $(x,t) \in \Omega_T$ and for all $\xi=(\xi_1, \xi_2,\ldots, \xi_n) \in \mathbb{R}^n$. Moreover, $\beta: \mathbb{R}^n \rightarrow [0, \infty)$ is a measurable function that is assumed to be in some Muckenhoupt class of weights. Due to this, the coefficient $\beta$ could vanish or blow up at some points $x \in \overline{\Omega}$.  As such, the coefficients of equation \eqref{main-eqn} could be singular or degenerate.  When $\beta =1$, the equation \eqref{main-eqn} is reduced to the classical second order parabolic equations in divergence form. 

\smallskip
In terms of modeling, the coefficient $\beta$ in \eqref{main-eqn} represents the volumetric heat capacity, which is the product of the material's density and its specific heat capacity. Meanwhile, the coefficient matrix $\bA$ is known as the thermal conductivity matrix, reflecting the anisotropic or orthotropic nature of materials in which heat conduction varies with directions. In this work, the coefficient matrix $\bA$ is assumed to be uniformly elliptic and bounded as in \eqref{ellip-cond}. Our primary concern is on the volumetric heat capacity coefficient $\beta$, where $\beta$ is allowed to blow up (be singular) or vanish (be degenerate) at certain points in $\overline{\Omega}$. The problem setting of this work is different from those available in the research line such as \cites{Chia-1, Chia-3, Fabes-1, Fabes, Tru-1, Tru-2, BS} and \cites{CMP, CJP, DPS, DP, DP1, DP2, DP3, DPT1, DPT2, MNS, MNS-1, MP}, in which the singularity and degeneracy typically arise either solely from the coefficient matrix $\bA$ or from both $\bA$ and $\beta$ in some balanced way.

\smallskip
Mathematical models in which the volumetric heat capacity coefficient $\beta$ becomes singular or degenerate play a vital role in accurately capturing heat transfer in materials undergoing phase transitions, or in domains involving vacuum,  and heterogeneous media. In Stefan type problems, for example, the inclusion of latent heat leads to an effective heat capacity that may become extremely large or even singular at the phase-change temperature. See \cites{AS, AA, PVPF, FKRH, Visintin} for details and related models. On the other hand, degenerate models arise in systems involving vacuum layers or low-density gases, where the heat capacity and conductivity may vanish locally, causing a breakdown of standard parabolicity and requiring specialized analytical tools such as entropy methods and weighted energy estimates. For example, see \cites{DIT, Li-Xi, PSW}. These singular and degenerate behaviors are not mere mathematical abstractions; they are essential for modeling thermal energy storage, composite materials, and insulation systems, where an accurate representation of sharp transitions or near-inert thermal regions is crucial for both physical fidelity and engineering design. More details and discussions can be found in  \cites{AS, crank, FKRH, HGH, Visintin}, for instance.

\smallskip
We also note that equations with singular or degenerate coefficients, as in \eqref{main-eqn}, arise in pure mathematics as well.  In particular, in the study of fine regularity and quantitative estimates of solutions to nonlinear equations, the class of equations \eqref{main-eqn} arises as the linearized equations of such nonlinear problems. In this context, the unknown solutions are implicitly encoded in the coefficients, which naturally become singular or degenerate. See \cites{DH, JX1, JX2, KLY} for examples of the study of nonlinear fast diffusion equations arising in porous media. See also \cites{Pop-2, JX, Le, STT} for a few examples on the studies of equations with singular-degenerate coefficients in mathematical finance, and geometric analysis.

\smallskip
This paper provides sufficient conditions on $\beta$ and the analysis of weighted Sobolev spaces through which the existence, uniqueness, and regularity estimates for solutions to \eqref{main-eqn} are proved. Besides the potential applications in both pure and applied mathematics that we just addressed, this work is the continuation of the research program in \cites{CMP, CJP, DPS, DP, DP1, DP2, DP3, DPT1, DPT2, MP}  which investigates fine properties and estimates of solutions to elliptic and parabolic equations with singular-degenerate coefficients. It is important to emphasize that this work provides a forward step in the program, as the class of equations \eqref{main-eqn} is not studied in the mentioned work. The results obtained in this paper, therefore, complement well those established in \cites{CMP, CJP, DPS, DP, DP1, DP2, DP3, DPT1, DPT2, MP}. They can be also considered as a further development in the research line \cites{Chia-1, Chia-2, Chia-3, BS, Fabes-1, Fabes, Tru-1, Tru-2} in which the H\"{o}lder regularity theory is established for elliptic equations with singular-degenerate coefficients in various settings.

\subsection{Main results} To state the main results, let us introduce a class of weighted function spaces that are relevant to \eqref{main-eqn}. For each $p \in (1, \infty)$ and for a given non-negative and locally integrable function $\beta$, we denote $\W^{1,p}(\Omega_T, \beta)$ the weighted Sobolev space consisting of functions $u \in L^p((0, T), W^{1,p}(\Omega))$ such that
\[
\beta(x) u_t \in L^p((0, T), W^{-1, p}(\Omega)).
\]
The space $\W^{1,p}(\Omega_T, \beta)$ is endowed with the norm
\begin{equation}\label{Norm-W1p}
\|u\|_{\W^{1,p}(\Omega_T, \beta)} = \|u\|_{L^p((0, T), W^{1,p}(\Omega))} + \| \beta u_t\|_{L^p((0, T), W^{-1, p}(\Omega))}, 
\end{equation}
for $u \in \W^{1,p}(\Omega_T, \beta)$. Here, let us recall that $W^{1,p}(\Omega)$ denotes the usual Sobolev space, and $W^{1,p}_0(\Omega)$ is the completion of the set of compactly supported functions $C_0^\infty(\Omega)$ in $W^{1,p}(\Omega)$. In addition, $W^{-1,p}(\Omega)$ is the dual space of $W^{1,p'}_0(\Omega)$, where $\frac{1}{p'}+ \frac{1}{p} =1$.  The analysis properties, weighted inequalities, and compactness of sequences of functions related to $\W^{1,p}(\Omega_T, \beta)$ are investigated in Section \ref{wei-ses} and Section \ref{Fn-spaces} below. 

\smallskip
We also denote by $\W^{1,p}_*(\Omega_T, \beta)$ the closure in $\W^{1,p}(\Omega_T)$ of the set 
\[
\big\{ \varphi \in C^\infty([0, T], C_0^\infty(\Omega)):\ \varphi=0 \ \text{in}\ \Omega \times [0, \epsilon]  \ \text{for some sufficiently small}\ \epsilon \in (0,T)\big\}.
\]
The space $\W^{1,p}_*(\Omega_T, \beta)$ is endowed with the same norm as in \eqref{Norm-W1p}, that is,
\[
\|u\|_{\W^{1,p}_{*} (\Omega_T, \beta)} = \|u\|_{L^p((0, T), W^{1,p}(\Omega))} + \| \beta u_t\|_{L^p((0, T), W^{-1, p}(\Omega))}
\]
for $u \in \W^{1,p}_*(\Omega_T, \beta)$. A function $u \in \W^{1,p}_{*} (\Omega_T, \beta)$ is said to be a weak solution to \eqref{main-eqn} if
\begin{align*}
& -\int_{\Omega_T} \beta(x) u(x,t) \partial_t \varphi(x,t) dx dt + \int_{\Omega_T} \wei{\bA(x,t) \nabla u(x,t), \nabla \varphi (x,t)} dx dt \\
& = -\int_{\Omega_T} \wei{F(x,t), \nabla \varphi(x,t)} dx dt, 
\end{align*}
for every $\varphi \in C^\infty([0, T], C^\infty_0(\Omega))$ satisfying $\varphi =0$ in $\Omega \times [T-\epsilon, T]$ for some sufficiently small $\epsilon \in (0,T)$.

\smallskip
Besides \eqref{ellip-cond} and \eqref{beta-cond}, we also assume that the mean oscillations of the matrix $\bA$ and the weight $\beta$ are sufficiently small. To introduce this, we need several notations. For a locally integrable function $f$ defined on $\Omega\subset \R^{n}$ and for a measurable subset $E\subset \Omega$, we denote 
\[
(f)_{E}=\fint_{E}f(x)dx = \frac{1}{|E|} \int_{E} f(x) dx,
\]
where $|E|$ denotes the Lebesgue measure of $E$. In addition, for $r>0$ and $x_0 \in \mathbb{R}^n$, $B_r(x_0)$ denotes the ball in $\mathbb{R}^n$ of radius $r$ centered at $x_0$.  Then, the mean oscillation of $\bA$ and the weighted mean oscillation of $\beta$ are defined by 
\begin{equation} \label{BMO-defi}
\begin{split}
& \Theta_{\bA}(\Omega_T, R_0)   = \sup_{\rho \in (0, R_0)}\, \sup_{z_0 =(x_0, t_0) \in \Omega_T}\Big(\fint_{Q_{\rho, \beta}(z_0) \cap \Omega_T} \Theta_{\bA, r, x_0} (x,t)^2 dxdt\Big)^{1/2}, \\[4pt]
& \Theta_{\beta}(\Omega, R_0) = \sup_{\rho \in (0, R_0)}\,  \sup_{x_0 \in \Omega}\Big(\frac{1}{\beta(B_{\rho}(x_0))}\int_{B_\rho(x_0)} \Theta_{\beta, r, x_0}(x)^2dx\Big)^{1/2}.
\end{split}
\end{equation}
In \eqref{BMO-defi}, $Q_{\rho,\beta}(z_0)$ is the weighted parabolic cylinder of radius $\rho>0$ centered at $z_0$ (see Definition \eqref{cylinder-def} below for details), 
\[
\beta(E)=\int_{E}\beta(x)dx
\] for every bounded measurable set $E \subset\R^n$, and 
\begin{equation} \label{point-oss-def}
\begin{split}
& \Theta_{\bA, r, x_0} (x,t)= |\bA(x,t) - (\bA)_{B_{\rho}(x_0) \cap \Omega} (t)|, \\[4pt]
& \Theta_{\beta, r, x_0}(x) = | \beta(x) - (\beta)_{B_{\rho}(x_0)}| \beta(x)^{-1/2}.
\end{split}
\end{equation}
In addition, we assume throughout the paper that the weight $\beta$ in \eqref{main-eqn}  satisfies the following conditions
\begin{equation} \label{beta-cond}
\beta^{-1} \in A_{1+\frac{2}{n_0}} \quad \text{and} \quad [\beta^{-1}]_{A_{1+\frac{2}{n_0}}} \leq M_0,  \quad \text{for} \quad n_0 = \max\{n, 2\}
\end{equation}
and for some given $M_0 \geq 1$. See Definition \ref{A-p-def} for more details on the definition of the $A_{q}$ class of Muckenhoupt weights for $q \in [1, \infty)$. By Remark \ref{A-p-remark}-(ii) below, we see that \eqref{beta-cond} is equivalent to 
\[
\beta^{\frac{n_0}{2}} \in A_{1+ \frac{n_0}{2}} \quad \text{and} \quad [\beta^{\frac{n_0}{2}}]_{A_{1+\frac{n_0}{2}}} \leq M_0^{\frac{n_0}{2}}.
\]

\bigskip
The following theorem is the main result of the paper.
\begin{theorem} \label{main-theorem} Let $p \in (1, \infty)$, $\nu \in (0,1)$, and $M_0 \geq 1$. Then there exists a constant $\delta = \delta (n, \nu, p, M_0) \in (0,1)$ sufficiently small such that the following assertions hold. Assume that  \eqref{ellip-cond} and \eqref{beta-cond} hold, $\partial \Omega \in C^1$, and
\begin{equation} \label{small-ness}
\Theta_{\bA}(\Omega_T, R_0) + \Theta_{\beta} (\Omega, R_0)< \delta,
\end{equation}
for some $R_0 \in (0,1)$. Then, for every $F \in L^p(\Omega_T)^n$ with $T>0$, there exists a unique weak solution $u \in \W^{1,p}_{*} (\Omega_T, \beta)$ solving \eqref{main-eqn}. Moreover,
\begin{equation} \label{main-thm-est}
\|u\|_{\W^{1,p}_{*}(\Omega_T,\beta)}\leq N \|F\|_{L^p(\Omega_T)},
\end{equation}
where $N = N(n, \nu, p, M_0, R_0, \Omega, T) >0$.
\end{theorem}

\smallskip
The novelty of this work lies in the discovery of the conditions \eqref{beta-cond} and \eqref{small-ness} on the coefficient $\beta$ in \eqref{main-eqn}, through which the analysis theory on the Sobolev space $\W^{1,p}(\Omega_T, \beta)$ is developed and an Aubin-Lions type compactness theorem (Proposition \ref{compactness-lemma} below) is established in Section \ref{wei-ses} and Section \ref{Fn-spaces} below. In addition, a quasi-distance geometrically encoded in \eqref{main-eqn} is discovered, and the corresponding non-homogeneous weighted parabolic cylinders are introduced in Section \ref{wei-ses}. From these results, a perturbation method can be implemented and the regularity estimates and the existence of solutions in $\W^{1,p}_{*}(\Omega_T, \beta)$ stated in Theorem \ref{main-theorem} are proved. To the best of our knowledge, these results and Theorem \ref{main-theorem} appear for the first time in the literature.

\smallskip
We note that $\Theta_{\bA}(\Omega_T, R_0) $ defined in \eqref{BMO-defi} is usually referred to as partial mean oscillations of $\bA$, a concept introduced in \cites{KK1, KK2}. It is clear that if the spatial mean oscillation of $\bA(x,t)$ is sufficiently small in small balls uniformly in the time variable, then $\bA$ satisfies \eqref{small-ness}.  For example, we consider $\bA(x,t) = a(t) \tilde{\bA}(x)$, then $\bA$ satisfies \eqref{small-ness} if $a(t)$ is bounded and measurable,  and the mean oscillation of $\tilde{\bA}(x)$ is sufficiently small. In other words, the smallness on the mean oscillation of $a(t)$ is not required. We also note that $\Theta_{\beta}(\Omega_T, R_0)$ is known as the weighted mean oscillation of $\beta$ with respect to the weight $\beta$. See \cite{M-W} for more details on the analysis of weighted mean oscillation functions. As demonstrated in \cites{CMP, CJP}, a typical example of $\beta$ satisfying \eqref{small-ness} is $\beta(x) = |x|^\alpha$ for $x \in \Omega = B_1$ and $|\alpha|$ sufficiently small.  

\smallskip
It is important to emphasize that similar interior and local boundary regularity estimates, as well as local-in-time but global in $\Omega$ regularity estimates, are also established in this paper. See Theorems \ref{inter-theorem}, \ref{inter-theorem-bdry}, and \ref{main-thm-2} below. In particular, when $\beta \equiv 1$, our results are reduced to known results; for example, see \cites{B, BW1, BW2, KK1, KK2}.  More importantly, it is possible to relax the condition $\partial \Omega \in C^1$ and replace it by some flatness assumptions such as those considered in \cites{B, BW1, BW2, HNP}. Similarly, existence and regularity estimates of solutions to \eqref{main-eqn} in weighted, weighted mixed-norm spaces, or Lorentz spaces as in \cites{DK, DP1, DP2, DPT1, DPT2, Phan} can also be established. However, we do not pursue these directions to avoid additional technical complexity.  

\subsection{Main difficulties and approaches} 
We establish the regularity estimates for solutions by using the level set method introduced in \cite{CP}. To implement the approach, we use a perturbation argument combined with the freezing coefficient technique. This requires to locally approximate the solutions of \eqref{main-eqn} in each parabolic cylinder of small radius $r$ centered at $z_0 = (x_0, t_0) \in \overline{\Omega}_T$ by solutions to auxiliary equations whose coefficients are the 
spatial averages over the ball $B_r(x_0)$ of the coefficients in \eqref{main-eqn}:
\begin{equation}\label{v-eqn-main-idea}
(\beta)_{B_r(x_0)} v_t - \text{div}\left(\bA_{B_r(x_0)}(t) \nabla v\right) = 0.
\end{equation}
The non-standard part in this process is that the equation for $u-v$ contains perturbation terms consisting of the time derivative of unknown solutions multiplied by $\beta(x) - (\beta)_{B_r(x_0)}$. Due to the lack of regularity in the derivative in time variable for the corresponding parabolic equations and of $\beta(x) - (\beta)_{B_r(x_0)}$, it is not possible to use standard energy estimates to directly obtain the smallness of $u-v$ in energy spaces. To overcome this, we employ a contradiction argument to control $u$ and $u-v$ in energy spaces. These tasks are accomplished in Lemmas \ref{u-L2-Cac}, \ref{L2-comparision}, and Proposition \ref{L2-gradient-comparision} for the interior case, and in Lemmas \ref{bdry-Poincare}, \ref{L2-comparision-bdry}, and Proposition \ref{L2-gradient-comparision-bdry} for the boundary case. The challenge in the contradiction argument is that the coefficients $\beta$ vary, and then the function spaces $\W^{1,2}(Q_{r, \beta}, \beta)$ to which solutions belong also change. A key ingredient in this analysis relies on a compactness theorem, Proposition \ref{compactness-lemma}, in which a compactness result is proved for a sequence of functions $\{u_k\}_k$ with $u_k \in \W^{1,2}(Q_{r, \beta_k}, \beta_k)$ for each $k$. The conditions on the weights given in \eqref{beta-cond} and \eqref{small-ness} are used crucially in these arguments. The Proposition \ref{compactness-lemma} can be considered as a general and weighted version of the classical Aubin-Lions compactness theorem. We note that as $\beta$ is not nice, some lemmas on energy estimates, such as Lemma \ref{u-L2-Cac} and Lemma \ref{bdry-Poincare}, are not standard and they seem to appear for the first time in this work. 

\smallskip
It is also worth mentioning that for the perturbation method to work, Lipschitz estimates that are uniform in $r$, $x_0$, and $(\beta)_{B_r(x_0)}$ are needed for solutions to the class of equations \eqref{v-eqn-main-idea}.  This is not possible in standard parabolic cylinders as the ellipticity constants in \eqref{v-eqn-main-idea} depend on $(\beta)_{B_r(x_0)}$, which can be very tiny or extremely large as $r$ and $x_0$ vary depending on the degeneracy or singularity of $\beta$. To overcome this difficulty, we introduce in Section \ref{cylinder-distance-sec} a family of non-homogeneous parabolic cylinders that are invariant under scalings and dilations for \eqref{main-eqn}. See Section \ref{scaling-sec} below for some special dilation properties of the class of equations \eqref{main-eqn}. In addition, a non-homogeneous quasi-distance function encoded in \eqref{main-eqn} is also discovered and proved in this section, in which the condition \eqref{beta-cond} is employed. Based on these constructions, and by employing suitable scalings and dilations, we establish the desired uniform Lipschitz estimates in the weighted parabolic cylinders for \eqref{v-eqn-main-idea} and its boundary counterpart in Section \ref{Lips-est-sec}.
\subsection{Relevant literature}
The literature on regularity theory and well-posedness of solutions to singular or degenerate elliptic and parabolic equations is extremely rich. Let us only describe results closely related to this work.

\smallskip
Theories on well-posedness and regularity estimates in Sobolev spaces for equations with singular or degenerate coefficients have recently been developed; see \cites{BDGP, BD, CMP, DR, DPS, DP, DP1, DP2, DP3,  DPT1, DPT2, MNS, MNS-1}, for examples. In particular, the paper \cites{CMP} studies a class of elliptic equations in which the singularity and degeneracy arise from $\bA$. Well-posedness and regularity estimates in suitable weighted Sobolev spaces were proved under some smallness assumption on a weighted mean oscillation of the coefficients $\bA$. See also \cite{BDGP} for similar results. It is noted that the problem settings in \cites{BDGP, CMP} are similar to those in the classical works \cites{Chia-1, Chia-2, Chia-3, Fabes, Fabes-1, Tru-1, Tru-2} in which Harnack inequalities and H\"{o}lder regularity estimates are established. In a different direction, classes of elliptic and parabolic equations of the Caffarelli-Silvestre extension type are investigated in \cites{BD, DPS, DP, DP1, DP2, DP3, MNS, MNS-1, MP}. Similar well-posedness and regularity estimates were developed for weighted and mixed-norm weighted Sobolev spaces. We also note that H\"{o}lder and Schauder regularity estimates for similar classes of equations were recently developed in \cites{STV1, STV2} for linear elliptic equations, \cites{AFV, AFV2} for linear parabolic equations,  and in \cite{JS} for fully nonlinear elliptic equations. 

\smallskip
The works most closely related to the present paper are \cites{DR, CJP, JX, DPT1, DPT2} and \cites{JX1, JX2}. In \cites{DPT1, DPT2}, the following class of equations is studied
\begin{equation} \label{DPT-eqn}
x_n^{-\alpha} u_t - \text{div}(\bA(x,t) \nabla u)  =f, \quad   x = (x', x_n) \in \mathbb{R}^{n-1} \times (0, \infty), \ t \in  \ (0,\infty),
\end{equation}
where $\alpha \in (0, 2)$. Regularity estimates in suitable weighted Sobolev and mixed-norm weighted Sobolev spaces are proved. In recent work \cite{DR}, the same type of equations as \eqref{DPT-eqn} with $\alpha=2$ and additional lower order terms is studied, and similar results are established. See also the recent work \cite{JX} for similar results in which equations with singular-degenerate coefficients on the boundary with co-dimension 2. In addition, results on H\"{o}lder and Schauder regularity estimates were proved in \cites{JX1, JX2} for a similar class of equations as \eqref{DPT-eqn} that arises in porous media. It is also worth mentioning that a class of equations similar to \eqref{main-eqn}, but in non-divergence form, has been studied recently in \cite{CJP}. Under some conditions on $\beta$ which are slightly different from \eqref{beta-cond}, Krylov-Safonov Harnack inequalities and H\"{o}lder regularity estimates were proved.

\smallskip

Note that the class of equations in \eqref{DPT-eqn} differs from \eqref{main-eqn} because the coefficient $\beta$ in \eqref{main-eqn} is general and it can be singular or degenerate anywhere in the whole domain $\overline{\Omega}$, whereas in \eqref{DPT-eqn}, the considered weight is specific and is singular along the boundary. In addition, $\partial \Omega$ in \eqref{main-eqn} is not flat as in \eqref{DPT-eqn}. As such, the class of equations \eqref{main-eqn} has not been studied in the existing literature. The results of the paper are new, and they have potential applications as discussed in Section \ref{mov-sec}. Note that the key difference of this work is the introduction of \eqref{beta-cond}, through which a class of weighted parabolic cylinders endowed with a weighted quasi-distance function is introduced. Moreover, a general and weighted version of the Aubin-Lions compactness theorem is proved. These results are of independent interest. To the best of our knowledge, the results in this paper are the first available ones in the research line that treats the well-posedness and regularity of solutions to the parabolic equation \eqref{main-eqn} in Sobolev spaces. 
\subsection{Organization of the paper} The rest of the paper is organized as follows. In the next section, Section \ref{wei-ses}, we recall the definitions of classes of Muckenhoupt weights and summarize several of their essential properties that we need. Some important weighted inequalities are also stated and proved in this section. Then, in Section \ref{Fn-spaces}, we define the functional spaces needed in the paper and prove their essential properties. In particular, a weighted version of the Aubin-Lions compactness theorem is proved in this section. In Section \ref{Lips-est-sec}, the interior and boundary Lipschitz estimates are established for simple equations. The interior regularity estimate theorem, Theorem \ref{inter-theorem}, local energy estimates, and local interior perturbation techniques are performed and proved in Section \ref{interior section}. The corresponding boundary theory is presented in Section \ref{bdr-section}, where Theorem \ref{inter-theorem-bdry} provides boundary regularity estimates on flat portions of the boundary. Then, in Section \ref{global-section}, Theorem \ref{main-thm-2}, which provides spatially global regularity estimates for solutions, is stated and proved using boundary flattening techniques and partition of unity. The proof of Theorem \ref{main-theorem} is also given in this section. We conclude the paper with several appendices that provide the proofs of several technical lemmas and estimates.

\section{Weighted inequalities and weighted parabolic cylinders}  \label{wei-ses}
\subsection{Weights and weighted inequalities}
\begin{definition} \label{A-p-def}  Let $q \in (1, \infty)$, and $\mu : \mathbb{R}^n \rightarrow [0, \infty)$ be a locally integrable function. We say that $\mu$ belongs to the class of \emph{$A_q$ Muckenhoupt  weights} if $[\mu]_{A_q} < \infty$, where
\[
[\mu]_{A_q}=\sup_{r>0,\, x_0 \in \mathbb{R}^n}\left(\fint_{B_r(x_0)}\mu(x)\, dx \right)\left(\fint_{B_r(x_0)}\mu(x)^{-\frac{1}{q-1}}\, dx\right)^{q-1}.
\]
\end{definition}
\begin{remark} \label{A-p-remark} The following important properties on Muckenhoupt weights follow directly from the definition.
\begin{itemize}
\item[\textup{(i)}] If $\mu \in A_{q}$ with $q \in (1, \infty)$, then  $[\mu]_{A_q} \geq 1$.
\item[\textup{(ii)}] If $\mu$ is a weight such that $\mu^{-1} \in A_{1+\frac{2}{n_0}}$ for $n_0 = \max\{n, 2\}$, then $\mu^{\frac{n_0}{2}} \in A_{1+ \frac{n_0}{2}}$ and
\[
[\mu^{\frac{n_0}{2}}]_{A_{1+\frac{n_0}{2}}} = [\mu^{-1}]_{A_{1+ \frac{2}{n_0}}}^{\frac{n_0}{2}}.
\]
Moreover, $\mu \in A_2$ and $[\mu]_{A_2} \leq [\mu^{-1}]_{A_{1+\frac{2}{n_0}}}$.
\end{itemize}
\end{remark}
\begin{proof} We start with the proof of (i). As $q\in (1,\infty)$,  for any ball $B\subset \R^n$, it follows from H\"{o}lder's inequality  that 
\begin{align} \label{1-mu-ball}
1 = \fint_{B} \mu(x)^{\frac{1}{q}} \mu(x)^{-\frac{1}{q}} dx \leq  \left(\fint_{B} \mu(x) dx \right)^{\frac{1}{q}} \left(\fint_{B} \mu(x)^{-\frac{1}{q-1}} dx \right)^{1-\frac{1}{q}}\leq [\mu]_{A_q}^{\frac{1}{q}}.
\end{align}
Hence, (i) follows for $q>1$. 

\smallskip
Note that the first assertion in (ii) follows directly from Definition \ref{A-p-def} with some simple manipulation.  For the second assertion in (ii), it is sufficient to consider the case $n\geq 3$. For any ball $B \subset \mathbb{R}^n$, by H\"{o}lder's inequality, we get
\[
\fint_{B} \mu(x) dx \leq \left(\fint_{B} \mu(x)^{\frac{n}{2}} dx \right)^{\frac{2}{n}}.
\]
Therefore,
\begin{align*}
[\mu]_{A_2}&= \sup_{B}\left(\fint_{B} \mu(x) dx\right)\left(\fint_{B} \mu(x)^{-1} dx\right)\\
&\leq \sup_{B} \left(\fint_{B} \mu(x)^{\frac{n}{2}} dx \right)^{\frac{2}{n}}\left(\fint_{B} \mu(x)^{-1} dx\right)= [\mu^{-1}]_{A_{1+\frac{2}{n}}}.
\end{align*}
The proof is then completed.
\end{proof}

\smallskip
As the weight $\beta(x)$ considered throughout the paper satisfies \eqref{beta-cond}, i.e.,
\begin{equation*}
\beta^{-1} \in A_{1+\frac{2}{n_0}} \quad \text{and} \quad [\beta^{-1}]_{A_{1+\frac{2}{n_0}}} \leq M_0  \quad \text{for } n_0 =\max\{n,2\},
\end{equation*}
it follows from Remark~\ref{A-p-remark}\, (ii) that $\beta^{\frac{n_0}{2}} \in A_{1+\frac{n_0}{2}}$ and $\beta\in A_2$, with
\[
[\beta^{\frac{n_0}{2}}]_{A_{1+\frac{n_0}{2}}}\leq M_0^{\frac{n_0}{2}} \quad \text{and}\quad [\beta]_{A_2}\leq M_0.
\]

Consequently, they have useful doubling properties stated in the following lemma. We refer the reader to \cite[Proposition~7.2.8, p.~521]{Grafakos-2} for the proof and for related results on more general weights.

\begin{lemma} \label{property} Let $M_0 \geq 1$, and let $\beta$ be a weight satisfying \eqref{beta-cond}. For $\bar{\beta}(x) = [\beta(x)]^{\frac{n_0}{2}}, \ x \in \R^n$, the following assertions hold.
\begin{itemize}
\item[\textup{(i)}] There exists a constant $N_1 = N_1(n, M_0) > 1$, called the doubling constant, such that
\begin{equation*} \label{doubling const}
\bar{\beta}(B_{2\rho}(x)) \leq N_1\, \bar{\beta}(B_\rho(x)) \quad \text{for all balls } B_\rho(x) \subset \mathbb{R}^n.
\end{equation*}
\item[\textup{(ii)}] There exist constants $N_2 = N_2(n, M_0) > 0$ and $\zeta_0 = \zeta_0(n, M_0) \in (0,1)$ such that
\begin{equation*} \label{reverse holder}
\bar{\beta}(S_1) \leq N_2 \left( \frac{|S_1|}{|S_2|} \right)^{\zeta_0} \bar{\beta}(S_2) \quad \text{for all measurable sets } S_1 \subset S_2 \subset \mathbb{R}^n.
\end{equation*}
\item[\textup{(iii)}] There exists a constant $\eta = \eta(n, M_0, \theta) \in (0,1)$ such that
\[
\bar{\beta}(S_1) \leq \eta\, \bar{\beta}(S_2)
\]
for all measurable sets $S_1 \subset S_2 \subset \mathbb{R}^n$ with $|S_1| \leq \theta\, |S_2|$ and $\theta \in (0,1)$. In particular, one may take
\[
\eta = 1 - (1 - \theta)^{1 + \frac{n_0}{2}} M_0^{-\frac{n_0}{2}}.
\]
\end{itemize}
Similar conclusions also hold for $\beta$ in place of $\bar{\beta}$.
\end{lemma}

The following lemma on the reverse H\"{o}lder property of $A_q$-weights is well known. We also refer the reader to~\cite[Theorem~7.2.2, p.~514; Corollary~7.2.6, p.~519]{Grafakos-2} for the proof.
\begin{lemma}\label{R-Holder} Let $M_0 \geq 1$ and $q \in (1, \infty)$. There exist a sufficiently small constant $\gamma = \gamma(n,q, M_0)>0$ and a constant $N = N(n, q, M_0)\geq 1$ such that the following assertions hold.
\begin{itemize}
\item[\textup{(i)}] If $\mu \in A_{q}$ satisfies $[\mu]_{A_q} \leq M_0$, then for every ball $B\subset\mathbb{R}^n$,
\[
\left( \fint_{B} \mu(x)^{1+\gamma} dx \right)^{\frac{1}{1+\gamma}} \leq N \fint_{B} \mu(x) dx.
\]

\item[\textup{(ii)}] If $\mu \in A_{q}$ satisfies $[\mu]_{A_q} \leq M_0$, then $\mu \in A_{q-\gamma}$ and
\[ [\mu]_{A_{q-\gamma}} \leq N.
\]
\end{itemize}
\end{lemma}
Next, let us state and prove the following lemma that provides the control of the un-weighted Lebesgue norm by an $L^{2}$-weighted norm.
\begin{lemma}\label{L-q-2-wei} Let $q \in (1,2]$, $M_0 \geq 1$, and $ \mu \in A_q $ with $[\mu]_{A_q} \leq M_0 $. Then, there exist a sufficiently small constant $\gamma = \gamma(n,  q, M_0) \in (0, q-1)$ and a constant $ N = N(n, q, M_0) > 0$ such that  for every ball $ B \subset \mathbb{R}^n $ and for $ g \in L^2(B, \mu)$,
\[
\left( \fint_{B} |g(x)|^{\frac{2}{q - \gamma}} \, dx \right)^{\frac{q - \gamma}{2}} 
\leq N \left( \frac{1}{\mu(B)} \int_{B} |g(x)|^2 \mu(x) \, dx \right)^{1/2}.
\]
\end{lemma}
\begin{proof} By Lemma \ref{R-Holder}, there exists a sufficiently small constant $\gamma = \gamma(n, q, M_0) \in (0, q-1)$ such that $\mu \in A_{q-\gamma}$ and
\begin{equation}\label{mu-q-gamma-1}
[\mu]_{A_{q-\gamma}} \leq N_0,
\end{equation}
where $N_0 = N_0(n, q, M_0)\geq 1$. We write
\[
\fint_{B} |g(x)|^{\frac{2}{q-\gamma}} dx = \fint_{B} \bigl[|g(x)|^2\mu(x)\bigr]^{\frac{1}{q-\gamma}} \mu(x)^{-\frac{1}{q-\gamma}}dx,
\]
and apply H\"{o}lder's inequality together with \eqref{mu-q-gamma-1} to obtain
\begin{align*}
\left(\fint_{B} |g(x)|^{\frac{2}{q-\gamma}} dx \right)^{\frac{q-\gamma}{2}}& \leq \left(\fint_{B} |g(x)|^2 \mu(x) dx \right)^{\frac{1}{2}} \left(\fint_{B} \mu(x)^{-\frac{1}{q-\gamma-1}} dx \right)^{\frac{q- \gamma-1}{2}} \\
& \leq [\mu]_{A_{q-\gamma}}^{\frac{1}{2}}\left(\fint_B\mu(x)dx\right)^{-\frac{1}{2}} \left(\fint_{B} |g(x)|^2 \mu(x) dx \right)^{\frac{1}{2}} \\
& \leq N\left(\frac{1}{\mu(B)}\int_{B} |g(x)|^2 \mu(x) dx \right)^{\frac{1}{2}},
\end{align*}
where $N = N(n, q, M_0)\geq 1$. This completes the proof.
\end{proof}

\begin{lemma} \label{wei-Les} For every $M_0 \geq 1$, there exist constants $\gamma=\gamma(n, M_0) \in (0, \frac{2}{n})$ sufficiently small and $N = N(n, M_0)\geq 1$ such that the following assertions hold for every weight $\beta$ satisfying \eqref{beta-cond}.
\begin{itemize}
\item[\textup{(i)}] If $n\geq 3$, then for every $g \in L^{\frac{2n}{n(1+\gamma)-2}}(B_1)$,
\begin{equation} \label{beta-1.inq}
 \frac{1}{\beta(B_1)}\int_{B_1} g(x)^2 \beta(x) dx \leq N \left(\fint_{B_{1}} |g(x)|^{\frac{2n}{n(1+\gamma)-2}} dx \right)^{\frac{n(1+\gamma)-2}{n}}.
\end{equation}
 
\item[\textup{(ii)}] If $n=1,2$, then for every $g \in L^{\frac{2(1+\gamma)}{\gamma}}(B_1)$,
\[
\frac{1}{\beta(B_1)}\int_{B} g(x)^2 \beta(x) dx    \leq N \left(\fint_{B_{1}} |g(x)|^{\frac{2(1+\gamma)}{\gamma}} dx \right)^{\frac{\gamma}{1+\gamma}}.
\]

\end{itemize}
In particular, the embedding $W^{1,2}(B_1) \hookrightarrow L^2(B_1, \beta)$ is compact.
\end{lemma}
\begin{proof} We begin with the proof of (i). By Lemma \ref{R-Holder}, there exist constants $\gamma = \gamma(n, M_0) \in (0, \frac{2}{n})$ and $N_0 = N_0(n, M_0)\geq 1$ such that
\begin{equation} \label{mu-q-gamma}
\beta^{-1} \in A_{1+\frac{2}{n}-\gamma}\quad \text{and}\quad
[\beta^{-1}]_{A_{1+\frac{2}{n}-\gamma}} \leq N_0.
\end{equation}
Then, by H\"{o}lder's inequality, the definition of Muckenhoupt weights, and \eqref{mu-q-gamma}, we obtain
\begin{align*}
\fint_{B_{1}} |g(x)|^2 \beta(x)dx
&\leq \left(\fint_{B_{1}} \beta(x)^{\frac{n}{2-\gamma n}} dx\right)^{\frac{2-\gamma n}{n}}\left(\fint_{B_{1}} |g(x)|^{\frac{2n}{n(1+\gamma)-2}} dx \right)^{\frac{n(1+\gamma)-2}{n}} \\
& \leq  [\beta^{-1}]_{A_{1+\frac{2}{n} -\gamma}} \left(\fint_{B_{1}} \beta(x)^{-1} dx\right)^{-1} \left(\fint_{B_{1}} |g(x)|^{\frac{2n}{n(1+\gamma)-2}} dx \right)^{\frac{n(1+\gamma)-2}{n}}  \\
& \leq N_0 \left(\fint_{B_{1}} \beta(x) dx\right)\left(\fint_{B_{1}} |g(x)|^{\frac{2n}{n(1+\gamma)-2}} dx \right)^{\frac{n(1+\gamma)-2}{n}},
\end{align*}
where in the last step we used an estimate as in \eqref{1-mu-ball} with $q=2$.  This yields \eqref{beta-1.inq}.  

\smallskip
Next, we prove (ii). From the assumption, we note that $\beta \in A_2$. Then, by Lemma \ref{R-Holder}, there exists a sufficiently small number $\gamma = \gamma(n, M_0)>0$ such that
\[
\left(\fint_{B_1} \beta(x)^{1+\gamma} dx \right)^{\frac{1}{1+\gamma}} \leq N\fint_{B_1} \beta(x) dx,\quad \text{where}\quad N = N(n, M_0)>1.
\]
From this, and by applying H\"{o}lder's inequality, we obtain
\begin{align*}
\fint_{B_{1}} |g(x)|^2 \beta(x)dx & \leq  \left(\fint_{B_{1}} |g(x)|^{\frac{2(1+\gamma)}{\gamma}} dx \right)^{\frac{\gamma}{1+\gamma}} \left(\fint_{B_{1}} \beta(x)^{1+\gamma} dx\right)^{\frac{1}{1+\gamma}} \\
& \leq N \left(\fint_{B_{1}} |g(x)|^{\frac{2(1+\gamma)}{\gamma}} dx \right)^{\frac{\gamma}{1+\gamma}} \left(\fint_{B_{1}} \beta(x) dx\right).
\end{align*}
Hence, this proves (ii).

\smallskip
Finally, observe that the last assertion in the lemma follows from the first assertion and the fact that the embedding $W^{1,2}(B_1) \hookrightarrow L^{p}(B_1)$ is compact for $p < 2^*$. The proof of the lemma is completed.
\end{proof}
\smallskip
The following lemma on interpolation inequality will be used in the local energy estimates for the equations considered in the paper.
\begin{lemma} \label{embedd-lemma} For $M_0 \geq 1$, there exist $\theta =\theta(n, M_0) \in (0,1)$ sufficiently close to $1$ and $N = N(n, M_0)>0$ such that the following assertions hold. Suppose that $\beta(x)$ is a weight satisfying \eqref{beta-cond}, and that $u\in L^2(\Gamma, W^{1,2}(B_r(x_0)))$ with some open interval $\Gamma\subset \R$. Then
\[
\begin{split}
& \frac{1}{(\beta)_{B_r(x_0)}}\fint_{Q} |u(x,t)|^2 \beta(x)dx dt  \\
& \leq N  \Big(\fint_{Q} |u|^2 dxdt \Big)^{1-\theta} \left[ \Big(\fint_{Q} |u|^{2} dxdt \Big)^{\theta} + r^{2\theta} \Big(\fint_{Q} |\nabla u|^{2} dxdt \Big)^{\theta} \right],
\end{split}
\]
where $Q= B_r(x_0) \times \Gamma$. The same assertion also holds when $Q$ is replaced by the upper-half cylinder $Q^+ = B_r^+(x_0) \times \Gamma$.
\end{lemma}
\begin{proof} We only provide the proof when the spatial domain is $B_r(x_0)$, as the proof for the upper-half cylinder case is similar. By using the dilation $(x,t) \mapsto (r (x-x_0), t)$, we can assume without loss of generality that $r=1$ and $x_0 = 0$.  

\smallskip
We consider the case $n \geq 3$. Due to \eqref{beta-cond} and Lemma \ref{wei-Les}-(i), there exist numbers $\gamma = \gamma(n, M_0) \in (0, \frac{2}{n})$ small and $N = N(n, M_0)\geq 1$ such that 
\begin{align*}
\fint_{B_{1}} |u(x,t)|^2 \beta(x)dx \leq N \left(\fint_{B_{1}} |u(x,t)|^{\frac{2n}{n(1+\gamma)-2}} dx \right)^{\frac{n(1+\gamma)-2}{n}}  \left(\fint_{B_{1}} \beta(x) dx\right),
\end{align*}
for a.e.~$t \in \Gamma$. On the other hand,  it follows from H\"{o}lder's inequality and the Sobolev inequality that
\begin{align*}
\|u(\cdot, t)\|_{L^{\frac{2n}{n(1+\gamma)-2}}(B_1)}^2 & \leq \|u(\cdot, t)\|_{L^2(B_1)}^{2(1-\theta)} \|u(\cdot, t)\|_{L^{2^*}(B_1)}^{2\theta} \\
& \leq N(n)\|u(\cdot, t)\|_{L^2(B_1)}^{2(1-\theta)}  \Big[ \| u(\cdot, t)\|_{L^{2}(B_1)}^{2\theta}+ \|\nabla u(\cdot, t)\|_{L^{2}(B_1)}^{2\theta}\Big],
\end{align*}
where 
\[
\theta = \frac{n}{2} - \frac{n(1+\gamma) -2}{2} = 1-\frac{n \gamma}{2} \in (0, 1) \quad \text{and} \quad 2^* = \frac{2n}{n-2}.
\]
As a result, we obtain
\begin{align*}
& \frac{1}{(\beta)_{B_1}} \fint_{B_{1}} |u(x,t)|^2 \beta(x)dx \\
& \leq N \Big(\fint_{B_{1}} |u(x,t)|^2 dx \Big)^{1-\theta} \left[ \Big(\fint_{B_{1}} | u(x,t)|^{2} dx \Big)^{\theta} + \Big(\fint_{B_{1}} |\nabla u(x,t)|^{2} dx \Big)^{\theta} \right]
\end{align*}
for a.e.~$t \in \Gamma$, where $N = N(n, M_0)>0$. Now, we integrate the last inequality with respect to $t$ on $\Gamma$, and then use H\"{o}lder's inequality and the triangle inequality in the time integration to obtain
\begin{align*}
& \frac{1}{(\beta)_{B_1}} \fint_{Q} |u(x,t)|^2 \beta(x)dx dt \\
& \leq N\Big(\fint_{Q} |u(x,t)|^2 dx \Big)^{1-\theta} \left[\Big(\fint_{Q} | u(x,t)|^{2} dxdt \Big)^{\theta}  + \Big(\fint_{Q} |\nabla u(x,t)|^{2} dxdt \Big)^{\theta} \right].
\end{align*}
This completes the proof of the lemma when $n \geq 3$.

\smallskip
Next, we consider the case $n=1, 2$. By  Lemma \ref{wei-Les}-(ii), there exists a sufficiently small number $\gamma = \gamma(n, M_0)>0$ such that
\[
\frac{1}{(\beta)_{B_1}}\fint_{B_{1}} |u(x,t)|^2 \beta(x)dx  \leq N \left(\fint_{B_{1}} |u(x,t)|^{\frac{2(1+\gamma)}{\gamma}} dx \right)^{\frac{\gamma}{1+\gamma}}.
\]
Then, the result follows similarly by using an interpolation inequality and the Sobolev inequality, as we did in the case $n\geq 3$. The proof of the lemma is completed.
\end{proof}

\subsection{Weighted parabolic cylinders and a quasi-distance function}  \label{cylinder-distance-sec}
Let $M_0 \geq 1$ be a fixed number and let $\beta: \mathbb{R}^n \rightarrow [0, \infty)$ be a weight satisfying \eqref{beta-cond}:
\begin{equation*} 
\beta^{-1} \in A_{1+\frac{2}{n_0}} \quad \text{and} \quad [\beta^{-1}]_{A_{1+\frac{2}{n_0}}} \leq M_0  \quad \text{for}\quad  n_0 =\max\{n,2\}.
\end{equation*}
Now, for $z_0 = (x_0, t_0) \in \R^{n}\times \R$ and for $r>0$, let us define the \emph{weighted parabolic cylinder} of radius $r$ centered at $z_0$ by
\begin{equation} \label{cylinder-def}
Q_{r, \beta}(z_0) =B_r(x_0)  \times \Gamma_{\beta, z_0}(r), \quad \text{where}   \quad \Gamma_{\beta, z_0} (r) =(t_0 -r^2 \Psi_{\beta,x_0}(r), t_0],
\end{equation}
in which
\begin{equation} \label{t-height}
\Psi_{\beta, x_0}(r)=  (\beta^{\frac{n_0}{2}})_{B_r(x_0)}^{\frac{2}{n_0}} = \Big( \fint_{B_{r}(x_0)} \beta(x)^{\frac{n_0}{2}} dx \Big)^{\frac{2}{n_0}}.
\end{equation}
By a simple computation, we have
\begin{equation}\label{heights}
r^2 \Psi_{\beta, x_0}(r) =\left\{
\begin{array}{cl}
   \frac{1}{2} r\, \beta(B_r(x_0)) & \text{if}\quad n = 1,\\[6pt]
 \sigma_n^{-\frac{2}{n}} \left[\beta^{\frac{n}{2}}(B_r(x_0))\right]^{\frac{2}{n}} & \text{if}\quad n\geq 2,
\end{array}\right.
\end{equation}
where $\sigma_n$ denotes the Lebesgue measure of the unit ball $B_1\subset \R^n$, 
\[
\beta(B_r(x_0)) = \int_{B_r(x_0)} \beta(x) dx, \quad \text{and}  \quad \beta^{\frac{n}{2}}(B_r(x_0)) = \int_{B_r(x_0)} \beta(x)^{\frac{n}{2}} dx.
\]
It then follows that the class of weighted cylinders satisfies the following inclusion property:
\[
Q_{r_1,\, \beta}(z_0) \subset Q_{r_2,\, \beta}(z_0), \quad \text{for\ } \, 0<r_1 \leq r_2, \  \forall\, z_0 \in \mathbb{R}^{n+1}.
\]
The Figure \ref{weighted cylinders} below illustrates our definitions on the class of weighted cylinders $Q_{r,\beta}(z_0)$ under special scenarios of the weight $\beta$.
\begin{figure}[h]
\centering
\definecolor{color2}{RGB}{229, 117, 117}
\resizebox{0.88\columnwidth}{!}{
\begin{tikzpicture}[scale=0.37] 
\fill[purple, fill opacity=0.4] 
(-10, 0) arc[start angle=0, end angle=180, x radius=2, y radius=0.7]  
-- (-14, 0) -- (-14, -6) -- (-10, -6) -- cycle ;
\fill[purple, fill opacity=0.4]      
(-10,-6) arc[start angle=0, end angle=-180, x radius=2, y radius=0.7];               
\fill[black] (-12, 0) circle (2.5pt);
\draw[thick] (-12,0) -- (-10.5,0.45); \node at (-10.2, 0.97){\Large {$r$}};
\draw[thick][<->] (-15, 0) -- (-15, -6);
\node at (-18.2, -3) {\large{{$r^2\Psi_{\beta, x_0}(r)$}}};
\node at (-12.7, 0){\Large{$z_0$}};
\node at (-12,-8.8){\Large{$ \beta(x)> 1$}};
\draw [thick] (-12, 0) ellipse (-2 and 0.7);
\draw [thick] (-14,0) -- (-14,-6);
\draw [thick] (-10,0) -- (-10,-6);
\draw [thick] (-14,-6) arc[start angle=180, end angle=360, x radius=2, y radius=0.7];
\draw [thick, dashed] (-10,-6) arc[start angle=0, end angle=180, x radius=2, y radius=0.7];
\fill[orange, fill opacity=0.2] 
(0,-2) arc[start angle=0, end angle=180, x radius=2, y radius=0.7]  
-- (-4, -2) -- (-4, -6) -- (0, -6) -- cycle ;
\fill[orange, fill opacity=0.2]      
(0,-6) arc[start angle=0, end angle=-180, x radius=2, y radius=0.7];             
\fill[black] (-2, -2) circle (2.5pt);
\draw[thick] (-2,-2) -- (-0.5,-1.55); \node at (-0.2,-1.0) {\Large{$r$}};
\draw[thick][<->] (0.7,-2) -- (0.7,-6);
\node at (1.7, -3.8) {\Large{$r^2$}};
\node at (-2.7,-2){\Large{$z_0$}};
\node at (-2,-8.8){\Large{$\beta(x)=1$}};
\draw [thick] (-2,-2) ellipse (2 and 0.7);
\draw [thick] (-4,-2) -- (-4,-6);
\draw [thick] (0,-2) -- (0,-6);
\draw [thick] (-4,-6) arc[start angle=180, end angle=360, x radius=2, y radius=0.7];
\draw [thick, dashed] (0,-6) arc[start angle=0, end angle=180, x radius=2, y radius=0.7];
\fill[purple, fill opacity=0.40] 
(10,-4) arc[start angle=0, end angle=180, x radius=2, y radius=0.7]  
-- (6, -4) -- (6, -6) -- (10, -6) -- cycle ;
\fill[purple, fill opacity=0.40]      
(10,-6) arc[start angle=0, end angle=-180, x radius=2, y radius=0.7]; 
\fill[black] (8, -4) circle (2.5pt);
\draw[thick] (8,-4) -- (9.5,-3.55); \node at (9.8,-3) {\Large{$r$}};
\draw[thick][<->] (10.7,-4) -- (10.7,-6);
\node at (14, -5) {\large{{$r^2\Psi_{\beta, x_0}(r)$}}};
\node at (7.3,-4){\Large{$z_0$}};
\node at (8,-8.8){\Large{$0<\beta(x)< 1$}};
\draw [thick] (8,-4) ellipse (2 and 0.7);
\draw [thick] (10,-4) -- (10,-6);
\draw [thick] (6,-4) -- (6,-6);
\draw [thick] (6,-6) arc[start angle=180, end angle=360, x radius=2, y radius=0.7];
\draw [thick, dashed] (10,-6) arc[start angle=0, end angle=180, x radius=2, y radius=0.7];
\end{tikzpicture}}
\caption{Weighted cylinders $Q_{r,\, \beta}(z_0)$ centered at $z_0=(x_0, t_0)$.}\label{weighted cylinders}
\end{figure}
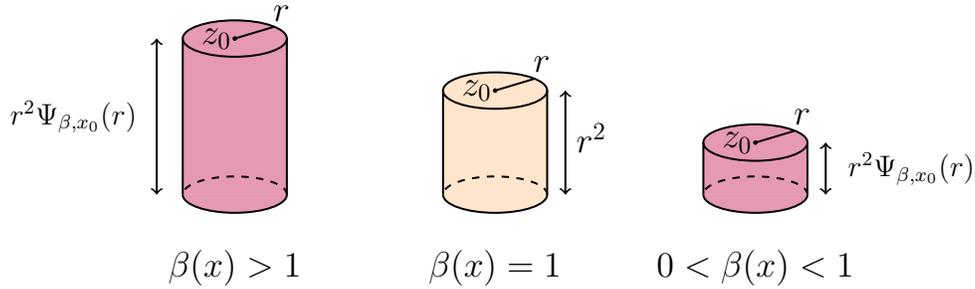

In addition, we denote the upper half ball in $\mathbb{R}^n$ by
\[
B_r^+(x_0) = \left\{x = (x', x_n) \in \mathbb{R}^{n-1} \times (0, \infty): x \in B_r(x_0)\right\},
\]
and 
its flat boundary portion by
\begin{equation} \label{flat-bdr}
T_r(x_0) =  B_r(x_0) \cap \big(\mathbb{R}^{n-1} \times \{0\}\big).
\end{equation}
In addition, the upper half parabolic weighted cylinder of radius $r$ centered at $z_0=(x_0, t_0)$ is denoted by
\begin{equation}\label{half cylinder}
Q_{r, \beta}^+(z_0) = B_r^+(x_0) \times \Gamma_{\beta, z_0}(r).
\end{equation}
Throughout the paper, when $z_0 =(0, 0)$ and $x_0 = 0$, we omit $z_0$ and $x_0$ in the notation. For example, we write
\[
 Q_{r, \beta} = Q_{r, \beta}(0), \quad B_r = B_r(0),\quad T_r = T_r(0) \quad \Psi_\beta(r) = \Psi_{\beta, 0}(r), \quad \Gamma_{\beta}(r) = \Gamma_{\beta, 0} (r).
\]
The same definitions can be defined for $Q_{r, \beta}^+$ and $B_r^+$.

\smallskip
Next, we introduce a quasi-distance that endows the class of parabolic weighted cylinders. To begin with, given $z_0=(x_0,t_0)\in\R^{n}\times \R$, we let $h_{x_0}: [0,\infty)\rightarrow [0,\infty)$ be defined by
\begin{equation*}
h_{x_0}(r)=\left\{
 \begin{array}{cl}
 0 \quad &\text{if}\quad r=0,\\[4pt]
r^2\Psi_{\beta,x_0}(r) \quad &\text{if}\quad r\in(0,\infty),\\
 \end{array}\right.
\end{equation*}
where $\Psi_{\beta, x_0}$ is given in \eqref{t-height}.  From  \eqref{heights}, it follows that $h_{x_0}(r)$ is strictly increasing in $r$. Hence, its inverse function, $h^{-1}_{x_0}$, exists.

\smallskip
Now, for any $z=(x,t)\in \R^{n} \times \R$ with $t\leq t_0$, we denote
\begin{equation*}
\begin{aligned}
\tilde \rho_{\beta}(z,z_0)
=\inf\big\{r>0: z\in Q_{r,\beta}(z_0)\big\}=\max\big\{|x-x_0|,\ h_{x_0}^{-1}(t_0-t)\big\}.
\end{aligned}
\end{equation*}
Then, for  any $z=(x,t)\in \R^{n}\times \R$, we define
\begin{equation}\label{rho def}
\rho_{\beta}(z, z_0)=
\left\{
\begin{array}{rl}
\tilde \rho_{\beta}(z, z_0) \quad &\text{if}\quad t\leq t_0,\\[4pt]
\tilde \rho_{\beta}(z_0, z)\quad &\text{if}\quad t>t_0.
\end{array}\right.
\end{equation}
Intuitively, $\rho=\rho_{\beta}(z, z_0)$ represents the spatial radius for which $z$ lies on the parabolic boundary $\partial_pQ_{\rho,\beta}(z_0)$ if $t \leq t_0$ or $z_0$ lies on the parabolic boundary $\partial_pQ_{\rho,\beta}(z)$ when $t >t_0$. 
From the definition in \eqref{rho def}, it is straightforward to see that
\begin{equation}\label{d-r}
|x-x_0|\leq\rho_{\beta}(z, z_0),\quad \text{for all}\quad z=(x,t),\ z_0=(x_0,t_0)\in \R^{n} \times \R.
\end{equation}
Additionally, for any $z_0 = (x_0, t_0) \in \R^{n}\times \R$ and $r>0$, we also have
\begin{equation}\label{cylinder-quasi}
Q_{r,\beta}(z_0) = \{z = (x,t) \in \R^{n}\times \R: \quad \rho_{\beta} (z_0, z) <r \ \text{ and }\  t  \leq t_0\}.
\end{equation}
\smallskip
It turns out that $\rho_\beta$ defines a quasi-metric on $\R^{n}\times \R$, as stated in the following lemma, whose proof is given in Appendix \ref{Appendix A}.

\begin{lemma}\label{quasi-metric lemma} Suppose that \eqref{beta-cond} holds. Then, the function $\rho_\beta$ defined in \eqref{rho def}  is a quasi-metric. In particular, there exists a constant $\Lambda= \Lambda (n, M_0) \geq 1$ such that
\begin{equation}\label{quasi-tri}
\rho_{\beta}(z_0,\bar{z})\leq \Lambda \big[\rho_{\beta}(z_0,z_1)+\rho_{\beta}(z_1,\bar{z})\big], \quad \forall z_0,\ z_1,\ \bar{z} \ \in \ \R^{n} \times \R.
\end{equation}
Moreover, $\Lambda$ can be chosen as
\[
\Lambda(n,M_0)=\max\{2^{\frac{1}{2\zeta_0}}N_2^{\frac{1}{n\zeta_0}},\, 2\}\geq 2,
\]
where $N_2(n,M_0)>1$ and $\zeta_0(n,M_0)\in (0,1)$ are the constants in Lemma \ref{property}.
\end{lemma}

\section{Function spaces, a compactness lemma, and weak solutions} \label{Fn-spaces}
\subsection{Function spaces and a compactness lemma} We begin with the following definition of function spaces that are needed for studying \eqref{main-eqn}.

\begin{definition} \label{W-space} Let $p \in (1, \infty)$, $\Omega \subset \mathbb{R}^n$ be a domain, and $S < T$ be real numbers. Also, let $\beta: \Omega\rightarrow [0, \infty)$ be a locally integrable function. A function $u \in L^p((S, T), W^{1,p}(\Omega))$ is said to be in $\W^{1,p}(\Omega \times (S,T), \beta)$ if 
\[   \beta u_t \in L^{p}((S, T), W^{-1, p}(\Omega)), 
\]  
where $W^{-1,p}(\Omega)$ denotes the dual space of $W^{1, p'}_0(\Omega)$ with $\frac{1}{p} + \frac{1}{p'} =1$. We write
\begin{align*}
\|u\|_{\W^{1,p}(\Omega \times (S,T), \beta)} & = \|u\|_{L^p((S,T), W^{1,p}(\Omega))}  +  \|\beta u_t\|_{ L^{p}((S, T), W^{-1, p}(\Omega))}.
\end{align*}
\end{definition} 

\smallskip
\noindent
The following definition will be needed in the study of local boundary regularity for equation \eqref{main-eqn}.
\begin{definition}  \label{hat-W-space} Let $B^+ = B_r^+(x_0)$ be an upper ball with some $r>0$ and $x_0 \in \mathbb{R}^n$. Let $S < T$ be real numbers, $p \in (1, \infty)$, and let $\beta: B^+ \rightarrow [0, \infty)$ be a locally integrable function. A function $u \in L^p((S, T), W^{1,p}(B^+))$ is said to be in $\hW^{1,p}(B^+ \times (S,T), \beta)$ if $u \vert_{T_r(x_0) \times (S, T)} =0$ in the sense of trace, and  
\[  
\beta u_t \in L^{p}((S, T), W^{-1, p}(B^+)), 
\]  
where $W^{-1,p}(B^+)$ denotes the dual space of $W^{1, p'}_0(B^+)$  with $\frac{1}{p} + \frac{1}{p'} =1$. We write
\begin{align*}
\|u\|_{\W^{1,p}(B^+ \times (S,T), \beta)} & = \|u\|_{L^p((S,T), W^{1,p}(B^+))}  +  \|\beta u_t\|_{ L^{p}((S, T), W^{-1, p}(B^+))}.
\end{align*}
\end{definition}
\noindent
Note that from the definitions, if $\beta(x)=1$, then
\begin{align*}
& \W^{1,p}(\Omega \times (S,T))=\W^{1,p}(\Omega \times (S,T), 1), \quad \text{and}\\
& \hW^{1,p}(B^+ \times (S,T)) =\hW^{1,p}(B^+ \times (S,T), 1),
\end{align*}
where $\W^{1,p}(\Omega \times (S,T))$ and $\hW^{1,p}(B^+ \times (S,T))$ are the standard un-weighted parabolic Sobolev function spaces.

\smallskip
\begin{lemma} \label{conti-weak-space}
Let $M_0 \geq 1$, and let $\beta$ be a weight satisfying \eqref{beta-cond}. Then, for every $u \in \W^{1,2}(\Omega \times (S,T), \beta)$ and for every $\varphi \in C_0^\infty(\Omega)$, we have $u \varphi \in L^\infty((S, T), L^2(\Omega, \beta))$, and $u \varphi$ has an equivalent function in $C([S,T], L^2(\Omega, \beta))$. Moreover, for each $h>0$ sufficiently small, we have
\begin{equation} \label{h-S-integrate-0402}
\begin{split}
& \int_{S}^{T-h}\| [u(\cdot, t+h) - u(\cdot, t)] \varphi\|_{L^2(\Omega, \beta)}^2 dt \\
& \leq  2 h^{1/2}\|\beta u_t\|_{L^2((S,T), W^{-1, 2}(\Omega))}  \|u\varphi^2\|_{L^2((S, T), W^{1,2}(\Omega))}^2.
\end{split}
\end{equation}
The same assertions also hold when $\W^{1,2}(\Omega \times (S,T), \beta)$ is replaced by $\hat{\W}^{1,2}(B_r^+(x_0) \times (S,T), \beta)$ with $\varphi \in C_0^\infty(B_r(x_0))$ for any $r>0$ and $x_0 \in \mathbb{R}^n$.
\end{lemma}

\begin{proof} The proof is a modification of the classical one when $\beta\equiv1$. Let $w (x,t)= \beta(x) u(x,t)$ with $(x, t) \in \Omega \times (S, T)$. Due to Lemma \ref{embedd-lemma}, 
\begin{equation} \label{tau-a.e.0401}
\| u(\cdot, t)\|_{L^2(\Omega, \beta)} <\infty \quad \text{for a.e.} \ t \in (S,T).
\end{equation}
We then fix some $s \in (S, T)$ so that 
\begin{equation}  \label{a.e.tau.fit}
 \| u(\cdot, s)\|_{L^2(\Omega, \beta)} <\infty  \quad \text{and} \quad \| u(\cdot, s)\|_{W^{1,2}(\Omega)} < \infty.
\end{equation}
Now, from Lemma \ref{embedd-lemma} and by taking a suitable modification in the time variable if needed, we have
\[
\frac{d}{dt} \| \big(u(\cdot, t) - u(\cdot, s)\big)\varphi\|_{L^2(\Omega, \beta)}^2 = 2 \int_{\Omega} w_t(x,t) \big(u(x,t) - u(x, s)\big) \varphi(x)^2 dx
\]
for any $\varphi \in C_0^\infty(\Omega)$. Therefore, for a.e.~$t \in (S, T)$, by Minkowski's inequality and by applying H\"{o}lder's inequality in the time integration, we obtain
\begin{equation}\label{u-s-t}
\begin{aligned}
& \|\big(u(\cdot, t) -  u(\cdot, s)\big) \varphi\|_{L^2(\Omega, \beta)}^2  =2 \int_s^t \int_{\Omega} w_t(x,\tau) \big(u(x,\tau) - u(x, s)\big) \varphi(x)^2 dx d\tau\\
 & \leq 2 \int_s^t  \|w_t(x,\tau)\|_{W^{-1, 2}(\Omega)} \|\big(u(\cdot,\tau) -  u(\cdot, s)\big)\varphi^2\|_{W^{1,2}(\Omega)} d\tau\\
 & \leq 2\|w_t\|_{L^2((s,t), W^{-1, 2}(\Omega))} \Big[ \|u \varphi^2\|_{L^2((s,t), W^{1,2}(\Omega))} + |t-s|^{\frac{1}{2}} \|u(\cdot, s)\varphi^2\|_{W^{1,2}(\Omega)}\Big].
\end{aligned}
\end{equation}
Thus,
\begin{align*}
& \| u(\cdot, t)\varphi \|_{L^2(\Omega, \beta)}^2  \leq 2 \| u(\cdot, s)\varphi\|_{L^2(\Omega, \beta)}^2 \\
& \quad + 4\|w_t\|_{L^2((S,T), W^{-1, 2}(\Omega))} \Big[ \|u \varphi^2\|_{L^2((S,T), W^{1,2}(\Omega))} + |T-S|^{\frac{1}{2}} \|u(\cdot, s)\varphi^2\|_{W^{1,2}(\Omega)}\Big]
\end{align*}
for a.e.~$t \in (S, T)$. From this and \eqref{a.e.tau.fit}, we have $u\varphi \in L^\infty((S, T), L^2(\Omega, \beta))$. On the other hand, from \eqref{u-s-t} and \eqref{tau-a.e.0401},
\[
\lim_{\tau \rightarrow t} \|\big(u(\cdot, \tau) - u(\cdot, t)\big )\varphi \|_{L^2(\Omega, \beta)} =0, \quad \text{for a.e. }t \in (S,T).
\]
Hence, $u \varphi$ has an equivalent function in $C([S,T], L^2(\Omega, \beta))$.  

\smallskip
Now, it remains to prove \eqref{h-S-integrate-0402}. As in the first line of \eqref{u-s-t}, we also have
\[
\| \big(u(\cdot, t +h) - u(\cdot, t)\big)\varphi\|_{L^2(\Omega, \beta)}^2 =  \int_{\Omega} \left(\int_{t}^{t+h}w_t(x,\tau) \big(u(x,t+h) - u(x, t)\big) \varphi^2d\tau \right)dx
\]
for a.e.~$t \in (S, T-h)$. It follows from Fubini's theorem and H\"{o}lder's inequality that
\begin{align*} 
& \|\big(u(\cdot, t+h) -  u(\cdot, t)\big)\varphi\|_{L^2(\Omega, \beta)}^2 \\
& \leq  \int_t^{t+h}  \|w_t(x,\tau)\|_{W^{-1, 2}(\Omega)} \|\big(u(\cdot,t+h) -  u(\cdot, t)\big)\varphi^2\|_{W^{1,2}(\Omega)} d\tau \\ 
 & \leq h^{1/2}\|w_t\|_{L^2((t,t+h), W^{-1, 2}(\Omega))} \|\big(u(\cdot,t+h) -  u(\cdot, t)\big)\varphi^2\|_{W^{1,2}(\Omega)}\\ 
&\leq h^{1/2}\|w_t\|_{L^2((S,T), W^{-1, 2}(\Omega))} \Bigl[\|u(\cdot,t+h)\varphi^2\|_{W^{1,2}(\Omega)} + \|u(\cdot, t)\varphi^2\|_{W^{1,2}(\Omega)}\Bigr].
\end{align*}
Integrating in $t$ over $(S,T-h)$ yields
\begin{align*}
&\int_{S}^{T-h} \| \big(u(\cdot, t+h) -  u(\cdot, t)\big)\varphi\|_{L^2(\Omega, \beta)}^2 dt \\
& \leq 2 h^{1/2}\|w_t\|_{L^2((S,T), W^{-1, 2}(\Omega))}  \|u\varphi^2\|_{L^2((S, T), W^{1,2}(\Omega))}^2.
\end{align*}
Hence, \eqref{h-S-integrate-0402} is proved. Note also that the assertion on $\hat{\W}^{1,2}(B_r^+(x_0) \times (S,T), \beta)$ can be proved similarly. The proof of the lemma is completed.
\end{proof}

The following Aubin-Lions type compactness result is essential in the paper.  Here, we denote $2^* = \frac{2n}{n-2}$ for $n \geq 3$ and $2^* = \infty$ when $n=1,2$.

\begin{proposition}[Aubin-Lions compactness theorem] \label{compactness-lemma} For every $M_0 \geq 1$, there exists a constant $l_0 = l_0(n, M_0)\in (1,\infty)$ sufficiently close to $1$ such that the following assertions hold. Suppose that $\{\beta_k\}_k$ is a sequence of weights satisfying 
\begin{equation} \label{compact-beta-ass}
\beta^{-1}_k \in A_{1+\frac{2}{n_0}}, \quad [\beta^{-1}_k]_{A_{1+\frac{2}{n_0}}} \leq M_0, \quad (\beta_k^{\frac{n_0}{2}})_{B_r}^{\frac{2}{n_0}}=1, \quad \forall \ k \in \mathbb{N},
\end{equation}
where $n_0 =\max\{n,2\}$, and $r>0$ is some given number. Suppose also that 
\begin{equation}\label{osc-assumption}
\lim_{k \rightarrow \infty} \frac{1}{\beta_k(B_r)}\int_{B_{r}} |\beta_k(x) - (\beta_k)_{B_r}|^2 \beta_k^{-1}(x)\, dx =0,
\end{equation}
and suppose that  $\{u_k\}_{k}$ is a sequence of functions satisfying $ u_k \in \W^{1,2}(Q_r, \beta_k)$ for every $k \in \mathbb{N}$, and there is some constant $M>0$ such that
\begin{equation} \label{u-k-bound-0402}
\|u_k\|_{\W^{1,2}(Q_r, \beta_k)} \leq M, \quad \forall \ k \in \mathbb{N}.
\end{equation}
Then there exist a constant $\beta_0\in [\frac{1}{M_0},1]$, a function $u \in \W^{1,2}(Q_r)$, and subsequences of $\{\beta_k\}_k$ and $\{u_k\}_k$, which we still denote respectively by $\{\beta_k\}_k$ and $\{u_k\}_k$, such that  
\begin{equation} \label{l-0-convergence}
\left\{
\begin{array}{ll}
\beta_k \partial_t u_k \rightharpoonup \beta_0u_t &\quad \text{in}\quad L^2((-r^2, 0), W^{-1,2}(B_r)),\\[4pt]
u_k \rightarrow u &\quad \text{in} \quad L^2((-r^2, 0), L^{l_0}(B_r)),\ \text{and} \\[4pt]
\nabla u_k  \rightharpoonup \nabla u  &\quad \text{in} \quad  L^2(Q_r), \qquad \text{as}\quad k \rightarrow \infty.
\end{array} \right.
\end{equation}
Moreover, for each $l \in [1, 2^*)$, we also have
\begin{equation} \label{L-q-convergence}
u_k \rightarrow u \quad \text{in} \quad L^2((-r^2, 0), L^{l}(B_r)), \quad \text{as} \quad k \rightarrow \infty.
\end{equation}
The same assertions also hold when $B_r$, $Q_r$, and $\W^{1,2}(Q_r, \beta_k)$ are replaced respectively by $B_r^+(x_0)$, $Q_r^+(z_0)$, and $\hW^{1,2}(Q_r^+(z_0), \beta_k)$ with $z_0 = (x_0, t_0) \in \mathbb{R}^n \times \mathbb{R}$.
\end{proposition}
\begin{proof} Using the dilation $x\mapsto r x$, without loss of generality, we assume that $r=1$. We present the proof in the case where the spatial domain is $B_1$, as the argument for the upper ball $B_1^+(x_0)$ is similar. To begin with, due to \eqref{compact-beta-ass}, it follows from Remark~\ref{A-p-remark} and H\"{o}lder's inequality that
\[
\frac{1}{M_0}\leq [(\beta_k^{-1})_{B_1}]^{-1}\leq(\beta_k)_{B_1}\leq (\beta_k^{\frac{n}{2}})_{B_1}^{\frac{2}{n}}=1 \quad \forall \ k\in\N,\quad {for}\quad n\geq 3.
\]
From this and as $(\beta_k)_{B_1}=1$ for all $k\in \N$ when $n=1,\, 2$, it follows that
\begin{equation}\label{bdd beta}
(\beta_k)_{B_1}\in \big[\tfrac{1}{M_0}, 1\big], \quad \forall \ k\in \N \quad  \text{and} \quad \forall \ n \in \mathbb{N}.
\end{equation}
Hence, there exist a number $\beta_0\in [\frac{1}{M_0},1]$ and a subsequence of $\{\beta_k\}_k$, still denoted by $\{\beta_k\}_k$, such that
\begin{equation}\label{beta-lim}
\lim_{k\rightarrow \infty}(\beta_k)_{B_1}=\beta_0.
\end{equation}

\smallskip
Now, let us denote 
\[ 
w_k (x,t)= \beta_k (x)u_k(x,t), \quad (x,t) \in Q_1.
\]
By the definition of the space $\W^{1,2}(Q_1, \beta_k)$ and the assumption \eqref{u-k-bound-0402}, we see that the sequence $\{u_k\}_k$ is bounded in $L^2((-1,0), W^{1,2}(B_1))$, and the sequence $\{\partial_t w_k\}_k$ is bounded in $L^2((-1,0), W^{-1,2}(B_1))$. Since these are Hilbert spaces, and by passing through subsequences, we can find 
\[ 
u \in   L^2((-1,0), W^{1,2}(B_1)) \quad \text{and} \quad  \quad w \in  L^2((-1,0), W^{-1,2}(B_1)),
\] 
such that
\begin{align} \label{u-k-weak}
u_k  \rightharpoonup  u &\quad \text{in} \quad L^2((-1,0), W^{1,2}(B_1)) \quad \text{as} \quad k\rightarrow \infty, \quad \text{and} \\ \label{mu-u-k-weak}
\partial_t w_k  \rightharpoonup  w &\quad \text{in} \quad L^2((-1,0), W^{-1,2}(B_1)) \quad \text{as} \quad k\rightarrow \infty.
\end{align}

\smallskip
We observe that the third assertion in \eqref{l-0-convergence} follows directly from \eqref{u-k-weak}.  Now, we prove the first assertion in \eqref{l-0-convergence}. Due to \eqref{mu-u-k-weak}, it suffices to show that
\begin{equation}\label{lim claim}
   w=\beta_0\partial_t u \quad \text{in}\quad L^2((-1,0), W^{-1,2}(B_1)).
\end{equation}   
To this end, let $\varphi \in C^\infty_0(Q_1)$. We note from Lemma \ref{embedd-lemma} that $w_k \in L^1(Q_1)$. Then, in the sense of distributions, 
\begin{align} \notag
& \int_{Q_1}  \partial_t w_k \varphi(x, t) dxdt   = - \int_{Q_1} w_k \partial_t \varphi(x, t) dx dt \\ \label{eqn-03-12-1}
& = - \int_{Q_1} \Big[ \beta_k (x) -(\beta_k)_{B_1}\Big] u_k \partial_t \varphi(x, t) dx dt   - \int_{Q_1} (\beta_k)_{B_1}u_k \partial_t\varphi(x, t) dx dt.
\end{align}
From \eqref{beta-lim} and \eqref{u-k-weak}, we obtain
\begin{equation} \label{eqn-03-12-2}
\lim_{k\rightarrow \infty} \int_{Q_1}  (\beta_k)_{B_1}u_k \partial_t\varphi(x, t) dx dt= \int_{Q_1}  \beta_0u\partial_t\varphi(x, t) dx dt.
\end{equation}
On the other hand, by using H\"{o}lder's inequality, Lemma \ref{embedd-lemma} and \eqref{bdd beta}, we have
\begin{align*}
& \left|  \int_{Q_1} \Big[ \beta_k (x) -(\beta_k)_{B_1}\Big] u_k \partial_t \varphi dx dt \right| \\
& \leq \left(\frac{1}{\beta_k(B_1)}\int_{B_1} |\beta_k (x) - (\beta_k)_{B_1}|^2 \beta_k(x)^{-1} dx \right)^{1/2}\left(\beta_k(B_1) \int_{Q_1}[u_k \partial_t \varphi]^2 \beta_k(x) dxdt\right)^{1/2} \\
& \leq N\left(\frac{1}{\beta_k(B_1)}\int_{B_1} |\beta_k (x) - (\beta_k)_{B_1}|^2 \beta_k(x)^{-1} dx \right)^{1/2} \|u_k\|_{L^2((-1, 0), W^{1,2}(B_1))},
\end{align*}
where $N=N(n, M_0, \|\partial_t \varphi\|_{L^{\infty}(Q_1)})$. From this, \eqref{osc-assumption}, and the boundedness of $\{u_k\}_k$ in $L^2((-1, 0), W^{1,2}(B_1))$ due to \eqref{u-k-bound-0402}, we deduce that
\begin{equation} \label{eqn-03-12-3}
\lim_{k\rightarrow \infty} \int_{Q_1} \Big[ \beta_k (x) -(\beta_k)_{B_1} \Big] u_k \partial_t \varphi(x, t) dx dt  =0.
\end{equation}
Finally, combining \eqref{eqn-03-12-1}, \eqref{eqn-03-12-2}, and \eqref{eqn-03-12-3}, we conclude that
\[ 
\int_{Q_1} w \varphi(x, t) dxdt =- \int_{Q_1}  \beta_0u \partial_t \varphi(x, t) dx dt \quad \text{for all} \quad \varphi \in  C^\infty_0(Q_1).
\]
Hence, $w=\beta_0\partial_t u$ in the sense of distributions. This implies that $u \in \W^{1,2}(Q_1,\beta_0)$, and thus \eqref{lim claim} is proved.

\smallskip
Next, let $l_0 = \frac{2}{2-\gamma} \in (1,2)$, where $\gamma = \gamma(n, M_0)\in (0,1)$ is the sufficiently small constant defined in Lemma \ref{L-q-2-wei} with $q=2$. Our aim is to prove the second convergence in \eqref{l-0-convergence}, that is, there exists a subsequence of $\{u_k\}_k$,  still denoted by $\{u_k\}_k$, such that
\begin{equation} \label{u-k-1-converge}
u_k \rightarrow u \in L^2((-1,0), L^{l_0}(B_1)) \quad \text{as} \quad k \rightarrow \infty.
\end{equation}
To prove \eqref{u-k-1-converge}, we claim that for any $\rho \in (0,1)$ sufficiently close to $1$, there exists a subsequence $\{u_{\rho, k}\}_k$ of $\{u_k\}_k$ such that
\begin{equation} \label{u-k-rho-converge}
u_{\rho, k} \rightarrow u \in L^{2}((-1, 0) \times L^{l_0}(B_\rho)) \quad \text{as} \quad k \rightarrow \infty.
\end{equation}
We observe that once the claim \eqref{u-k-rho-converge} is proved, the assertion \eqref{u-k-1-converge} follows by applying the Cantor diagonal process. It therefore remains to establish \eqref{u-k-rho-converge}.
 
\smallskip
To this end, we note that $\beta_k \in A_2$ with $[\beta_k]_{A_2} \leq M_0$ due to \eqref{compact-beta-ass} and Remark \ref{A-p-remark}-(ii). Then, as $(\beta_k)_{B_1} \in [M_0^{-1}, 1]$ for all $k\in \N$ from \eqref{bdd beta}, we can apply the doubling properties of the $A_2$-weight $\beta$ in Lemma \ref{property}  to show that 
\[ (\beta_k)_{B_\rho}  \in [\frac{1}{2M_0}, 2] \quad \text{for all } k \in \N \ \text{ and all }\ \rho \in (0,1) \ \text{ sufficiently close to }1.
\] 
Then, it follows from Lemma \ref{L-q-2-wei} and our choice of $l_0$ that there is $N = N(n, M_0, \rho)>0$ so that
\begin{equation} \label{eqn-03-12-4-b}
\|g\|_{L^{l_0}(B_\rho)} \leq N \|g\|_{L^2(B_\rho, \beta_k)}, \quad \forall \  g \in L^2(B_\rho, \beta_k).
\end{equation}
Now, our goal is to apply the Aubin-Lions theorem (see \cite[Theorem 1]{Simon} for example) to achieve \eqref{u-k-rho-converge}. To this end, we need to verify the following two conditions:
\begin{itemize}
\item[\textup{(i)}] For each $S < t_1 < t_2 < T$, the set
\[
\Big\{\int_{t_1}^{t_2} u_k(x, t) dt,\ k \in \mathbb{N} \Big\} \quad \text{is relatively compact in } L^{l_0}(B_\rho).
\]
\item[\textup{(ii)}] 
\[
\lim_{h\rightarrow 0^+}\int_{-1}^{-h}\|u_k(\cdot, t+h) - u_k(\cdot, t)\|_{L^{l_0}(B_\rho)}^2 dt = 0, \quad \text{uniformly in } k.
\]
\end{itemize}
To see $\textup{(i)}$, we note that
\begin{align*}
\Big\|\int_{t_1}^{t_2} u_k(x, t) dt \Big \|_{W^{1,2}(B_\rho)} & \leq \int_{t_1}^{t_2} \| u_k(\cdot, t)\|_{W^{1,2}(B_\rho)}dt \\
&  \leq (t_2-t_1)^{1/2} \|u_k\|_{L^{2}((S, T), W^{1,2}(B_1))} \leq N, \quad \forall \ k \in \mathbb{N}.
\end{align*}
Hence, (i) follows from the compact embedding $W^{1,2}(B_\rho) \hookrightarrow L^{l_0}(B_\rho)$. On the other hand, to check (ii), we apply \eqref{eqn-03-12-4-b} to obtain
\[
\|u_k(\cdot, t+h) - u_k(\cdot, t)\|_{L^{l_0}(B_\rho)} \leq N \|u_k(\cdot, t+h) - u_k(\cdot, t)\|_{L^2(B_\rho, \beta_k)}, \quad \forall\  k \in \mathbb{N},
\]
for a.e.~$t \in (-1, -h)$, where $N = N(n, M_0, \rho)>0$. Then, from this, \eqref{u-k-bound-0402},  and \eqref{h-S-integrate-0402} of Lemma \ref{conti-weak-space} with a suitable cut-off function $\varphi$, it follows that
\[
\int_{-1}^{-h}\|u_k(\cdot, t+h) - u_k(\cdot, t)\|_{L^{l_0}(B_\rho)}^2 dt \leq N h^{1/2}, \quad \forall \ k \in \mathbb{N},
\]
where $N = N(n,M_0,M,\rho) > 0$.  This proves \textup{(ii)}.

\smallskip
Next, we apply the Aubin-Lions theorem (see \cite[Theorem 1]{Simon}) to conclude that the sequence $\{u_k\}_k$ is relatively compact in $L^2((-1, 0), L^{l_0}(B_\rho))$. Due to \eqref{u-k-weak}, we infer that there exists a subsequence $\{u_{\rho, k}\}_k$ of the sequence $\{u_k\}_k$ such that
\[
u_{\rho, k} \rightarrow u \quad \text{in} \quad L^2((-1, 0), L^{l_0}(B_\rho)) \quad \text{as} \quad k \rightarrow \infty.
\]
This proves \eqref{u-k-rho-converge}, and therefore \eqref{u-k-1-converge} follows.


\smallskip
It now remains to prove the assertion \eqref{L-q-convergence}. It is sufficient to consider the case $n\geq 3$ as the proof can be done similarly when $n=1,\, 2$. Due to the second convergence in \eqref{l-0-convergence}, we only need to consider the case $l \in (l_0, 2^*)$. By an interpolation inequality and the Sobolev inequality, there exists $N = N(n)>0$ such that 
\begin{align*}
\|u_k(\cdot, t) - u(\cdot, t)\|_{L^{l}(B_1)} & \leq \|u_k(\cdot, t) - u(\cdot, t)\|_{L^{l_0}(B_1)}^{1-\theta} \|u_k(\cdot, t) - u(\cdot, t)\|_{L^{2^*}(B_1)}^{\theta} \\
& \leq N\|u_k(\cdot, t) - u(\cdot, t)\|_{L^{l_0}(B_1)}^{1-\theta} \|u_k(\cdot, t) - u(\cdot, t)\|_{W^{1,2}(B_1)}^{\theta}
\end{align*}
for a.e.~$t \in (-1,0)$. Here, $\theta \in (0,1)$ is the number such that
\[
\frac{1}{l} = \frac{\theta}{2^*} + \frac{1-\theta}{l_0} \quad \text{with} \quad 2^* = \frac{2n}{n-2}.
\]
From this, we can apply H\"{o}lder's inequality in the time variable to conclude that
\begin{align*}
& \|u_k(\cdot, t) - u(\cdot, t)\|_{L^2((-1, 0), L^{l}(B_1))} \\
& \leq N(n)\|u_k(\cdot, t) - u(\cdot, t)\|_{L^2((-1, 0), L^{l_0}(B_1))}^{1-\theta} \|u_k(\cdot, t) - u(\cdot, t)\|_{L^2((-1, 0), W^{1,2}(B_1))}^{\theta}\\
& \leq N(n,M)\|u_k(\cdot, t) - u(\cdot, t)\|_{L^2((-1, 0), L^{l_0}(B_1))}^{1-\theta}, \quad \forall \ k \in \mathbb{N}.
\end{align*}
Hence, \eqref{L-q-convergence} follows from \eqref{l-0-convergence}, and the proof of the proposition is completed.
\end{proof}
\subsection{Weak solutions and special dilation properties} \label{scaling-sec} Let $\beta$ be a weight satisfying \eqref{beta-cond}. We now introduce the following definition of weak solutions that is needed for the theory on interior estimates for \eqref{main-eqn}.
\begin{definition} \label{interior-weak-def} Let $p \in (1, \infty)$, $r>0$ and $z_0 = (x_0, t_0) \in \mathbb{R}^{n} \times \mathbb{R}$. We write $Q = Q_{r, \beta}(z_0)$. A function $u \in \W^{1,p}(Q, \beta)$ is said to be a weak solution of
\begin{equation} \label{scaling-eqn}
\beta(x) u_t - \textup{div}(\bA(x,t) \nabla u) = \textup{div}(F(x,t))\quad \text{in} \quad Q,
\end{equation}
if  for every test function $\varphi \in C_0^\infty(Q)$, 
\begin{align*}
& -\int_{Q} \beta(x) u(x,t) \partial_t\varphi(x,t) dx dt + \int_{Q} \wei{\bA(x,t) \nabla u(x,t), \nabla \varphi(x,t)} dx dt \\
& = -\int_{Q} \wei{F(x,t), \nabla \varphi(x,t)} dx dt.
\end{align*}
\end{definition}

In addition to studying local boundary estimates for \eqref{main-eqn}, we also introduce the following notion of weak solutions.
\begin{definition} \label{boundary-weak-def} Let $p \in (1, \infty)$, $r>0$ and $z_0 = (x_0, t_0) \in \mathbb{R}^{n} \times \mathbb{R}$. We write $Q^+ = Q_{r, \beta}^+(z_0)$. A function $u \in \hW^{1,p}(Q^+, \beta)$ is said to be a weak solution of
\begin{equation*} 
\left\{
\begin{aligned}
\beta(x) u_t - \textup{div}(\bA(x,t) \nabla u) & = \textup{div}(F(x,t)) \quad &&\text{in} \quad Q^+,\\[4pt]
u & =  0  \quad &&\text{on} \quad T_{r}(x_0) \times \Gamma_{\beta, z_0}(r),
\end{aligned} \right.
\end{equation*}
if for every test function $\varphi \in C_0^\infty(Q^+)$,
\begin{align*}
& -\int_{Q^+} \beta(x) u(x,t) \partial_t\varphi(x,t) dx dt + \int_{Q^+} \wei{\bA(x,t) \nabla u(x,t), \nabla \varphi(x,t)} dx dt \\
& = -\int_{Q^+} \wei{F(x,t), \nabla \varphi(x,t)} dx dt.
\end{align*}
\end{definition}
Besides the standard heat dilation $(x,t) \mapsto (rx, r^2t)$, let us introduce two special dilation properties for our considered class of equations \eqref{main-eqn} that will be used frequently throughout the paper. Let us consider the equation \eqref{scaling-eqn}. Taking a translation, we may assume without loss of generality that $z_0 = (0,0)$.
Recall that $\Psi_{\beta} (r) = \Psi_{\beta, 0}(r)$, as defined in \eqref{t-height}.  For $(x,t) \in Q_1=B_1\times (-1,0)$, let us now define the following:
\begin{equation} \label{scale-1}
\left\{
\begin{aligned}
\tilde{u}(x,t) &= u(rx, r^2\Psi_\beta(r) t),
\quad &&\tilde{\beta}(x) = [\Psi_{\beta}(r)]^{-1} \beta(rx),\\[4pt]
\tilde{\bA} (x,t)&= \bA(rx, r^2\Psi_{\beta}(r) t), \quad
&&\tilde{F}(x,t) =rF(rx, r^2\Psi_\beta(r) t).
\end{aligned} \right.
\end{equation}
In addition, let us also denote the following dilation only in the time variable, which is very special for the class of equations in \eqref{main-eqn}:
\begin{equation} \label{scale-2}
\left\{
\begin{aligned}
\hat{u}(x,t) &= u(x, \Psi_\beta(r) t),
\quad &&\hat{\beta}(x) = [\Psi_{\beta}(r)]^{-1} \beta(x),\\[4pt]
\hat{\bA} (x,t)&= \bA(x, \Psi_{\beta}(r) t), \quad
&&\hat{F}(x,t) =F(x, \Psi_\beta(r) t).
\end{aligned}  \right.
\end{equation}
It is straightforward to see that
\[
\Psi_{\tilde \beta}(1)=1 \quad \text{and} \quad \Psi_{\hat{\beta}}(r) =1 \quad \text{so that}\quad Q_{1,\tilde \beta} = Q_1 \quad \text{and}  \quad Q_{r,\hat{\beta}} = Q_r.
\]
The following simple lemma will be used frequently throughout the paper. Its proof is also straightforward, and thus omitted.
\begin{lemma} \label{scaling-lemma} Let $p \in (1, \infty)$, and suppose that $u \in \W^{1,p}(Q_{r,\beta}, \beta)$ is a weak solution of \eqref{scaling-eqn} in $Q_{r,\beta}$. Then $\tilde u \in \W^{1,p}(Q_{1}, \tilde{\beta})$ is a weak solution of
\begin{equation*} 
\tilde{\beta}(x) \tilde{u}_t - \textup{div}(\tilde{\bA}(x,t) \nabla \tilde{u}) = \textup{div}(\tilde{F}(x,t)) \quad \text{in} \quad Q_1,
\end{equation*}
where $\tilde{u}$, $\tilde{\beta}$, $\tilde{\bA}$, and $\tilde{F}$ are defined in \eqref{scale-1}. Moreover, $\hat{u} \in  \W^{1,p}(Q_{r}, \hat{\beta})$ is a weak solution of
\[
\hat{\beta}(x) \hat{u}_t - \textup{div}(\hat{\bA}(x,t) \nabla \hat{u}) = \textup{div}(\hat{F}(x,t)) \quad \text{in} \quad Q_{r},
\]
where $\hat u$, $\hat{\beta}$, $\hat{\bA}$, and $\hat{F}$ are defined in \eqref{scale-2}. Similar assertions also hold for weak solutions defined in Definition~\ref{boundary-weak-def}.
\end{lemma}

\section{Lipschitz estimates for simple equations} \label{Lips-est-sec}
Let $\bar{\bf{A}}(t)$ be a measurable $n\times n$ matrix function defined on some open interval to be specified. We assume that $\bar{\bf{A}}(t)$ satisfies the ellipticity and boundedness conditions: there exists a constant $\nu \in (0,1)$ such that
\begin{equation} \label{ellip-constant}
\nu |\xi|^2 \leq \wei{\bar{\bA}(t) \xi,\, \xi} \quad \text{and} \quad |\bar{\bA}(t)| \leq \nu^{-1}, \qquad \forall \ \xi \in \mathbb{R}^n.
\end{equation}
In this section, we study the parabolic equation \eqref{main-eqn} with frozen coefficients. We begin with the following interior Lipschitz estimates.
\begin{lemma} \label{Lipschitz-constant} For every $\nu \in (0,1)$ and $M_0 \geq 1$, there exists a constant $N = N(n, \nu, M_0)>0$ such that the following assertion holds.  Let $\beta$ be a weight satisfying \eqref{beta-cond}, and let $\bar{\bA}(t)$ be a matrix satisfying \eqref{ellip-constant} on $\Gamma_{\beta }(2r)$. Suppose that $v \in \W^{1,2}(Q_{2r, \beta})$ is a weak solution of
\begin{equation*} 
(\beta)_{B_{2r}} v_t - \textup{div}(\bar{\bA}(t)\nabla v) = 0 \quad \text{in} \quad Q_{2r, \beta}.
\end{equation*}
Then 
\begin{equation} \label{Lip-inter-est-1013}
r (\beta)_{B_{2r}} \|v_t\|_{L^\infty(Q_{r, \beta})} + \|\nabla v\|_{L^\infty(Q_{r, \beta})} \leq  N\left(\fint_{Q_{2r, \beta}} |\nabla v|^2\, dxdt  \right)^{1/2}.
\end{equation}
\end{lemma} 
\begin{proof}By the dilation $(x,t) \mapsto (rx, r^2 \Psi_{\beta}(2r)t)$ and Lemma \ref{scaling-lemma}, it is sufficient to prove the lemma when
\[
r =1 \quad \text{and} \quad \Psi_{\beta}(2) =1.
\]
In this setting, we note that $Q_{2,\beta} = Q_2$, and the lemma is reduced to the case where $v \in \W^{1,2}(Q_2)$ is a weak solution of 
\[
(\beta)_{B_2}v_t - \textup{div}(\bar{\bA} (t) \nabla v)  = 0 \quad \text{in} \quad Q_{2}.
\]
Recall the formulas of $r^2\Psi_{\beta}(r)$ in \eqref{heights} and as $\beta(x)$ satisfies \eqref{beta-cond}, it follows from Lemma \ref{property} that
\begin{equation}\label{cut-off time}
4\eta_0\Psi_{\beta}(2) \leq 4\Psi_{\beta}(2) - \Psi_{\beta}(1) \leq  4 \Psi_\beta(2),
\end{equation} 
for some positive constant $\eta_0\in (0,1)$ depending only on $n, M_0$.  Then, by applying the difference quotient with respect to the $x$ variable and using standard energy estimates, it follows that $v(\cdot, t) \in C^\infty(B_2)$ for all $t \in (-4, 0)$, and
\[
 \|\nabla v\|_{L^\infty(Q_{1, \beta})} \leq N \left(\fint_{Q_{2}} |\nabla v|^2dxdt  \right)^{1/2}
\]
for $N = N(n, \nu, M_0)>0$. This implies the desired estimate for $\nabla v$ in \eqref{Lip-inter-est-1013}. Similarly, we also have
\begin{align*}
\|\nabla^2 v\|_{L^\infty(Q_{1, \beta})} \leq N \Big(\fint_{Q_{3/2, \beta}} |\nabla^2 v|^2dxdt  \Big)^{1/2}\leq N \Big(\fint_{Q_{2}} |\nabla v|^2dxdt  \Big)^{1/2},
\end{align*}
where the last estimate follows from the standard energy estimate. From this and the PDE of $v$, we obtain
\[
\|(\beta)_{B_2}v_t\|_{L^\infty(Q_{1, \beta})} \leq N \|\nabla^2 v\|_{L^\infty(Q_{1, \beta})} \leq N \Big( \fint_{Q_{2}} |\nabla v|^2dxdt  \Big)^{1/2}.
\]
This proves the estimate of $v_t$ in \eqref{Lip-inter-est-1013}. The proof of the lemma is completed.
\end{proof}
Next, for the reader's convenience, we refer to the definitions of $T_{2r}(x_0)$ in \eqref{flat-bdr} and  $\Gamma_{\beta, z_0}(2r)$ in \eqref{cylinder-def}. We then state the following boundary Lipschitz regularity estimates for the class of equations with frozen coefficients. 
\begin{lemma} \label{Bdr-Lipschitz-constant} For every $\nu \in (0,1)$ and $M_0 \geq 1$, there is a constant $N = N(n, \nu, M_0)>0$ such that the following assertion holds. Let $r>0$, $z_0 = (x_0, t_0) \in \mathbb{R}^n \times \mathbb{R}$, $\beta(x)$ be a weight satisfying \eqref{beta-cond}, and $\bar{\bA}(t)$ be a matrix satisfying \eqref{ellip-constant} on $\Gamma_{\beta, z_0}(2r)$. Suppose that $v \in \hW^{1,2}(Q_{2r, \beta}^+(z_0))$ is a weak solution of
\begin{equation*}
\left\{
\begin{aligned}
(\beta)_{B_{2r}(x_0)} v_t - \textup{div}(\bar{\bA}(t)\nabla v) & =0  \quad &&\text{in} \quad Q_{2r, \beta}^+(z_0),\\[4pt]
v & = 0\quad &&\text{on} \quad T_{2r}(x_0) \times \Gamma_{\beta, z_0}(2r).
\end{aligned} \right.
\end{equation*}
Then
\[
r (\beta)_{B_{2r}(x_0)} \|v_t\|_{L^\infty(Q_{r, \beta}^+(z_0))} + \|\nabla v\|_{L^\infty(Q_{r, \beta}^+(z_0))} \leq  N\Big(\fint_{Q_{2r, \beta}^+(z_0)} |\nabla v|^2dxdt  \Big)^{1/2}.
\]
\end{lemma}
\begin{proof} The proof follows similarly to that of Lemma \ref{Lipschitz-constant} with some care on the boundary condition, and we skip the details.
\end{proof}
\section{Interior regularity estimates}\label{interior section}
Similar to the notation used in \eqref{point-oss-def}, for $x_0 \in \R^n$ and $r>0$, we denote
\begin{equation} \label{Theta-beta-def}
\Theta_{\bA, r, x_0}(x,t) = |\bA(x,t) - (\bA)_{B_r(x_0)}(t)|, \quad \Theta_{\beta, r, x_0}(x) = |\beta(x) - (\beta)_{B_r(x_0)}|\beta(x)^{-1/2}.
\end{equation}
When $x_0=0$, we use the following notations for abbreviation:
\begin{equation*} 
\Theta_{\bA, r} (x,t) = |\bA(x,t) - (\bA)_{B_{r}}(t)|, \quad\Theta_{\beta, r} (x)=  |\beta(x) - (\beta)_{B_r}|\beta(x)^{-1/2}.
\end{equation*}
Moreover, we recall that $\Lambda$ is the constant given in Lemma \ref{quasi-metric lemma}. The following result on interior regularity estimates for the class of equations \eqref{main-eqn} is the main result of this section.
\begin{theorem} \label{inter-theorem} For every $\nu \in (0,1)$, $M_0 \geq 1$, and $p \in [2,\infty)$, there exist a sufficiently small constant $\delta_0 = \delta_0(n, \nu, p, M_0)>0$ and a constant $N = N(n, \nu, p, M_0)>0$ such that the following assertions hold. Suppose that  $\beta$ is a weight satisfying \eqref{beta-cond}, $\bA$ is a matrix satisfying \eqref{ellip-cond} in $Q_{2\Lambda,\beta}$, and 
\[
\sup_{r \in (0,1)}\ \sup_{\substack{z=(x,t)\\ \in Q_{1,\beta}}}\left(\fint_{Q_{r,\beta}(z)}\Theta_{\bA, r, x}(y,s)^2\, dyds + \frac{1}{\beta(B_{r}(x))} \int_{B_{r}(x)} \Theta_{\beta, r, x}(y)^2\, dy\right)  \leq \delta_0^2.
\]
Then for every weak solution $u \in \W^{1,2}(Q_{2\Lambda, \beta}, \beta)$ of the equation
\begin{equation} \label{Q2-eqn}
\beta(x) u_t - \textup{div}(\bA (x,t) \nabla u) = \textup{div}(F(x,t))  \quad \text{in} \quad Q_{2\Lambda, \beta},
\end{equation}
with $F \in L^p(Q_{2\Lambda,\beta})^n$, it holds that $u \in \W^{1,p}(Q_{1,\beta}, \beta)$ and
\begin{align*}
&\| \nabla u\|_{L^p(Q_{1,\beta})}+\|\beta u_t\|_{L^p(\Gamma_{\beta}(1),W^{-1,p}(B_1))}\\
&\leq N\left(|Q_{1,\beta}|^{\frac{1}{p}-\frac{1}{2}}\|u\|_{L^2(Q_{2\Lambda,\beta})}+\|F\|_{L^p(Q_{2\Lambda,\beta})}\right).
\end{align*}
\end{theorem}
The rest of this section is to prove Theorem \ref{inter-theorem}. In Subsection \ref{inter-energy}, local energy estimates are established. Local interior approximations of weak solutions to the class of equations \eqref{Q2-eqn} and those of the frozen coefficient equations are achieved in Subsection \ref{inter-appr-est}. From these two subsections, the complete proof of Theorem \ref{inter-theorem} will be presented in Subsection \ref{inter proof}.
\subsection{Interior energy estimates} \label{inter-energy}
In this subsection, we study the problem:
\begin{equation} \label{eqn-Q_R}
\beta(x) u_t - \textup{div}(\bA(x,t) \nabla u)  = \textup{div}(F(x,t)) 
\end{equation}
in some given parabolic cylinder. We begin with the following simple lemma on Caccioppoli type estimates for \eqref{eqn-Q_R}.
\begin{lemma} \label{Caccio-1} For every $\nu \in (0,1)$ and  $M_0 \geq 1$, there exists a constant $N = N(n, \nu, M_0)>0$ such that the following holds. Let $z_0 = (x_0,t_0) \in \R^n \times \R$ and $r > 0$. Suppose that $u \in \W^{1,2}(Q_{2r, \beta}(z_0), \beta)$ is a weak solution of \eqref{eqn-Q_R} in $Q_{2r, \beta}(z_0)$, where $\beta$ is a weight satisfying \eqref{beta-cond}, $\bA$ is a matrix satisfying \eqref{ellip-cond} in $Q_{2r, \beta}(z_0)$, and $F \in L^2(Q_{2r, \beta}(z_0))^n$. Then
\begin{align*}
& \sup_{t \in \Gamma_{\beta, z_0}(3r/2)} \int_{B_{3r/2}(x_0)} u(x,t)^2 \beta(x)\, dx + \int_{Q_{3r/2, \beta}(z_0)} |\nabla u(x,t)|^2\, dx dt\\
&\qquad \leq N \int_{Q_{2r, \beta}(z_0)} u(x,t)^2\Big [ 1+\frac{1}{r^2}+ \frac{\beta(x)}{r^2\Psi_{\beta, x_0}(2r)}\Big]\, dxdt+ N \int_{Q_{2r, \beta}(z_0)} |F|^2\, dxdt.
\end{align*}
\end{lemma}
\begin{proof} By using the translation $(x,t) \mapsto (x - x_0, t-t_0)$, we can assume without loss of generality that $z_0 = (0, 0) \in \R^n \times \R$. Let $\varphi \in C^\infty(Q_{2r, \beta})$ be a standard cut-off function: $\varphi =0$ near $\partial_p Q_{2r, \beta}$, $0 \leq \varphi \leq 1$ in $Q_{2r, \beta}$, 
\begin{equation} \label{test.function}
\varphi =1 \quad \text{in} \quad Q_{3r/2, \beta}, \quad |\nabla \varphi| \leq \frac{N_0}{r}, \quad \text{and} \quad |\partial_t \varphi| \leq \frac{N_0}{ r^2\Psi_{\beta}(2r)},
\end{equation}
where $N_0 = N_0(n, M_0)>0$. We note that the same calculation as in \eqref{cut-off time} is used to achieve the last assertion in \eqref{test.function}. Now, by using Steklov's average if needed (see \cite[p. 18]{DiB} for example), we can test the equation \eqref{eqn-Q_R} with $u\varphi^2$ to obtain
\begin{align*}
& \frac{1}{2}\frac{d}{dt}\int_{B_{2r}} u^2 \varphi^2 \beta(x) dx + \int_{B_{2r}} \wei{\bA \nabla u, \nabla u} \varphi^2 dx \\
& = -\int_{B_r}\Big[2 \wei{\bA \nabla u, \nabla \varphi} u \varphi  + \varphi^2 \wei{F, \nabla u} + 2 u\varphi \wei{F, \nabla \varphi} \Big]dx  + \int_{B_r} u^2 \varphi \partial_t \varphi \beta(x) dx.
\end{align*}
From this, the ellipticity and boundedness conditions on $\bA$ in \eqref{ellip-cond}, we can follow the standard energy estimates, using H\"{o}lder's inequality and Young's inequality, to obtain
\begin{align*}
& \frac{d}{dt}\int_{B_{2r}} u^2 \varphi^2 \beta(x) dx + \nu \int_{B_{2r}} |\nabla (u\varphi)|^2 dx \\
&\leq N(n,\nu)\left\{\int_{B_{2r}} u^2\Big [ \varphi^2+|\nabla \varphi|^2 + |\partial_t \varphi| \beta(x) \Big] dx +\int_{B_{2r}} |F|^2 \varphi^2 dx\right\}.
\end{align*}
Hence, by integrating the last estimate in the time variable, we obtain
\begin{align*}
&\sup_{t \in \Gamma_{\beta}(2r)}\int_{B_{2r}} u^2 \varphi^2 \beta(x) dx + \nu \int_{Q_{2r, \beta}} |\nabla (u\varphi)|^2 dx dt\\
&\leq N(n,\nu)\left\{\int_{Q_{2r, \beta}} u^2\Big [ \varphi^2+ |\nabla \varphi|^2 + |  \partial_t \varphi| \beta(x) \Big] dx dt+\int_{Q_{2r, \beta}} |F|^2 \varphi^2 dx dt\right\}.
\end{align*}
From this and \eqref{test.function}, the lemma is proved.
\end{proof}

Next, we combine Lemma \ref{Caccio-1} with Lemma \ref{embedd-lemma} to obtain an improved version of Lemma \ref{Caccio-1}, in which there is no $L^2(Q_{2r, \beta}, \beta)$-norm of $u$ in the right-hand side. 
\begin{lemma} \label{improved-Caccio-1}  Let $\nu \in (0,1)$ and $M_0 \geq 1$. Assume that $\beta$ is a weight satisfying \eqref{beta-cond}, and that  $\bA$ is a matrix satisfying \eqref{ellip-cond} in $Q_{2, \beta}$. Then, for every weak solution $u \in \W^{1,2}(Q_{2, \beta}, \beta)$ of \eqref{eqn-Q_R} in $Q_{2,\beta}$ with $F \in L^2(Q_{2, \beta})^n$, it holds that
\begin{align} \label{inter-energy-0317}
\sup_{t \in \Gamma_{\beta}(3/2)}\int_{B_{3/2}} u(x,t)^2 \beta(x) dx +  \int_{Q_{3/2, \beta}} |\nabla u|^2dxdt\leq N\int_{Q_{2,\beta}} \big( u^2 + |F|^2\big)dx dt,
\end{align}
where $N = N(n, \nu, M_0)>0$.
\end{lemma}
\begin{proof} By Lemma \ref{Caccio-1}, we have
\begin{align*}
& \sup_{t \in \Gamma_{\beta}(3r/2)} \int_{B_{3r/2}} u(x,t)^2 \beta(x) dx + \int_{Q_{3r/2, \beta}} |\nabla u|^2 dx dt\\
&\qquad \leq N\int_{Q_{2r, \beta}} u^2\Big [ \frac{1}{r^2}  + \frac{\beta(x)}{r^2\Psi_{\beta}(2r)}  \Big] dx dt+ N\int_{Q_{2r, \beta}} |F|^2 dx dt,
\end{align*}
for every $r \in (0,1)$ and $N = N(n, \nu, M_0)>0$.  Note that $(\beta)_{B_{2r}}\leq \Psi_{\beta}(2r)$ by H\"{o}lder's inequality, we can apply Lemma \ref{embedd-lemma} together with Young's inequality to obtain
\[
\frac{1}{r^2\Psi_{\beta}(2r)}\int_{Q_{2r,\beta}}u^2\beta(x)dxdt
\leq \epsilon\int_{Q_{2r, \beta}} |\nabla u|^2 dx dt+\frac{N(n, M_0, \epsilon)}{r^2}\int_{Q_{2r,\beta}}u^2dxdt,\quad \forall\, \epsilon>0.
\]
Then, by choosing $\epsilon$ sufficiently small depending on $n, \nu$ and $M_0$, we infer that
\begin{align*}
& \sup_{t \in \Gamma_{\beta}(3r/2)} \int_{B_{3r/2}} u(x,t)^2 \beta(x) dx + \int_{Q_{3r/2, \beta}} |\nabla u|^2 dx dt\\
&\leq \frac{1}{2} \int_{Q_{2r, \beta}} |\nabla u|^2 dx dt + \frac{N(n,\nu, M_0)}{r^2}  \int_{Q_{2r, \beta}} u^2 dx dt+ N\int_{Q_{2r, \beta}} |F|^2 dx dt,
\end{align*}
for any $r \in (0,1)$. Then \eqref{inter-energy-0317} follows from the last estimate and a standard iteration argument (see \cite[Lemma 4.3]{Lin}, for example). The proof of the lemma is completed.
\end{proof}
Our next lemma is non-standard in the theory on energy estimates for partial differential equations. However, it provides key estimates to prove the control of $\W^{1,2}(Q_{r,\beta}, \beta)$-norm of weak solutions of \eqref{eqn-Q_R}. 
\begin{lemma}\label{u-L2-Cac} For every $\nu \in (0,1)$ and  $M_0 \geq 1$, there exists a constant $N = N(n, \nu, M_0)>0$ such that the following assertion holds. Suppose that $\beta$ is a weight satisfying \eqref{beta-cond},  and $\bA$ is a matrix satisfying \eqref{ellip-cond} in $Q_{r, \beta}$ with some $r>0$.  Suppose also that $u \in \W^{1,2}(Q_{r, \beta}, \beta)$ is a weak solution of \eqref{eqn-Q_R} in $Q_{r, \beta}$. Then
\begin{align*}
\int_{Q_{r, \beta}} |u - (u)_{Q_{r,\beta}}|^2 dxdt \leq &N r^2\int_{Q_{r, \beta}}(|\nabla u|^2 + |F|^2) dxdt\\
&+ N \Big(\frac{1}{\beta(B_{r})}\int_{B_{r}} \Theta_{\beta, r}(x)^2 dx\Big) \Big(\int_{Q_{r,\beta}} |u - (u)_{Q_{r,\beta}}|^{2} dxdt\Big),
\end{align*}
where $\Theta_{\beta,r}(x)$ is defined in \eqref{Theta-beta-def}.
\end{lemma}
\begin{proof} By applying the dilation \eqref{scale-1} and using Lemma \ref{scaling-lemma}, it suffices to prove the lemma under the assumption that
\[
r =1 \quad \text{and} \quad \Psi_{\beta}(1) =1.
\]
Note that in this setting, we have $Q_1=Q_{1,\beta}$. We prove the lemma by a contradiction argument. Suppose that the assertion does not hold. Then, there exist sequences $\{\beta_k\}_k$, $\{u_k\}_k$, $\{\bA_k\}_k$, and $\{F_k\}_k$ such that, for each $k\in \N$,\, $u_k \in \W^{1,2}(Q_1, \beta_k)$ is a weak solution of
\begin{equation} \label{eqn-u-k}
\beta_k(x) \partial_t u_k - \textup{div}(\bA_k \nabla u_k) = \textup{div}(F_k) \quad \text{in} \quad Q_1,
\end{equation}
where the weight $\beta_k$ satisfying $\Psi_{\beta_k}(1) =1$, and
\[
\beta_k^{-1} \in A_{1+\frac{2}{n_0}} \quad \text{and} \quad [\beta_k^{-1}]_{A_{1+\frac{2}{n_0}}} \leq M_0  \quad \text{for} \quad n_0 =\max\{n, 2\}.
\]
Moreover, $u_k$ satisfies
\begin{align*}
\int_{Q_{1}} |u_k - (u_k)_{Q_{1}}|^2 dxdt &\geq k \left[ \int_{Q_{1}}\Big[|\nabla u_k|^2 + |F_k|^2\Big] dxdt \right.\\
& \quad +  \left.  \left(\frac{1}{\beta_k(B_1)}\int_{B_{1}} \Theta_{\beta_k, 1}(x)^2 dx\right)  \left(\int_{Q_1} |u_k - (u_k)_{Q_1}|^2 dx dt\right) \right].
\end{align*}
Now, by using the scaling $u_k \mapsto u_k/\lambda$ and dividing the PDE by $\lambda = \|u_k-(u_k)_{Q_1}\|_{L^2(Q_1)}$, we can normalize and assume that
\begin{equation} \label{normalize-uk}
\int_{Q_{1}}|u_k - (u_k)_{Q_{1}}|^2 dxdt =1, \quad \forall \ k \in \mathbb{N}.
\end{equation}
Then, for each $k \in \mathbb{N}$,
\begin{equation} \label{uF-k-beta}
\|\nabla u_k\|_{L^2(Q_1)}^2 + \|F_k\|_{L^2(Q_1)}^2+ \frac{1}{\beta_k(B_1)}\|\Theta_{\beta_k, 1}\|_{L^2(B_1)}^2 
\leq \frac{1}{k}.
\end{equation}

\smallskip
Now, from \eqref{normalize-uk}, \eqref{uF-k-beta}, and the PDE \eqref{eqn-u-k}, we can find a constant $N = N(n, \nu)>0$ so that
\begin{align*}
 \|u_k - (u_k)_{Q_1}\|_{\W^{1,2}(Q_1, \beta_k)}  &\leq N \left(\|u_k - (u_k)_{Q_1}\|_{L^{2}(Q_1)}+ \|\nabla u_k\|_{L^{2}(Q_1)} + \|F_k\|_{L^{2}(Q_1)} \right)\\
& \leq 3N, \quad \forall \ k \in \mathbb{N}.
\end{align*}
Note that the condition \eqref{osc-assumption} in Proposition \ref{compactness-lemma} is also satisfied due to \eqref{uF-k-beta}. It then follows from Proposition \ref{compactness-lemma} that there exist a function $u_0 \in \W^{1,2}(Q_1)$ and a subsequence of $\{u_k - (u_k)_{Q_1}\}_k$, which we still denote by $\{u_k - (u_k)_{Q_1}\}_k$, such that
\begin{equation} \label{conv-lemma-ener}
   \left\{
\begin{aligned}
u_k - (u_k)_{Q_1} &\rightarrow u_0 \ &&\text{in} \quad L^2((-1,0),L^{l}(B_1)),\\ 
\nabla u_k  &\rightharpoonup \nabla u_0 \ &&\text{in} \quad L^2(Q_1), 
\end{aligned} \quad \quad \text{as}\quad k \rightarrow \infty,
   \right.
\end{equation}
for every $l \in [1, 2^*)$. Due to \eqref{normalize-uk}, \eqref{uF-k-beta} and \eqref{conv-lemma-ener}, it follows that
\begin{equation} \label{u-zero-proper}
(u_0)_{Q_1} =0, \quad \nabla u_0 =0, \quad \text{and} \quad \fint_{Q_1} |u_0|^2 dxdt =1.
\end{equation}

\smallskip
On the other hand, applying the same argument used to prove \eqref{beta-lim} in Proposition \ref{compactness-lemma}, there exist a constant $\beta_0\in [\frac{1}{M_0},1]$ and a subsequence of $\{\beta_k(x)\}_k$, which we also still denote $\{\beta_k(x)\}_k$, such that
\begin{equation}\label{beta-lim-2}
\lim_{k\rightarrow \infty}(\beta_k)_{B_1}=\beta_0,\quad \text{and}\quad
\beta_k(B_1)\leq |B_1| \leq \beta_k^{-1}(B_1). 
\end{equation}
Due to this, we apply Lemma \ref{L-q-2-wei} with $\mu =\beta_k^{-1} \in A_{1+ \frac{2}{n_0}}$, and \eqref{uF-k-beta} to get
\begin{align*}
 \left(\fint_{B_1} |\beta_k(x) - (\beta_k)_{B_1}|^{q_0} dx \right)^{\frac{1}{q_0}} 
 &\leq N(n, M_0) \left(\frac{1}{\beta_k^{-1}(B_1)}\int_{B_1} \Theta_{\beta_k, 1}(x)^2 dx \right)^{\frac{1}{2}}\\
 &\leq N(n, M_0) \left(\frac{1}{\beta_k(B_1)}\int_{B_1} \Theta_{\beta_k, 1}(x)^2 dx \right)^{\frac{1}{2}}\\
 &\rightarrow 0, \quad \text{as}\quad k\rightarrow \infty,
\end{align*}
where $q_0=q_0(n, M_0)\in [1,2)$ with
\begin{equation}\label{q_0 definition}
q_0 = \frac{2}{1+\tfrac{2}{n_0}-\gamma}=
\left\{
\begin{array}{cl}
\frac{2}{2-\gamma} \quad &\text{for}\quad n=1, 2, \\[4pt]
\frac{2n}{n(1-\gamma)+2}\quad &\text{for} \quad n\geq 3,
\end{array}\right. 
\end{equation}
for some small constant $\gamma = \gamma(n, M_0) \in (0,\frac{2}{n_0})$ defined in Lemma \ref{L-q-2-wei}. Moreover, using the triangle inequality and \eqref{beta-lim-2}, we infer that
\[
\|\beta_k(x)-\beta_0\|_{L^{q_0}(B_1)}\leq \|\beta_k(x)-(\beta_k)_{B_1}\|_{L^{q_0}(B_1)}+\|(\beta_k)_{B_1}-\beta_0\|_{L^{1_0}(B_1)}\rightarrow 0,
\]
as $k\rightarrow \infty$. This implies that
\begin{equation} \label{mu-k-converge-2}
\beta_k(x) \rightarrow \beta_0\in[M_0^{-1}, 1] \quad \text{in} \quad L^{q_0}(B_1) \quad \text{as} \quad k \rightarrow \infty.
\end{equation}

\smallskip
Next,  for each $\varphi \in C_0^\infty(Q_1)$, as $u_k - (u_k)_{Q_1}$ is a weak solution of \eqref{eqn-u-k}, we have
\[
-\int_{Q_{1}}\beta_k  (x) (u_k - (u_k)_{Q_1}) \partial_t \varphi dxdt + \int_{Q_2} \wei{\bA_k \nabla u_k ,\nabla \varphi} dxdt = -\int_{Q_1} \wei{F_k, \nabla \varphi} dxdt.
\]
We claim that, by passing to the limit as $k\rightarrow \infty$ in this equation, we obtain
\begin{equation} \label{u-zero-PDE}
\partial_t u_0 =0.
\end{equation}
If the claim holds, from this and the first two assertions in \eqref{u-zero-proper}, we conclude that $u_0 =0$.  However, this contradicts the last assertion in \eqref{u-zero-proper}, which completes the proof of the lemma.

\smallskip
It remains to prove \eqref{u-zero-PDE}. Indeed, from \eqref{uF-k-beta}, it follows that
\begin{equation} \label{F-k-nalba-k-energy}
\lim_{k\rightarrow \infty}\int_{Q_1} \wei{\bA_k \nabla u_k ,\nabla \varphi} dxdt = \lim_{k\rightarrow \infty} \int_{Q_1}  \wei{F_k, \nabla \varphi} dxdt =0.
\end{equation}
Moreover, we find
\begin{equation}\label{u-k-t-con}
\begin{aligned} 
& \left|\int_{Q_1} \beta_k (x) (u_k - (u_k)_{Q_1}) \partial_t\varphi  dxdt - \int_{Q_1} \beta_0u_0 \partial_t\varphi dx dt\right| \\
& \leq \|\partial_t\varphi\|_{L^\infty(Q_1)}\int_{Q_1} \beta_k (x) |u_k - (u_k)_{Q_1} -u_0|   dxdt  \\ 
& \quad +  \|\partial_t\varphi\|_{L^\infty(Q_1)} \int_{Q_1} |\beta_k (x) -\beta_0||u_0|   dx dt.
\end{aligned}
\end{equation}

Next, we estimate the two terms on the right-hand side of \eqref{u-k-t-con}. When $n\geq 3$, due to the assumption that
$(\beta_k^{\frac{n}{2}})_{B_1}^{\frac{2}{n}}=\Psi_{\beta_k}(1)=1$, we infer that
\begin{align*}
& \int_{B_1} \beta_k (x) |u_k(x,t) - (u_k)_{Q_1} -u_0(x,t)|   dx \\
& \leq  \|\beta_k\|_{L^{n/2}(B_1)} \|u_k(\cdot, t)- (u_k)_{Q_1} -u_0(\cdot, t)\|_{L^{\frac{n}{n-2}}(B_1)} \\
& \leq N(n) \|u_k(\cdot, t) - (u_k)_{Q_1} -u_0(\cdot, t)\|_{L^{\frac{n}{n-2}}(B_1)}.
\end{align*}
Then, by using H\"{o}lder's inequality in the time integration and  the first assertion in \eqref{conv-lemma-ener}, we obtain
\begin{align*}
& \int_{Q_1} \beta_k (x) |u_k(x,t) - (u_k)_{Q_1} -u_0(x,t)|   dx dt \\
& \leq N(n) \|u_k(\cdot, t) - (u_k)_{Q_1} -u_0(\cdot, t)\|_{L^2((-1, 0), L^{\frac{n}{n-2}}(B_1))} \rightarrow 0 \quad \text{as} \quad k \rightarrow \infty.
\end{align*}
On the other hand, when $n=1, 2$, as $\beta_k \in A_2$, it follows from Lemma \ref{R-Holder} that there exist a sufficiently small number $\gamma_0 = \gamma_0(n, M_0) >0$ and a number $N = N(n, M_0)>0$ such that
\[
\left(\fint_{B_1} \beta_k(x)^{1+\gamma_0} dx \right)^{\frac{1}{1+\gamma_0}} \leq N \fint_{B_1} \beta_k(x) dx =N, \quad \forall \ k \in \mathbb{N}.
\]
Then, doing the same thing as we just did, we have
\begin{align*}
& \int_{B_1} \beta_k (x) |u_k(x,t) - (u_k)_{Q_1} -u_0(x,t)|   dx \\
& \leq  \|\beta_k\|_{L^{1+\gamma_0}(B_1)} \|u_k(\cdot, t)- (u_k)_{Q_1} -u_0(\cdot, t)\|_{L^{\frac{1+\gamma_0}{\gamma_0}}(B_1)} \\
&\leq N(n, M_0)\|u_k(\cdot, t)- (u_k)_{Q_1} -u_0(\cdot, t)\|_{L^{\frac{1+\gamma_0}{\gamma_0}}(B_1)} \rightarrow 0 \quad \text{as} \quad k \rightarrow \infty,
\end{align*}
where we used the first assertion of \eqref{conv-lemma-ener} in our last step. 
In summary, for all $n \in \mathbb{N}$, we have
\begin{equation} \label{eqn-03015}
 \lim_{k\rightarrow \infty}\int_{Q_1} \beta_k (x) |u_k(x,t) - (u_k)_{Q_1} -u_0(x,t)|   dx dt =0.
\end{equation}

Next, we consider the second term on the right-hand side of \eqref{u-k-t-con}. Let $q_0'>1$ be the H\"{o}lder conjugate number of $q_0$ in \eqref{q_0 definition}, that is, 
\[
q_0' =
\left\{
   \begin{array}{cl}
   \frac{2}{\gamma}<\infty \quad &\text{for}\quad n=1, 2,\\[4pt]
    \frac{2n}{n(1+\gamma)-2}<2^* = \frac{2n}{n-2} \quad &\text{for}\quad n\geq 3.
   \end{array}\right.
\]
By H\"{o}lder's inequality and the Sobolev embedding theorem, it follows that
\begin{align*}
\int_{B_1} |\beta_k (x) -\beta_0||u_0(x, t) |   dx & \leq \|\beta_k -\beta_0\|_{L^{q_0}(B_1)} \|u_0(\cdot, t)\|_{L^{q_0'}(B_1)} \\
& \leq N(n)\|\beta_k -\beta_0\|_{L^{q_0}(B_1)} \|u_0(\cdot, t)\|_{W^{1,2}(B_1)}
\end{align*}
for a.e. $t \in (-1, 0)$ and all $k \in \mathbb{N}$. From this and \eqref{mu-k-converge-2}, we obtain
\begin{align} \notag
& \int_{Q_1} |\beta_k (x) -\beta_0||u_0|   dx dt  \\ \label{u-zero-term-2}
& \leq N(n) \|\beta_k(x) -\beta_0\|_{L^{q_0}(B_1)} \|u_0\|_{L^2((-1, 0), W^{1,2}(B_1))}\rightarrow 0 \quad \text{as} \quad k \rightarrow 0.
\end{align}
Therefore, we can conclude from \eqref{u-k-t-con}, \eqref{eqn-03015},  and \eqref{u-zero-term-2} that 
\begin{equation} \label{u-k-convergence-1}
\lim_{k\rightarrow \infty}\int_{Q_1} \beta_k (x) (u_k - (u_k)_{Q_1}) \partial_t\varphi  dxdt = \int_{Q_1} \beta_0u_0 \partial_t\varphi dx dt.
\end{equation}
From this and \eqref{F-k-nalba-k-energy}, we infer that $\beta_0\partial_t u_0=0$. Then, because of $\eqref{mu-k-converge-2}$, the claim \eqref{u-zero-PDE} is proved. The proof of the lemma is then completed.
\end{proof}
\begin{lemma} \label{boundedness-u-sol}  For every $\nu \in (0,1)$ and  $M_0 \geq 1$, there exist a sufficiently small positive constant $\tilde \delta= \tilde \delta(n, \nu, M_0)$ and a constant $N=N(n,\nu, M_0)>0$ such that the following assertion holds. Suppose that $\beta$ is a weight satisfying \eqref{beta-cond},  $\bA$ satisfies \eqref{ellip-cond} in $Q_{1, \beta}$, and
\begin{equation} \label{epsilon-1}
\frac{1}{\beta(B_1)} \int_{B_1} \Theta_{\beta, 1}(x)^2 dx \leq \tilde \delta^2,
\end{equation}
Then, for every weak solution $u \in \W^{1,2}(Q_{1, \beta}, \beta)$ of \eqref{eqn-Q_R} in $Q_{1, \beta}$, it holds that
\begin{align*}
\|u - (u)_{Q_{1,\beta}}\|_{\W^{1,2}(Q_{1, \beta}, \beta)} & \leq N \Big[\|\nabla u\|_{L^2(Q_{1, \beta})} + \|F\|_{L^2(Q_{1,\beta})}\Big].
\end{align*}
\end{lemma}
\begin{proof} 
By Lemma \ref{u-L2-Cac}, there exists a constant $N = N(n, \nu, M_0)>0$ such that
\begin{align*}
\int_{Q_{1, \beta}} |u - (u)_{Q_{1,\beta}}|^2 dxdt &\leq N  \int_{Q_{1, \beta}}\Big[|\nabla u|^2 + |F|^2\Big] dxdt\\
&\quad + N  \left(\frac{1}{\beta(B_{1})}\int_{B_{1}} \Theta_{\beta, 1}(x)^2 dx\right) \left(\int_{Q_{1, \beta}}  |u - (u)_{Q_{1,\beta}}|^{2} dxdt  \right).
\end{align*}
Then, we can take $\tilde\delta =\tilde \delta(n, \nu, M_0) >0$ sufficiently small so that $N \tilde \delta^2 \leq \tfrac{1}{2}$. Then, it follows from the assumption \eqref{epsilon-1} that
\[
N\left(\frac{1}{\beta(B_{1})}\int_{B_{1}} \Theta_{\beta, 1}(x)^2 dx\right) \leq \frac{1}{2}.
\]
From this, it follows that
\begin{align} \label{u-L2-Cac-0316}
\int_{Q_{1, \beta}} |u(x,t)- (u)_{Q_{1,\beta}}|^2 dxdt \leq N \int_{Q_{1, \beta}}\Big[|\nabla u(x,t)|^2 + |F(x,t)|^2\Big] dxdt.
\end{align}
On the other hand, by the definition of the space $\W^{1,2}(Q_{1, \beta}, \beta)$ and the PDE of $u$, we see that
\begin{align*}
\|u - (u)_{Q_{1,\beta}}\|_{\W^{1,2}(Q_{1, \beta}, \beta)}\leq N \left[ \|u - (u)_{Q_{1,\beta}}\|_{L^{2}(Q_{1, \beta})} + \|\nabla u\|_{L^{2}(Q_{1, \beta})} + \|F\|_{L^{2}(Q_{1, \beta})}\right].
\end{align*}
Therefore, the assertion of the lemma follows from the last estimate and \eqref{u-L2-Cac-0316}.
\end{proof}
\subsection{Interior approximation estimates} \label{inter-appr-est} Given $r \in (0,1)$ and $z_0 = (x_0, t_0) \in Q_{2\Lambda, \beta}$,  we plan to use freezing-coefficients technique to locally approximate solutions of \eqref{eqn-Q_R} in $Q_{r, \beta}(z_0)$ by solutions of the following frozen coefficient equation:
\begin{equation} \label{v-Qr-sol}
(\beta)_{B_r(x_0)} v_t - \textup{div} ((\bA)_{B_{r}(x_0)}(t) \nabla v)  =  0 \quad \text{in} \quad Q_{r, \beta}(z_0).
\end{equation}
Due to Lemma \ref{Lipschitz-constant}, we see that each weak solution of \eqref{v-Qr-sol} is sufficiently smooth. Our goal is to show that $\nabla u$ is sufficiently close to $\nabla v$ when $\beta(x)$ is sufficiently close to $(\beta)_{B_r(x_0)}$ and $\bA(x,t)$ is sufficiently close to $(\bA)_{B_r(x_0)}(t)$.  Here, we refer the reader to the definitions of $\Theta_{\bA, r, x_0} (x,t)$ and $\Theta_{\beta, r, x_0}(x)$ in \eqref{Theta-beta-def}.

\smallskip
We begin with the following important lemma on local energy estimates or Caccioppoli type estimates for the difference of two solutions of \eqref{eqn-Q_R} and \eqref{v-Qr-sol}.

\begin{lemma} \label{compare-lemma-1} Suppose that $u \in \W^{1,2}(Q_{r,\beta}(z_0), \beta)$ is a weak solution of \eqref{eqn-Q_R} in $Q_{r, \beta}(z_0)$, where the weight $\beta$ satisfies \eqref{beta-cond} and the matrix $\bA$ satisfies \eqref{ellip-cond} in $Q_{r, \beta}(z_0)$. Suppose also that $v \in \W^{1,2}(Q_{r,\beta}(z_0))$ is a weak solution of \eqref{v-Qr-sol} in $Q_{r, \beta}(z_0)$.  Then, for 
\[
w = u- (u)_{Q_{r,\beta}(z_0)}-v,
\] 
it holds that
\begin{align*}
&\fint_{Q_{r, \beta}(z_0)} |\nabla w|^2 \varphi^2 dx dt \leq   N \fint_{Q_{r, \beta}(z_0)} |F|^2 \varphi^2 dx dt \\
& \quad + N\fint_{Q_{r, \beta}(z_0)} w^2 \Big[\varphi^2+ |\nabla \varphi|^2 + \beta(x)\Big(\frac{\varphi^2}{r^2(\beta)_{B_r(x_0)}} + |\partial_t \varphi| \Big)\Big]dxdt \\
& \quad +   N \|\varphi \nabla v\|_{L^\infty(Q_{r, \beta}(z_0))}^2 \fint_{Q_{r, \beta}(z_0)} \Theta_{\bA, r, x_0}(x,t)^2 dx dt \\
& \quad +N\|r(\beta)_{B_{r}(x_0)} v_t \varphi\|_{L^\infty(Q_{r, \beta}(z_0))}^2  \left( \frac{1}{\beta(B_r(x_0))} \int_{B_{r}(x_0)} \Theta_{\beta, r, x_0}(x)^2 dx\right),
\end{align*}
for every $\varphi  \in C_0^\infty(Q_{r, \beta}(z_0))$, where $N = N(n, \nu)>0$.
\end{lemma}
\begin{proof} Observe that from Lemma \ref{Lipschitz-constant}, we see that $v_t, \nabla v \in L^{\infty}_{\text{loc}}(Q_{r,\beta}(z_0))$. Consequently, if $u \in \W^{1,2}(Q_{r, \beta}(z_0), \beta)$ is a solution to \eqref{eqn-Q_R} in $Q_{r, \beta}(z_0)$, then 
\[   
w = u - (u)_{Q_{r,\beta}(z_0)}- v \in \W^{1,2}_{\text{loc}}(Q_{r, \beta}(z_0), \beta)
\]   
is a weak solution of
\begin{align*}
\beta(x) w_t  - \textup{div}( \bA \nabla w) & = \textup{div}(F)- (\beta(x) - (\beta)_{B_r(x_0)}) v_t \\
&\quad + \textup{div}[(\bA - (\bA)_{B_{r}(x_0)}(t))\nabla v] \quad \text{in} \quad Q_{r,\beta}(z_0).
\end{align*}
By using Steklov's average if needed (see \cite[p. 18]{DiB} for example), we can test the equation of $w$ with $w\varphi^2$ and use the ellipticity and boundedness condition in \eqref{ellip-cond} for $\bA$ to obtain
\begin{align} \notag
& \frac{1}{2}\frac{d}{dt} \int_{B_r(x_0)} \beta(x) w^2 \varphi^2 dx + \nu \int_{B_{r}(x_0)} |\nabla w|^2 \varphi^2 dx \\ \notag
& \leq N \int_{B_r(x_0)} \Big[ |w| |\nabla w| |\nabla \varphi | \varphi + |F| \big( |\nabla w| \varphi^2 + \varphi |\nabla \varphi| |w| \big)\Big] dx \\ \notag
& \quad + \int_{B_r(x_0)}  |\bA - (\bA)_{B_{r}(x_0)}(t)| |\nabla v| \big( |\nabla w| \varphi^2 + \varphi |\nabla \varphi| |w| \big)dx \\ \label{cacci-1}
& \quad + \int_{B_r(x_0)} \Big[ |\beta - (\beta)_{B_r(x_0)}| |v_t| |w| \varphi^2  + \beta(x) w^2 \varphi |\partial_t \varphi| \Big] dx,
\end{align}
where $N = N(n, \nu)>0$. For simplicity in writing, we denote 
\[
K=\|r (\beta)_{B_{r}(x_0)}v_t\|_{L^\infty(Q_{r, \beta}(z_0))}
\]
in the following proof of this lemma. We now apply Young's inequality to control terms on the right-hand side.  In particular, for the term involving $v_t$, we have 
\begin{align*}
 & \int_{B_r(x_0)} |\beta - (\beta)_{B_r(x_0)}| |v_t| |w| \varphi^2  dx \\
& \leq  K \left( \frac{1}{(\beta)_{B_r(x_0)}} \int_{B_r(x_0)}  \Theta_{\beta, r, x_0}(x)^2 dx \right)^{1/2} \left(\frac{1}{r^2 (\beta)_{B_r(x_0)}} \int_{B_r(x_0)} w^2 \beta(x) \varphi^2 dx \right)^{1/2}\\
& \leq \frac{K^2}{2(\beta)_{B_r(x_0)}} \int_{B_r(x_0)}  \Theta_{\beta, r, x_0}(x)^2  dx+ \frac{1}{2r^2 (\beta)_{B_r(x_0)}} \int_{B_r(x_0)} w^2 \beta(x) \varphi^2 dx.
\end{align*}
By treating all the other terms on the right-hand side of \eqref{cacci-1} in the standard way, we obtain
\begin{align*}
& \frac{d}{dt} \int_{B_r(x_0)} \beta(x) w^2 \varphi^2 dx + \nu \int_{B_{r}(x_0)} |\nabla w|^2 \varphi^2 dx \\
&\leq N\int_{B_r(x_0)} w^2 \Big[ \varphi^2+|\nabla \varphi|^2 +\beta(x) \Big( \frac{\varphi^2}{r^2(\beta)_{B_r(x_0)}} + |\partial_t \varphi| \Big)\Big]dx   \\
&\quad + N \int_{B_r(x_0)} |F|^2 \varphi^2 dx +  N \|\varphi \nabla v\|_{L^\infty(Q_{r, \beta}(z_0))}^2 \int_{B_r(x_0)} \Theta_{\bA, r, x_0}(x,t)^2 dx  \\
&\quad + N K^2 \left( \frac{|B_r(x_0)|}{\beta(B_r(x_0))} \int_{B_r(x_0)} \Theta_{\beta, r, x_0}(x)^2 dx \right).
\end{align*}
Integrating this last estimate with respect to $t$ on $\Gamma_{\beta,z_0}(r)=(t_0 -r^2\Psi_{\beta,x_0}(r), t_0)$, we obtain 
\begin{align*}
& \frac{1}{r^2\Psi_{\beta,x_0}(r)}\sup_{t \in \Gamma_{\beta,z_0}(r)} \fint_{B_r(x_0)} w^2\varphi(x,t)^2 dx + \fint_{Q_{r, \beta}(z_0)} |\nabla w|^2 \varphi^2 dx dt\\
& \leq N\fint_{Q_{r, \beta}(z_0)} w^2 \Big[ \varphi^2+|\nabla \varphi|^2  + \beta(x) \Big(\frac{\varphi^2}{r^2(\beta)_{B_r(x_0)}} + |\partial_t \varphi|\Big) \Big]dxdt \\
&\quad +  N \fint_{Q_{r, \beta}(z_0)} |F|^2 \varphi^2 dx dt +   N \|\varphi \nabla v\|_{L^\infty(Q_{r, \beta}(z_0))}^2 \fint_{Q_{r, \beta}(z_0)} \Theta_{\bA, r, x_0}(x,t)^2 dx dt \\
&\quad +N \|r(\beta)_{B_{r}(x_0)} v_t \varphi\|_{L^\infty(Q_{r, \beta}(z_0))}^2  \left( \frac{1}{\beta(B_r(x_0))} \int_{B_{r}(x_0)}\Theta_{\beta, r, x_0}(x)^2 dx\right).
\end{align*}
This is the desired estimate.
\end{proof}
The following lemma provides the closeness of the two PDEs \eqref{eqn-Q_R} and \eqref{v-Qr-sol}. The proof of the lemma is the most technical one in this subsection.
\begin{lemma} \label{L2-comparision} Let $\nu \in (0,1)$ and $M_0 \geq 1$ be fixed. Then, for every $\epsilon \in (0, 1)$, there exists a sufficiently small constant $\bar{\delta} =\bar{\delta}(n, \nu, M_0, \epsilon)>0$  such that the following assertions hold. Suppose that $\beta$ is a weight satisfying \eqref{beta-cond}, and that $\bA$ is a matrix satisfying \eqref{ellip-cond} in $Q_{4, \beta}$. Moreover, assume that
\begin{align*}
& \fint_{Q_{4,\beta}}\Theta_{\bA, 4}(x,t)^2 dxdt  + \frac{1}{\beta(B_4)} \int_{B_4} \Theta_{\beta, 4}(x)^2 dx  + \fint_{Q_{4,\beta}}  |F(x,t)|^2dx dt \leq \bar{\delta}^2.
\end{align*}
Then, for every weak solution $u \in \W^{1,2}(Q_{4,\beta},\beta)$ of \eqref{eqn-Q_R} in $Q_{4,\beta}$ satisfying
\[
\fint_{Q_{4,\beta}} |\nabla u(x,t)|^2dxdt \leq 1,
\]
there exists a weak solution $v \in \W^{1,2}(Q_{4,\beta})$ of \eqref{v-Qr-sol} in $Q_{4,\beta}$ such that
\begin{equation} \label{L-2-w-small}
\left(\fint_{Q_{4, \beta}} |u-(u)_{Q_{4,\beta}}-v|^2 dx dt\right)^{1/2}  \leq \epsilon.
\end{equation}
In addition, there exist constants $N = N(n, \nu, M_0)>0$ and $\theta = \theta(n, M_0) \in (0,1)$ such that
\begin{equation} \label{L-2-mu}
\begin{split}
& \fint_{Q_{7/2,\beta}} |\nabla v|^2 dx dt \leq N, \quad \text{and} \\
& \frac{1}{(\beta)_{B_{7/2}}} \fint_{Q_{7/2, \beta}} |u-(u)_{Q_{4, \beta}}-v|^2 \beta(x) dx dt \leq N \epsilon^{2(1-\theta)}.
\end{split}
\end{equation}
\end{lemma}
\begin{proof}
Observe that it suffices to prove the lemma with some $\bar{\delta} \in (0, \tilde \delta]$, where $\tilde \delta = \tilde \delta(n, \nu, M_0)>0$ is the constant defined in Lemma \ref{boundedness-u-sol}. In addition, by applying the time dilation as in \eqref{scale-2}, and replacing $u$, $\beta$, $\bA$, and $F$ with their rescaled ones, we may assume without loss of generality that
\begin{equation} \label{scaling-beta-k}
\Psi_{\beta}(4) =1 \quad \text{so that}\quad Q_{4,\beta} =Q_4.
\end{equation}
Here, it is worth noting that this procedure is admissible as the conditions in the lemma are invariant under the time dilation \eqref{scale-2}. 

\smallskip
We begin by proving \eqref{L-2-w-small} via a contradiction argument. Assume that the assertion \eqref{L-2-w-small} is not true, then there would be $\epsilon_0>0$, a sequence of weights $\{\beta_k\}_k$ satisfying \eqref{beta-cond} for each $k\in \N$, i.e.,
\begin{equation} \label{nabla-u-k}
\beta_k^{-1} \in A_{1+\frac{2}{n_0}} \quad \text{and} \quad [\beta_k^{-1}]_{A_{1+\frac{2}{n_0}}} \leq M_0  \quad \text{for} \quad n_0 = \max\{n, 2\}, 
\end{equation} 
a sequence of coefficient matrices $\{\bA_k\}_k$ satisfying \eqref{ellip-cond} in $Q_{4}$, and a sequence $\{F_k\}_k \subset L^2(Q_4)^n$ such that
\begin{align} \label{a-mu-k}
\fint_{Q_{4}} \Theta_{\bA_k, 4}^2 dxdt+\frac{1}{\beta_k(B_4)}\int_{B_4} \Theta_{\beta_k, 4}(x)^2 dx+\fint_{Q_{4}} |F_k|^2 dx dt \leq \min\big\{\frac{1}{k^2}, \tilde \delta^2\big\}.
\end{align}
Moreover, there exists a sequence $\{u_k\}_k \subset \W^{1,2}(Q_4, \beta_k)$ satisfying
\begin{align}\label{gradient small}
\fint_{Q_{4}} |\nabla u_k(x,t)|^2dxdt \leq 1, \quad \forall \ k \in \mathbb{N},
\end{align}
where each $u_k$ is a weak solution of 
\[
\beta_k\partial_tu_k-\textup{div}(\bA_k(x,t)\nabla u_k)=\textup{div}(F_k(x,t))\quad \text{in}\quad Q_4.
\]
However, for every $k \in \mathbb{N}$,
\begin{equation} \label{w-k-contra}
\fint_{Q_{4}} |u_k - (u_k)_{Q_{4}} -v|^2 dxdt \geq \epsilon_0
\end{equation}
for any weak solution $v \in \W^{1,2}(Q_{4})$ of
\[
(\beta_k)_{B_4} v_t - \textup{div} [(\bA_k)_{B_{4}}(t) \nabla v] =0 \quad \text{in} \quad Q_{4}.
\]
\smallskip
From \eqref{scaling-beta-k}, \eqref{nabla-u-k}, and \eqref{a-mu-k}, we apply the same argument used to prove \eqref{beta-lim-2} and \eqref{mu-k-converge-2} in Lemma \ref{u-L2-Cac} to find a constant $\beta_0\in [\frac{1}{M_0},1]$ and a subsequence of $\{\beta_k\}_k$, which we still denote $\{\beta_k\}_k$, so that
\begin{equation}\label{beta_0}
\lim_{k\rightarrow \infty}(\beta_k)_{B_4}=\beta_0,\quad \text{and}\quad
\beta_k(B_4)\leq |B_4| \leq \beta_k^{-1}(B_4).
\end{equation}

\smallskip
Now, due to the assumption \eqref{ellip-cond}, we see that the sequence of matrices $\{(\bA_k)_{B_4}(t)\}_k$ is bounded in $L^{\infty}((-16,0), \M^{n\times n})$. Hence, there exist a subsequence that we still denote by $\{(\bA_k)_{B_4}(t)\}_k$ and a matrix $\bA_0(t)\in L^{\infty}((-16,0), \M^{n\times n})$  satisfying \eqref{ellip-cond} such that $(\bA_k)_{B_4}(t)\rightharpoonup \bA_0(t)$ weak-$^*$ in $L^{\infty}((-16,0), \M^{n\times n})$, that is, 
\begin{equation}\label{weak star convergence}
\lim_{k\rightarrow \infty}\int_{-16}^{0}\textbf{tr}\Big([(\bA_k)_{B_4}(t)-\bA_0(t)]\bB(t)\Big) dt=0,
\end{equation}
for all $\bB(t)\in L^{1}((-16,0), \M^{n\times n})$. Also, from \eqref{a-mu-k}, \eqref{gradient small}, we can apply Lemma \ref{boundedness-u-sol} to find a number $N = N(n, \nu, M_0)>0$ such that
\[
\|u_k - (u_k)_{Q_{4}}\|_{\W^{1,2}(Q_4, \beta_k)} \leq N, \quad \forall \ k \in \mathbb{N}. 
\]
From this, \eqref{nabla-u-k}, Proposition \ref{compactness-lemma}, and by passing a subsequence, we can find $u_0 \in \W^{1,2}(Q_4)$ such that
\begin{equation} \label{uk-convergence}
\begin{aligned}
u_k - (u_k)_{Q_4} \rightarrow u_0 \quad &\text{in} \quad L^2((-16, 0), L^{l}(B_4)),\\
 \nabla u_k \rightharpoonup \nabla u_0  \quad &\text{in} \quad  L^2(Q_4), 
\end{aligned} \quad \quad \text{as} \quad k \rightarrow \infty,
\end{equation}
for any $l \in [1, 2^*)$. 

\smallskip
Next, for a fixed $\varphi \in C_0^\infty(Q_4)$, and for every $k \in \mathbb{N}$, we use $\varphi$ as a test function for the equation of $u_k$ to obtain
\begin{equation} \label{weak-form-u-k}
\begin{split}
\int_{Q_4} \beta_k (u_k - (u_k)_{Q_4}) \partial_t \varphi  dxdt-  \int_{Q_4} \wei{\bA_k \nabla u_k, \nabla \varphi} dxdt = \int_{Q_4} \wei{F_k, \nabla \varphi}  dxdt.
\end{split}
\end{equation}
Our goal now is to handle each term in \eqref{weak-form-u-k} and pass through the limit as $k\rightarrow \infty$ to derive the equation for $u_0$. For the first term on the left-hand side of \eqref{weak-form-u-k}, we have
\begin{equation}\label{u-k-convergence}
\lim_{k\rightarrow \infty}\int_{Q_4} \beta_k (x) (u_k - (u_k)_{Q_4}) \partial_t\varphi  dxdt = \int_{Q_4} \beta_0u_0 \partial_t\varphi dx dt.
\end{equation}
Since the proof is the same as that of \eqref{u-k-convergence-1} in Lemma \ref{u-L2-Cac}, we omit the details.

\smallskip
Now, for the second term in \eqref{weak-form-u-k}, we note that 
 \[
\begin{aligned}
\int_{Q_4} \wei{\bA_k \nabla u_k, &\nabla \varphi} - \wei{\bA_0 \nabla u_0, \nabla \varphi} dxdt =\int_{Q_4}\Big\{\wei{\bA_k(\nabla u_k-\nabla u_0), \nabla \varphi}\\
&+\wei{[\bA_k-(\bA_k)_{B_4}(t)]\nabla u_0, \nabla \varphi}+\wei{[(\bA_k)_{B_4}(t)-\bA_0(t)]\nabla u_0, \nabla \varphi}\Big\}dxdt.
\end{aligned}
\]
Using the boundedness of $\bA_k$ and $\nabla \varphi$ on $Q_4$, the second assertion of \eqref{uk-convergence}, \eqref{a-mu-k}, and \eqref{weak star convergence}, we infer that
\begin{equation} \label{A-k-weak-con-3-17}
\lim_{k\rightarrow \infty}\int_{Q_4} \wei{\bA_k \nabla u_k, \nabla \varphi} dxdt=\int_{Q_4} \wei{\bA_0(t) \nabla u_0, \nabla \varphi} dxdt.
\end{equation}

Lastly, the term on the right-hand side of \eqref{weak-form-u-k} can easily be seen to converge to zero as $k\rightarrow \infty$ due to \eqref{a-mu-k}. From this, \eqref{u-k-convergence}, \eqref{A-k-weak-con-3-17}, and by taking $k \rightarrow \infty$, it follows from \eqref{weak-form-u-k} that $u_0 \in \W^{1,2}(Q_4)$ is a weak solution of
\begin{equation} \label{u-0-eqn}
\beta_0\partial_t u_0 - \textup{div} (\bA_0(t) \nabla u_0) = 0 \quad \text{in} \quad Q_4.
\end{equation}
Note that due to \eqref{gradient small} and the convergences in \eqref{uk-convergence}, we have
\begin{equation} \label{L2-u_0}
\fint_{Q_4} u_0(x,t) dxdt =0 \quad \text{and} \quad \fint_{Q_4} |\nabla u_0|^2 dx dt \leq 1.
\end{equation}

\smallskip
Next, let us denote
\[
g_k = [(\beta_k)_{B_4}-\beta_0]\partial_t u_0 -\textup{div} ([(\bA_k)_{B_4} (t)-\bA_0(t)] \nabla u_0).
\]
From the triangle inequality, the PDE \eqref{u-0-eqn}, the second estimate in \eqref{L2-u_0}, the ellipticity and boundedness conditions of $\bA_k$ in \eqref{ellip-cond}, the definition of $\bA_0$ in \eqref{weak star convergence}, and the boundedness of $\beta_0$ in \eqref{beta_0}, we infer that there is a constant $N = N(n, \nu, M_0)>0$ such that
\begin{equation}  \label{g-k-convergence-09-13}
\begin{split} 
& \|g_k \|_{L^2((-16, 0), W^{-1, 2}(B_4))} \leq N, \quad \forall \ k \in \mathbb{N},\\ 
& g_k  \rightharpoonup  0 \quad \text{in} \quad L^2((-16, 0), W^{-1, 2}(B_4)) \quad \text{as} \quad  k \rightarrow \infty.
\end{split}
\end{equation}
Now, let $h_k \in \W^{1,2}_*(Q_4)$ be the weak solution of the equation
\begin{equation} \label{h-k-sol}
\left\{
\begin{aligned}
(\beta_k)_{B_4}\partial_t h_k -\textup{div}[(\bA_k)_{B_4}(t) \nabla h_k] & = g_k \quad &&\text{in} \quad Q_4, \\
h_k &=0  \quad &&\text{on}  \quad \partial_p Q_4.
\end{aligned}\right.
\end{equation}
Note that the estimate of $g_k$ in \eqref{g-k-convergence-09-13}, the ellipticity and boundedness conditions of $\bA_k$ \eqref{ellip-cond}, and the definitions of $(\bA_k)_{B_4}$ and $(\beta_k)_{B_4}$ allow us to obtain the existence of $h_k$ by the Galerkin method. In addition, from \eqref{beta_0}, and the estimate of $g_k$ in \eqref{g-k-convergence-09-13}, and the standard energy estimate for $h_k$, we infer that there is $N = N(n, \nu, M_0)>0$ such that
\[
\|h_k\|_{\W^{1,2}_*(Q_4)} \leq N, \quad \forall \ k \in \mathbb{N}.
\]
From this, we apply the classical Aubin-Lions theorem (see \cite[Theorem 1]{Simon}), and by passing through a subsequence if needed, we can find $h \in \W^{1,2}_*(Q_4)$ such that
\begin{equation} \label{hk-convergence-h}
   \left\{
\begin{aligned}
h_k \rightarrow h \quad &\text{in} \quad L^2((-16, 0), L^{2}(B_4)),\\
\nabla h_k  \rightharpoonup \nabla h  \quad & \text{in} \quad L^2((-16, 0), L^{2}(B_4)), \\
\partial_t h_k \rightharpoonup \partial_t h  \quad &\text{in} \quad    L^2((-16, 0), W^{-1,2}(B_4)),
\end{aligned} \quad \quad \text{as} \quad k \rightarrow \infty,
   \right.
\end{equation}
From this, \eqref{beta_0}, \eqref{weak star convergence}, and the convergences of $g_k$ in \eqref{g-k-convergence-09-13}, we can use the weak form of \eqref{h-k-sol}  to pass the limit as $k\rightarrow \infty$ in the same way as the proof of \eqref{u-0-eqn} to conclude that $h \in \W^{1,2}_*(Q_4)$ is a weak solution of 
\begin{equation*} 
\left\{
\begin{aligned}
\beta_0 \partial_t h  -\textup{div}(\bA_0(t) \nabla h)&= 0 \quad &&\text{in} \quad Q_4, \\
h &=0 \quad &&\text{on}  \quad \partial_p Q_4.
\end{aligned} \right.
\end{equation*}
This implies $h =0$ and consequently
\begin{equation}\label{h-k-unweight}
\begin{aligned}
&\|h_k\|_{L^2(Q_4)} \rightarrow 0 \quad \text{as} \quad k \rightarrow \infty.
\end{aligned}
\end{equation}
Then, set $v_k = u_0 - h_k$, we see that $v_k \in \W^{1,2}(Q_4)$ is a weak solution of
\[
(\beta_k)_{B_4}\partial_t v_k -\textup{div}((\bA_k)_{B_4}(t) \nabla v_k)  = 0 \quad \text{in}\quad Q_4.
\]
However, with $w_k = u_k -(u_k)_{Q_4}-v_k=u_k -(u_k)_{Q_4}-u_0+h_k$, it follows from \eqref{uk-convergence} and \eqref{h-k-unweight} that
\begin{align*}
\| w_k\|_{L^2(Q_4)} \leq   \|u_k - (u_k)_{Q_4} -u_0\|_{L^2(Q_4)} + \|h_k\|_{L^2(Q_4)} \rightarrow 0 \quad \text{as} \quad k \rightarrow \infty.
\end{align*}
This contradicts \eqref{w-k-contra} when $v_k$ is in place of $v$ with sufficiently large $k$. The assertion \eqref{L-2-w-small} is then proved.

\smallskip
The remaining part of the proof is to prove \eqref{L-2-mu}.  First, recall that $\bar\delta \leq \tilde \delta$, where $\tilde \delta$ is the constant defined in Lemma \ref{boundedness-u-sol}. Due to this and the assumption \eqref{beta-cond}, we can apply Lemma \ref{boundedness-u-sol} to infer that
\begin{align*}
\fint_{Q_{4}} |u - (u)_{Q_{4}}|^2 dx dt \leq N \fint_{Q_{4}} \Big[ |\nabla u|^2 + |F|^2 \Big] dxdt\leq N\big(1 + \bar{\delta}^2) \leq 2N,
\end{align*}
where $N = N(n, \nu, M_0)>0$. From this, \eqref{L-2-w-small}, and the triangle inequality, we see that
\begin{align*}
 \fint_{Q_{4}}|v|^2 dx dt 
 & \leq 2\fint_{Q_{4}} |v- [u - (u)_{Q_{4}}]|^2 dxdt  + 2\fint_{Q_{4}} |u - (u)_{Q_{4}}|^2 dx dt\\
& \leq 2\epsilon^2 + 4N \leq  2+ 4N,
\end{align*}
where we also used the fact that $\epsilon \in (0,1)$. From this, and by the standard energy estimate for the equation of $v$ and the doubling properties of $\beta^{\frac{n}{2}}$ (when $n\geq 3$) or of $\beta$ (when $n=1,\, 2$), we obtain
\[
\fint_{Q_{7/2, \beta}} |\nabla v|^2 dx dt \leq N \fint_{Q_{4}} |v|^2 dx dt \leq N.
\]
This proves the first assertion in \eqref{L-2-mu}. It now remains to prove the second assertion in \eqref{L-2-mu}. Due to the assumptions in the statement of the lemma and the doubling properties, we can apply Lemma \ref{embedd-lemma} to $w=u-(u)_{Q_{4}}-v$ to infer that 
\[
\begin{split}
&\frac{1}{(\beta)_{B_{7/2}}}\fint_{Q_{7/2,\beta}} |w(x,t)|^2 \beta(x)dxdt\\
& \leq N  \Big(\fint_{Q_{7/2,\beta}} |w|^2 dxdt \Big)^{1-\theta}  \left[ \Big(\fint_{Q_{7/2, \beta}} |w|^{2} dxdt \Big)^{\theta} +  \Big(\fint_{Q_{7/2,\beta}} |\nabla w|^{2} dxdt \Big)^{\theta} \right]\\
& \leq N \epsilon^{2(1-\theta)} \left[ \epsilon^{2\theta} +  \Big(\fint_{Q_{4}} |\nabla u|^{2} dxdt \Big)^{\theta}  + \Big(\fint_{Q_{7/2,\beta}} |\nabla v|^{2} dxdt \Big)^{\theta}\right] \\
& \leq N \epsilon^{2(1-\theta)},
\end{split}
\]
where $N = N(n, \nu, M_0)>0$. Hence, the second assertion of \eqref{L-2-mu} is proved, and the proof of the lemma is completed.
\end{proof}
The following proposition is the main result of this subsection.
\begin{proposition} \label{L2-gradient-comparision} Let $\nu\in (0,1)$ and $M_0 \geq 1$ be fixed. Then, for every $\epsilon \in (0,1)$, there exists a sufficiently small constant $\delta = \delta(n, \nu, M_0, \epsilon) \in (0,1)$  such that the following assertions hold. Suppose that $\beta$ is a weight satisfying \eqref{beta-cond}, and that $\bA$ is a matrix satisfying \eqref{ellip-cond} in $Q_{4, \beta}$. Moreover, assume that
\begin{equation} \label{small-data-interior}
 \fint_{Q_{4,\beta}}\Theta_{\bA, 4}(x,t)^2 dxdt + \frac{1}{\beta(B_4)} \int_{B_4} \Theta_{\beta, 4}(x)^2 dx +\fint_{Q_{4,\beta}} |F|^2 dx dt \leq \delta^2.
\end{equation}
Then, for every weak solution $u \in \W^{1,2}(Q_{4,\beta}, \beta)$ of \eqref{eqn-Q_R} in $Q_{4,\beta}$ satisfying
\[
\fint_{Q_{4,\beta}} |\nabla u(x,t)|^2dxdt \leq 1,
\]
there exists a weak solution $v \in \W^{1,2}(Q_{4,\beta})$ of \eqref{v-Qr-sol} in $Q_{4,\beta}$ such that
\begin{equation} \label{L-2-gradient-w-small}
\left(\fint_{Q_{2, \beta}} |\nabla u-\nabla v|^2 dx dt\right)^{1/2}  \leq \epsilon \quad \text{and} \quad \|\nabla v\|_{L^\infty(Q_{3, \beta})}  \leq N,
\end{equation}
where $N = N(n, \nu, M_0)>0$.
\end{proposition}
\begin{proof} For the given $\epsilon>0$, let $\bar{\epsilon},\,  \delta'\in (0,1)$ be small numbers so that 
\[
\bar{\epsilon}^{2(1-\theta)} \leq \frac{\epsilon^2}{2N_0}\quad \text{and}\quad (\delta')^2 < \frac{\epsilon^2}{2N_0},
\] 
where $\theta = \theta (n, M_0) \in (0,1)$ is defined in Lemma \ref{L2-comparision}, and $N_0 = N_0(n, \nu, M_0)$ is the positive number given in \eqref{last-interior-prop} below.  Also, let us denote 
\[
\delta = \min \big \{\bar{\delta}(n, \nu, M_0, \bar{\epsilon}),\, \tilde \delta (n, \nu, M_0),\,  \delta' \big \},
\]
where $\bar{\delta}(n, \nu, M_0, \bar{\epsilon})$ is  the number defined in Lemma \ref{L2-comparision}, and $\tilde \delta (n, \nu, M_0)$ is defined in Lemma \ref{boundedness-u-sol}. We prove the proposition with this choice of $\delta$. 

\smallskip
To this end, we note that as \eqref{small-data-interior} holds, it follows from Lemma \ref{L2-comparision} and the choice of $\delta$ that there is a weak solution $v \in \W^{1,2}(Q_{4,\beta})$ of \eqref{v-Qr-sol} in $Q_{4,\beta}$ such that
\begin{equation} \label{L-2-w-small-bar}
\left(\fint_{Q_{4, \beta}} |w(x,t)|^2 dx dt\right)^{1/2}  \leq \bar{\epsilon}, \quad \fint_{Q_{7/2,\beta}} |\nabla v|^2 dx dt \leq N, 
\end{equation}
and
\begin{equation} \label{L-2-mu-bar}
 \frac{1}{(\beta)_{B_{7/2}}} \fint_{Q_{7/2, \beta}} |w(x,t)|^2 \beta(x) dx dt \leq N \bar{\epsilon}^{2(1-\theta)},
\end{equation}
where $w= u-(u)_{Q_{4,\beta}}-v$, $N = N(n, \nu, M_0)>0$ and $\theta = \theta(n, M_0) \in (0,1)$.  From this, and by applying Lemma \ref{Lipschitz-constant} and the doubling properties stated in Lemma \ref{property}, we obtain
\begin{equation} \label{L-infty-bar}
\begin{split}
& \|\nabla v\|_{L^\infty(Q_{3,\beta})} \leq N\left(\fint_{Q_{7/2,\beta}} |\nabla v|^2 dx dt\right)^{\frac{1}{2}} \leq N \quad \text{and} \\
& \|(\beta)_{B_{4}} v_t\|_{L^\infty(Q_{3,\beta})} \leq N\left(\fint_{Q_{7/2,\beta}} |\nabla v|^2 dx dt\right)^{\frac{1}{2}} \leq N
\end{split}
\end{equation}
for $N = N(n, \nu, M_0)>0$.  

\smallskip
Note that the second assertion in \eqref{L-2-gradient-w-small} follows from the first estimate in \eqref{L-infty-bar}. It remains to prove the first assertion in \eqref{L-2-gradient-w-small}. By applying Lemma \ref{compare-lemma-1} to $w$ with a suitable cut-off function $\varphi$, and the doubling properties, we have
\begin{align*}
 & \fint_{Q_{2,\beta}} |\nabla w(x,t)|^2 dx dt   \leq N  \fint_{Q_{4,\beta}} \Big[  |F(x,t)|^2+ |w(x,t)|^2 \Big] dxdt   \\
& \qquad + \frac{N}{(\beta)_{B_{7/2}}}  \fint_{Q_{7/2,\beta}} w(x,t)^2 \beta(x)dx dt +  N \|\nabla v\|_{L^\infty(Q_{3, \beta})}^2 \fint_{Q_{4,\beta}} \Theta_{\bA, 4}(x,t)^2 dx dt \\
& \qquad + N \| (\beta)_{B_4} v_t\|_{L^\infty(Q_{3, \beta})}^2 \left(\frac{1}{\beta(B_4)}\int_{Q_{4, \beta}}\Theta_{\beta, 4}(x)^2 dx\right).
\end{align*}
From this, the assumption \eqref{small-data-interior},  the estimates \eqref{L-2-w-small-bar}, \eqref{L-2-mu-bar}, \eqref{L-infty-bar}, and the fact that $\bar{\epsilon} \in (0,1)$, we infer that
\begin{equation} \label{last-interior-prop}
\fint_{Q_{2,\beta}} |\nabla w(x,t)|^2 dx dt   \leq   N_0 \Big[ \delta^2  + \bar{\epsilon}^{2(1-\theta)} \Big] \quad\text{with}\quad N_0 = N_0(n, \nu, M_0)>0.
\end{equation} 
By the choices of $\delta$ and $\bar{\epsilon}$, it follows that
\[
\fint_{Q_{2,\beta}} |\nabla w(x,t)|^2 dx dt \leq \epsilon^2.
\]
Hence, the first assertion in \eqref{L-2-gradient-w-small} follows, and the proposition is proved.
\end{proof}
\subsection{Level set estimates and proof of Theorem \ref{inter-theorem}}\label{inter proof} For $r>0$ and $z_0=(x_0, t_0)\in \R^n\times \R$, we recall the definition of $Q_{r,\beta}(z_0)$ in \eqref{cylinder-def}:
\begin{equation*} 
Q_{r, \beta}(z_0) =B_r(x_0)  \times \Gamma_{\beta, z_0}(r), \quad \text{where}   \quad \Gamma_{\beta, z_0} (r) = (t_0 -r^2 \Psi_{\beta,x_0}(r),\, t_0],
\end{equation*}
in which
\begin{equation}\label{Lambda} 
\Psi_{\beta, x_0}(r)= (\beta^{\frac{n_0}{2}})_{B_r(x_0)}^{\frac{2}{n_0}} = \Big( \fint_{B_{r}(x_0)} \beta(x)^{\frac{n_0}{2}} dx \Big)^{\frac{2}{n_0}}
\quad \text{with}\quad n_0=\max\{n, 2\}.
\end{equation}
Let $C_{r,\beta}(z_0)$ be the cylinder defined as
\begin{equation}\label{centered cylinder}
C_{r,\beta}(z_0)=B_r(x_0)\times \big(t_0-\frac{1}{2}r^2 \Psi_{\beta,x_0}(r),\, t_0+\frac{1}{2}r^2 \Psi_{\beta,x_0}(r)\big).
\end{equation}
We now state some simple properties related to this class of cylinders. Firstly, note that the cylinder $C_{r,\beta}(z_0)$ can be derived from $Q_{r,\beta}(x_0, \hat{t}_0)$ as
\[
C_{r,\beta}(z_0) = Q_{r,\beta}(x_0, \hat{t}_0) \setminus (B_r(x_0) \times \{\hat{t}_0\}) , \quad \text{where}  \quad \hat{t}_0 = t_0+ \frac{1}{2}r^2 \Psi_{\beta,x_0}(r).
\]
Secondly, we can see that
\begin{equation}\label{inclusions}
Q_{r,\beta}(z_0)\subset\{z\in \R^{n}\times \R:\ \rho_{\beta}(z,z_0)\leq r\}\subset C_{2r,\beta}(z_0).
\end{equation}
Lastly, for any $\bar{z}=(\bar{x} ,\bar{t})\in C_{r,\beta}(z_0)$ with $\bar{t} >t_0$, it is not hard to check that $z_0\in Q_{2r,\beta}(\bar{z})$. Then it follows from  \eqref{cylinder-quasi} that $\rho_{\beta}(\bar{z}, z_0)\leq 2r$. Consequently, we can prove that
\begin{equation} \label{quasi-dist-bdry1}
\rho_{\beta}(z,z_0)\leq 2r \quad \text{for all}\quad z\in C_{r,\beta}(z_0).
\end{equation}
\begin{definition}
\textup{(i)} For a given $g\in L^1_{\textit{loc}}(\R^{n}\times \R)$, the Hardy-Littlewood maximal function of $g$ is defined as
\[
(\M g)(z)=\sup_{\rho>0}\fint_{C_{\rho,\beta}(z)}|g(y,s)|dyds,\quad \forall\, z=(x,t)\in \R^{n}\times \R.
\]
\textup{(ii)} If $g$ is defined in a set $U\subset \R^{n}\times \R$, we denote
\[
(\M_{U}g)=\M (g\chi_{U}),
\]
where $\chi_U$ is the characteristic function on $U$.
\end{definition}

Now, we state and prove the following lemma on the density estimate for the upper level sets of $\M_{Q_{2\Lambda,\beta}}(|\nabla u|^2)$.
\begin{lemma}\label{measure density}
Let $\nu\in (0,1)$ and $M_0\geq 1$ be fixed. For any $q_0\in(0,1)$, there exist $K=K(n,\nu, M_0)>1$ and $\hat \delta=
\hat \delta(n,\nu,M_0, q_0)\in (0,1)$ sufficiently small such that the following assertion holds. Suppose that $\beta$ satisfies \eqref{beta-cond}, and $\bA$ is a matrix satisfying \eqref{ellip-cond} in $Q_{2\Lambda, \beta}$ with $\Lambda$  defined in Lemma \ref{quasi-metric lemma}. Moreover, assume that
\begin{equation} \label{smallness-10-21}
\fint_{Q_{r,\beta}(z)}\Theta_{\bA, r, x}(y,s)^2 dyds + \frac{1}{\beta(B_{r}(x))} \int_{B_{r}(x)} \Theta_{\beta, r, x}(y)^2 dy  \leq \hat \delta^2,
\end{equation}
for all $r \in (0,1)$ and for all $z = (z,t) \in Q_{1,\beta}$. Assume also that $u\in \W^{1,2}(Q_{2\Lambda,\beta},\beta)$ is a weak solution  of \eqref{eqn-Q_R} in $Q_{2\Lambda,\beta}$. Then, for any cylinder $C_{\rho,\beta}(z_0)$ with $z_0=(x_0,t_0)\in Q_{1,\beta}$ and $\rho\in(0,\frac{1}{4})$, if
\begin{equation}\label{nonempty}
C_{\rho,\beta}(z_0)\cap\left\{Q_{1,\beta}: \M_{Q_{2\Lambda,\beta}}(|\nabla u|^2)\leq 1\right\}\cap \left\{\M_{Q_{2\Lambda,\beta}}(|F|^2)\leq \hat \delta^2\right\}\neq \emptyset,
\end{equation}
we have
\begin{equation} \label{assertion-1021}
|C_{\rho,\beta}(z_0)\cap\left\{Q_{1,\beta}: \M_{Q_{2\Lambda,\beta}}(|\nabla u|^2)> K\right\}|\leq q_0 |C_{\rho,\beta}(z_0)|.
\end{equation}
\end{lemma}

\begin{proof} Let $\epsilon_0=\epsilon_0(n, M_0, q_0)>0$ be the sufficiently small number to be defined in \eqref{choice of epsilon}. For this $\epsilon_0$, let $\delta=\delta(n, \nu, M_0, \epsilon_0)$ be the number given in Proposition \ref{L2-gradient-comparision}.  Then, let us define
\begin{equation}\label{choice delta}
  \hat \delta=\frac{\delta}{\sqrt{2}}, \quad \text{and} \quad K=\max\big\{2(1+N_2^2),\, N_3\big\},
\end{equation}
where $N_2 = N_2(n, \nu, M_0)$ and $N_3 = N_3(n, M_0)$ are the positive numbers defined in \eqref{smallness of diff} and \eqref{N3-def-1021} below.  We prove the assertion of the lemma with this choice of $\hat{\delta}$ and $K$. 

\smallskip
Firstly, for the given $z_0 = (x_0, t_0)$, let $\hat{z}_0$ be the center point of the top of $\overline{C}_{2\rho, \beta}(z_0) \cap Q_{1,\beta}$, that is,
\[
\hat z_0=(x_0, \hat t_0) \quad \text{with}\quad \hat t_0=\min\big\{0,\ t_0+\tfrac{1}{2}(2\rho)^2 \Psi_{\beta,x_0}(2\rho)\big\},
\]
where $\Psi_{\beta,x_0}(\cdot)$ is given in \eqref{Lambda}. We claim that there exists a constant $N_1=N_1(n, M_0)\geq 1$ such that
\begin{equation}\label{average N-1}
\fint_{Q_{4\rho,\beta}(\hat z_0)}|\nabla u|^2dxdt\leq N_1\quad \text{and}\quad
\fint_{Q_{4\rho,\beta}(\hat z_0)}|F|^2dxdt \leq N_1\hat \delta^2.
\end{equation}
Indeed, from the assumption \eqref{nonempty}, there is $\bar{z} \in C_{\rho,\beta}(z_0) \cap Q_{1,\beta}$ such that
\begin{equation}\label{X-0}
\M_{Q_{2\Lambda,\beta}}(|\nabla u|^2)(\bar z)\leq 1 	\quad \text{and}\quad \M_{Q_{2\Lambda,\beta}}(|F|^2)(\bar z)\leq \hat \delta^2.
\end{equation}
Note that 
\begin{equation}\label{inside property}
\left(C_{\rho,\beta}(z_0)\cap Q_{1,\beta}\right)\subset \left(C_{2\rho,\beta}(z_0)\cap Q_{2\Lambda,\beta}\right)\subset Q_{2\rho,\beta}(\hat z_0).
\end{equation}
Therefore, $\bar z \in Q_{2\rho,\beta}(\hat z_0)$. As $\hat z_0\in Q_{1,\beta}$ and $\bar z \in Q_{2\rho,\beta}(\hat z_0)$, it follows from \eqref{cylinder-quasi} and Lemma \ref{quasi-metric lemma} that 
\[
\left\{\begin{aligned}
\rho_{\beta}(z,0)&\leq \Lambda [\rho_{\beta}(z, \hat z_0)+\rho_{\beta}(\hat z_0, 0)]\leq \Lambda(4\rho+1),\\[4pt]
\rho_{\beta}(z, \bar z)&\leq \Lambda [\rho_{\beta}(z,\hat z_0)+\rho_{\beta}(\hat z_0, \bar z)]\leq 6 \Lambda \rho,
\end{aligned}\right. \quad \text{for all}\quad z\in Q_{4\rho, \beta}(\hat z_0).
\]
By using this, \eqref{inclusions}, and the fact that $\rho \in (0, 1/4)$, we infer that
\begin{equation}\label{inclu-1}
Q_{4\rho,\beta}(\hat z_0)\subset Q_{2\Lambda,\beta} \quad\text{and}\quad Q_{4\rho,\beta}(\hat z_0)\subset  C_{12 \Lambda \rho,\beta}(\bar z).
\end{equation}
Then, by combining \eqref{inclu-1}, \eqref{X-0}, and the doubling properties in Lemma \ref{property},  we can find some constant $N_1=N_1(n, M_0)\geq 1$ so that
\begin{align*}
\fint_{Q_{4\rho,\beta}(\hat z_0)}|\nabla u|^2dxdt 
&\leq\frac{|C_{12\Lambda\rho,\beta}(\bar z)|}{|Q_{4\rho,\beta}(\hat z_0)|} \left(\frac{1}{|C_{12\Lambda\rho,\beta}(\bar z)|}\int_{C_{12\Lambda\rho,\beta}(\bar z)\cap Q_{2\Lambda,\beta}}|\nabla u|^2dxdt\right)\\
&\leq N_1\M_{Q_{2\Lambda,\beta}}(|\nabla u|^2)(\bar z)\leq N_1.
\end{align*}
This proved the first assertion in \eqref{average N-1}. The second assertion also follows similarly. Hence, \eqref{average N-1} is verified.

\smallskip
Secondly, we restrict equation \eqref{eqn-Q_R} to $Q_{4\rho,\beta}(\hat z_0)$. Due to \eqref{average N-1}, the smallness assumption \eqref{smallness-10-21}, and the choice of $\hat{\delta}$ in \eqref{choice delta}, after taking a linear translation and a dilation as in \eqref{scale-1}, we can apply Proposition \ref{L2-gradient-comparision} for $u/\sqrt{2N_1}$ with $F$ replaced by $F/\sqrt{2N_1}$. Then, after scaling back, we can find a weak solution $v \in \W^{1,2}(Q_{4\rho,\beta}(\hat z_0))$ of equation \eqref{v-Qr-sol} in $Q_{4\rho,\beta}(\hat z_0)$ such that
\begin{equation}\label{smallness of diff}
\fint_{Q_{2\rho, \beta}(\hat z_0)} |\nabla u-\nabla v|^2 dx dt \leq 2N_1\epsilon_0^2 \quad \text{and} \quad \|\nabla v\|_{L^\infty(Q_{3\rho, \beta}(\hat z_0))}  \leq N_2,
\end{equation}
where $N_2= N_2(n, \nu, M_0)>0$. 

\smallskip
Finally, we claim that
\begin{equation}\label{claim}
\begin{split}
& \left\{z \in C_{\rho,\beta}(z_0): \M_{Q_{2\rho,\beta}(\hat z_0)}(|\nabla u-\nabla v|^2)(z)\leq 1 \right\} \\
& \subset \left\{z \in C_{\rho,\beta}(z_0): \M_{Q_{2\Lambda,\beta}}(|\nabla u|^2)(z)\leq K\right\}.
\end{split}
\end{equation}
We note that if \eqref{claim} holds, then the assertion \eqref{assertion-1021} of the lemma follows. In fact, \eqref{claim} implies that 
\[
\left\{C_{\rho,\beta}(z_0): \M_{Q_{2\Lambda,\beta}}(|\nabla u|^2)> K\right\}
\subset \left\{C_{\rho,\beta}(z_0): \M_{Q_{2\rho,\beta}(\hat z_0)}(|\nabla u-\nabla v|^2)> 1\right\}.
\]
Then, it follows from the ($1-1$)-weak estimate of the Hardy-Littlewood maximal function and \eqref{smallness of diff} that
\begin{equation*}\label{K-density}
\begin{aligned}
& |\left\{C_{\rho,\beta}(z_0): \M_{Q_{2\Lambda,\beta}}(|\nabla u|^2)> K\right\}| \leq |\left\{C_{\rho,\beta}(z_0): \M_{Q_{2\rho,\beta}(\hat z_0)}(|\nabla u-\nabla v|^2)> 1\right\}| \\
& \leq N(n ,M_0) \int_{Q_{2\rho,\beta}(\hat z_0)} |\nabla u - \nabla v|^2 dx dt \leq 2 N(n, M_0) N_1 \epsilon_0^2 |Q_{2\rho,\beta}(\hat z_0)| \\
&\leq \tilde N(n, M_0) \epsilon_0^2 |C_{\rho,\beta}(z_0)|,
\end{aligned}
\end{equation*}
where the last inequality is from the doubling properties stated in Lemma \ref{property}. Then, by taking $\epsilon_0=\epsilon_0(n, M_0, q_0)\in (0,1)$ sufficiently small so that
\begin{equation}\label{choice of epsilon}
\tilde N(n,M_0)\epsilon_0^2 \leq q_0
\end{equation}
we obtain the conclusion \eqref{assertion-1021}.

\smallskip
It remains to prove \eqref{claim}. To this end, we note that from the definition of the cylinders in  \eqref{centered cylinder}, and Lemma \ref{property}, there is a sufficiently small number $\alpha_0=\alpha_0(n, M_0)\in (0,\frac{1}{2})$ such that
\[
C_{r,\beta}(\tilde z)\subset C_{2\rho,\beta}(z_0)\quad \text{for all}\quad r\in (0, \alpha_0\rho),\ \tilde z\in C_{\rho,\beta}(z_0).
\]
Now, let $z\in C_{\rho,\beta}(z_0)$ be any point in the set on the left-hand side of \eqref{claim}. We have
\begin{equation}\label{Max small}
\M_{Q_{2\rho,\beta}(\hat z_0)}(|\nabla u-\nabla v|^2)(z)\leq 1.
\end{equation}
We only need to prove that 
\begin{equation} \label{different tau}
\frac{1}{|C_{\tau,\beta}(z)|}\int_{C_{\tau,\beta}(z)\cap Q_{2\Lambda,\beta}}|\nabla u|^2dxdt \leq K, \quad \forall \ \tau >0.
\end{equation}
We split the proof of \eqref{different tau} into two cases depending on the size of $\tau$.

\smallskip
\noindent
\textbf{Case 1: $\tau< \alpha_0\rho$.}
In this case, it follows from the definition of $\alpha_0$ that $C_{\tau,\beta}(z)\subset C_{2\rho,\beta}(z_0)$.
Moreover, by \eqref{inside property}, we infer that
\[
C_{\tau,\beta}(z)\cap Q_{2\Lambda,\beta}\subset C_{2\rho,\beta}(z_0)\cap Q_{2\Lambda,\beta}\subset Q_{2\rho,\beta}(\hat z_0).
\]
This, together with \eqref{Max small} and the second assertion in \eqref{smallness of diff}, implies
\begin{align} \notag
\frac{1}{|C_{\tau,\beta}(z)|}\int_{C_{\tau,\beta}(z)\cap Q_{2\Lambda,\beta}}|\nabla u|^2dxdt 
&\leq  \frac{2}{|C_{\tau,\beta}(z)|}\int_{C_{\tau,\beta}(z)\cap Q_{2\rho,\beta}(\hat z_0)}\big(|\nabla u-\nabla v|^2+|\nabla v|^2\big)dxdt\\ \notag
&\leq 2\M_{Q_{2\rho,\beta}(\hat z_0)}(|\nabla u-\nabla v|^2)(z)+2\|\nabla v\|^2_{L^\infty(Q_{2\rho, \beta}(\hat z_0))}\\  \label{case1-1021}
&\leq 2(1+N_2^2).
\end{align}

\noindent
\textbf{Case 2: $\tau\geq \alpha_0\rho$.} Recall that $z,\, \bar{z} \in C_{\rho,\beta}(z_0)$, it then follows from Lemma \ref{quasi-metric lemma} and \eqref{quasi-dist-bdry1} that for all $z' \in C_{\tau,\beta}(z)$,
\begin{align*}
\rho_{\beta}(z', \bar z)
&\leq \Lambda[\rho_{\beta}(z',z)+\rho_{\beta}(z,\bar z)]\\
&\leq \Lambda\left(\rho_{\beta}(z', z)+\Lambda[\rho_{\beta}(z,z_0)+\rho_{\beta}(z_0, \bar z)]\right)\\
&\leq 2\tau \Lambda+4\Lambda^2\rho\leq 4(\Lambda+\alpha_0^{-1}\Lambda^2)\tau,
\end{align*}
where $\Lambda=\Lambda(n, M_0)\geq 2$ is given in Lemma \ref{quasi-metric lemma}.
By \eqref{inclusions}, we infer that
\[
C_{\tau,\beta}(z)\subset C_{\alpha_1\tau,\beta}(\bar z)\quad \text{with}\quad \alpha_1=8(\Lambda+\alpha_0^{-1}\Lambda^2).
\]
Then,  from \eqref{X-0}, and the doubling properties stated in Lemma \ref{property}, we can find some constant $N_3=N_3(n, M_0)>0$ such  that
\begin{equation} \label{N3-def-1021}
\frac{1}{|C_{\tau,\beta}(z)|}\int_{C_{\tau,\beta}(z)\cap Q_{2\Lambda,\beta}}|\nabla u|^2dxdt
\leq \frac{|C_{\alpha_1\tau,\beta}(\bar z)|}{|C_{\tau,\beta}(z)|}\M_{Q_{2\Lambda,\beta}}(|\nabla u|^2)(\bar z)\leq N_3.
\end{equation}

Therefore, \eqref{different tau} follows from \eqref{case1-1021}, \eqref{N3-def-1021}, and the choice of $K$ in \eqref{choice delta}. The proof of the lemma is completed.
\end{proof}

As a consequence of Lemma \ref{measure density}, we obtain the following corollary.
\begin{corollary}\label{corollary-1} Assume that the assumptions of Lemma \ref{measure density} hold.  If 
\[
|C_{\rho,\beta}(z_0)\cap\left\{\M_{Q_{2\Lambda,\beta}}(|\nabla u|^2)> K\right\}|> q_0 |C_{\rho,\beta}(z_0)|,
\]
then
\[
C_{\rho,\beta}(z_0)\cap Q_{1,\beta}\subset\left\{\M_{Q_{2\Lambda,\beta}}(|\nabla u|^2)> 1\right\}\cup\left\{\M_{Q_{2\Lambda,\beta}}(|F|^2)> \hat \delta^2\right\}.
\]
\end{corollary}

Next, to control the level sets of $\M_{Q_{2\Lambda,\beta}}(|\nabla u|^2)$, we need the following covering lemma due to Krylov and Safonov \cite{Krylov-Safonov}. The proof of the lemma is given in Appendix \ref{Appendix B}.
\begin{lemma}[Krylov-Safonov covering lemma]\label{covering} Let $M_0 \geq 1$ and $\beta$ satisfy  \eqref{beta-cond}. Then, for every $q_0\in (0,1)$, there is $\gamma_1 = \gamma_1(n, M_0)>0$ such that the following assertions hold. For every open set $S\subset Q_{1,\beta}$ satisfying $|S|\leq q_0|Q_{1,\beta}|$, let
\[
\A=\big\{C_{\rho,\beta}(z):\  \ z \in Q_{1,\beta},\ \rho>0, \ |C_{\rho,\beta}(z)\cap S|> q_0|C_{\rho,\beta}(z)|\big\},
\]
and
\begin{equation}\label{def-E}
E = \bigcup_{C_{\rho,\beta}(z) \in \mathcal{A}} C_{\rho,\beta}(z) \cap Q_{1,\beta}.
\end{equation}
Then,
\begin{equation*}\label{cover conclu}
|S\setminus E|=0 \quad \text{and}\quad |S|\leq \gamma_1 q_0|E|.
\end{equation*}
\end{lemma}

\smallskip
From this covering lemma and Corollary \ref{corollary-1}, we derive the following lemma on the decay estimates of the measure of the level-sets of  $\M_{Q_{2\Lambda,\beta}}(|\nabla u|^2)$.

\begin{lemma}\label{measure decay lemma}  Suppose that the assumptions in  Lemma~\ref{measure density} hold. Let $K$ be the constant defined in Lemma~\ref{measure density}, and 
\[
S=\left\{\, Q_{1,\beta} : \ \M_{Q_{2\Lambda,\beta}}(|\nabla u|^2)> K \,\right\}.
\]
If $|S|\leq q_0|Q_{1,\beta}|$, then for all $m\in \N$,  
\begin{equation}\label{measure decay}
\begin{aligned}
\bigl|\{Q_{1,\beta} : \M_{Q_{2\Lambda,\beta}}(|\nabla u|^2) > K^m\}\bigr| 
&\leq l_0^{m}\,\bigl|\{Q_{1,\beta} : \M_{Q_{2\Lambda,\beta}}(|\nabla u|^2)> 1\}\bigr| \\
&\quad + \sum_{i=1}^m l_0^{i}\,\bigl|\{Q_{1,\beta} : \M_{Q_{2\Lambda,\beta}}(|F|^2)> K^{m-i}\hat \delta^2\}\bigr|,
\end{aligned}
\end{equation}
where $l_0 = \gamma_1 q_0$ with $\gamma_1$ defined in Lemma \ref{covering}.
\end{lemma}

\smallskip
\begin{proof} For $m=1$, as $|S|\leq q_0|Q_{1,\beta}|$, we define sets $\mathcal{A}$ and $E$ in the way as in the statement of Lemma \ref{covering}. By the construction of $\mathcal{A}$ and $E$ in Lemma \ref{covering}, it follows from Corollary \ref{corollary-1} that  
\[
E\subset \left\{Q_{1,\beta}:\ \M_{Q_{2\Lambda,\beta}}(|\nabla u|^2)> 1\right\}\cup\left\{Q_{1,\beta}:\ \M_{Q_{2\Lambda,\beta}}(|F|^2)> \hat \delta^2\right\}.
\]
Hence, \eqref{measure decay} follows from Lemma \ref{covering}. 

\smallskip
For $m \in \N$, the assertion \eqref{measure decay} can be proved by an induction argument using Corollary \ref{corollary-1} with the scaling of $u \mapsto u/K$ and $F \mapsto F/K$.
\end{proof}
\begin{proof}[Proof of Theorem \ref{inter-theorem}]Observe that when $p =2$, Theorem \ref{inter-theorem} follows from Lemma \ref{improved-Caccio-1} and the PDE \eqref{eqn-Q_R}. Hence, we only need to consider the case $p \in (2, \infty)$.  We begin the proof by taking $q_0 \in (0,1)$ be sufficiently small so that
\begin{equation}\label{choice of q_0}
q_0= \frac{1}{2K^{\frac{p}{2}}\gamma_1} <1,
\end{equation} 
where $K=K(n,\nu,M_0)>1$ is the number defined in Lemma \ref{measure density}, and $\gamma_1 = \gamma_1 (n, M_0)>0$ is the number  defined in Lemma \ref{covering}. Also, let 
\[
\delta_0=\hat \delta(n,\nu, M_0, q_0)\in (0,1),
\]
where $\hat \delta$ is defined in Lemma \ref{measure density}. We prove the theorem with the choice of $\delta_0$.

\smallskip
Let $S$ be the set defined in Lemma \ref{measure decay lemma}, that is,
\[
S=\left\{\, Q_{1,\beta} : \ \M_{Q_{2\Lambda,\beta}}(|\nabla u|^2)> K \,\right\},
\]
and we assume for the moment that
\begin{equation} \label{S-q-0930}
|S|\leq q_0|Q_{1,\beta}|.
\end{equation}
We denote
\[
h(s)=|\left\{z \in Q_{1,\beta}: \M_{Q_{2\Lambda,\beta}}(|\nabla u|^2)(z)> s\right\}| \quad \text{for}\quad s\in [0,\infty).
\]
Then, for $q = \frac{p}{2}>1$, and we note that
\[
\int_{Q_{1,\beta}}\big[|\M_{Q_{2\Lambda,\beta}}(|\nabla u|^2)\big]^{q}dxdt=q\int_{0}^{\infty}s^{q-1}h(s)ds.
\]
In addition, as $h(s)$ is a non-negative and non-increasing function of $s$, it follows that
\begin{align*}
\int_{Q_{1,\beta}}\big[|\M_{Q_{2\Lambda,\beta}}(|\nabla u|^2)\big]^{q}dxdt
&=q\int_{0}^{K}s^{q-1}h(s)ds+q\sum_{j=1}^{\infty}\int_{K^j}^{K^{j+1}}s^{q-1}h(s)ds\\
&\leq qh(0)\int_0^Ks^{q-1}ds+q\sum_{j=1}^{\infty}h(K^{j})\int_{K^j}^{K^{j+1}}s^{q-1}ds\\
&= K^{q}|Q_{1,\beta}|+(K^{q}-1)\sum_{j=1}^{\infty}K^{qj}h(K^j).
\end{align*}
On the other hand, as \eqref{S-q-0930} holds,  it follows from Lemma \ref{measure decay lemma} that
\begin{align*}
h(K^j) & \leq l_0^{j}h(1)+\sum_{i=1}^{j}l_0^{i}|\{Q_{1,\beta}: \M_{Q_{2\Lambda,\beta}}(|F|^2)> K^{j-i}\hat \delta^2\}|.
\end{align*}
Then, by combining the last two estimates, we obtain
\begin{align} \notag
\int_{Q_{1,\beta}}\big[|\M_{Q_{2\Lambda,\beta}}(|\nabla u|^2)\big]^{q}dxdt&\leq 
 K^{q}|Q_{1,\beta}|+(K^{q}-1)h(1)\sum_{j=1}^{\infty}K^{qj}l_0^{j}\\ \label{M-L-p}
&\quad+(K^{q}-1)\sum_{j=1}^{\infty}\sum_{i=1}^{j}K^{qj}l_0^{i}|\{Q_{1,\beta}: \M_{Q_{2\Lambda,\beta}}(|F|^2)> K^{j-i}\hat\delta^2\}|.
\end{align}
Note that from \eqref{choice of q_0} and as $l_0 =\gamma_1 q_0$, we have $\sum_{j=1}^{\infty}K^{qj}l_0^{j} \leq 1$.  Then, it follows that
\begin{equation}\label{M-L-p-10-1}
\begin{aligned}
&\int_{Q_{1,\beta}}\big[|\M_{Q_{2\Lambda,\beta}}(|\nabla u|^2)\big]^{q}dxdt\\
&\leq  N |Q_{1,\beta}| +(K^{q}-1)\sum_{j=1}^{\infty}\sum_{i=1}^{j}K^{qj}l_0^{i}|\{Q_{1,\beta}: \M_{Q_{2\Lambda,\beta}}(|F|^2)> K^{j-i}\hat\delta^2\}|,
\end{aligned}
\end{equation}
for $N = N(n, \nu, M_0, p)>0$.  We now control the last term in the right-hand side of \eqref{M-L-p-10-1}. To this end, we observe that
\begin{align} \notag
 \int_{Q_{2\Lambda, \beta}}\M_{Q_{2\Lambda,\beta}}(|F|^2)(x,t)^q  dxdt &= q\int_0^\infty \lambda^{q -1} |\{Q_{2\Lambda,\beta}: \M_{Q_{2\Lambda,\beta}}(|F|^2) > \lambda \}| d\lambda \\ \notag
& \geq q \sum_{j=0}^\infty \int_{\hat{\delta}^{2}K^{j-1}}^{\hat{\delta}^{2}K^{j}}\lambda^{q -1} |\{Q_{2\Lambda,\beta}: \M_{Q_{2\Lambda,\beta}}(|F|^2) > \lambda \}| d\lambda \\ \label{integ-10-1}
& \geq \sum_{j=0}^\infty (\hat{\delta}^2 K^{j-1})^{q} \big(K^{q} -1 \big)|\{Q_{2\Lambda,\beta}: \M_{Q_{2\Lambda,\beta}}(|F|^2) > \hat{\delta}^{2}K^{j} \}|.
\end{align}
Then, by using Fubini's theorem, we see that
\begin{align*}
&\big(K^{q} -1 \big) \sum_{j=1}^{\infty}\sum_{i=1}^{j}K^{qj}l_0^{i}|\{Q_{1,\beta}: \M_{Q_{2\Lambda,\beta}}(|F|^2)> K^{j-i}\hat \delta^2\}|\\
& = (K \hat{\delta}^{-2})^q\sum_{i=1}^{\infty} (K^ql_0)^{i}\Big[\sum_{k=i}^{\infty}\big(K^{q} -1 \big) \hat{\delta}^{2q} K^{q(k-i-1)}  |\{Q_{1,\beta}: \M_{Q_{2\Lambda,\beta}}(|F|^2)> K^{k-i}\hat \delta^2\}| \Big]\\
& = (K \hat{\delta}^{-2})^q\sum_{i=1}^{\infty} (K^ql_0)^{i}\Big[\sum_{j=0}^{\infty}\big(K^{q} -1 \big) (\hat{\delta}^{2} K^{(j-1)})^q |\{Q_{1,\beta}: \M_{Q_{2\Lambda,\beta}}(|F|^2)> K^{j}\hat \delta^2\}| \Big]\\
& \leq (K \hat{\delta}^{-2})^q  \Big( \int_{Q_{2\Lambda, \beta}}\M_{Q_{2\Lambda,\beta}}(|F|^2)(x,t)^q dxdt \Big) \sum_{i=1}^{\infty} (K^ql_0)^{i} \\
& \leq (K \hat{\delta}^{-2})^q  \Big( \int_{Q_{2\Lambda, \beta}}\M_{Q_{2\Lambda,\beta}}(|F|^2)(x,t)^q dxdt \Big),
\end{align*}
where we used \eqref{integ-10-1} and \eqref{choice of q_0} in the last two steps. From the last estimate, \eqref{M-L-p-10-1}, and the ($q-q$)-strong estimate of the Hardy-Littlewood maximal function, we infer that
\begin{equation*}\label{max L-p}
\begin{aligned}
\int_{Q_{1,\beta}}\big[|\M_{Q_{2\Lambda,\beta}}(|\nabla u|^2)\big]^{q}dxdt
&\leq  N\left(|Q_{1,\beta}|+\|F\|_{L^p(Q_{2\Lambda,\beta})}^p\right),
\end{aligned}
\end{equation*}
where $N=N(n,\nu, p, M_0)>0$.
On the other hand, by the Lebesgue differentiation theorem, for a.e. $(x,t)\in Q_{1,\beta}$,
\[
|\nabla u(x,t)|^2=\lim_{\rho\rightarrow 0^+}\fint_{C_{\rho,\beta}(x,t)}|\nabla u(y,s)|^2dyds\leq \M_{Q_{2\Lambda,\beta}}(|\nabla u|^2)(x,t).
\]
Hence, it follows that
\begin{equation} \label{est-1-0930}
\int_{Q_{1,\beta}}|\nabla u(x,t)|^pdxdt\leq N\left(|Q_{1,\beta}|+\|F\|_{L^p(Q_{2\Lambda,\beta})}^p\right).
\end{equation}

\smallskip
Note that so far, we proved \eqref{est-1-0930} under the assumption  \eqref{S-q-0930}. We now remove  \eqref{S-q-0930}. Let us denote
\begin{equation}\label{scaling-N_0}
N_0 =\left(\frac{\tilde N \| \nabla u\|_{L^2(Q_{1,\beta})}^2}{Kq_0|Q_{1,\beta}|}\right)^{\frac{1}{2}}=\left(\frac{\tilde N \fint_{Q_{1,\beta}}|\nabla u|^2dxdt}{Kq_0}\right)^{\frac{1}{2}},
\end{equation}
where $\tilde N = \tilde N(n, M_0)>0$ is the constant in the ($1-1$)-weak estimate of the Hardy-Littlewood maximal function. Then, from the weak type estimates for the maximal operator and \eqref{scaling-N_0}, it follows that
\[
|\left\{Q_{1,\beta}:\ \M_{Q_{2\Lambda,\beta}}(|\nabla u|^2)> N_0^2K\right\}|\leq \frac{\tilde N \| \nabla u\|_{L^2(Q_{1,\beta})}^2}{ N_0^2K} = q_0|Q_{1,\beta}|.
\]
Hence, by dividing the equation \eqref{eqn-Q_R} by $N_0$, we apply \eqref{est-1-0930} for $u/N_0$ with $F$ replaced by $F/N_0$, and then scale back to obtain
\[
\int_{Q_{1,\beta}}|\nabla u|^pdxdt\leq N\left(|Q_{1,\beta}|^{1-q}\| \nabla u\|_{L^2(Q_{1,\beta})}^p+\|F\|_{L^p(Q_{2\Lambda,\beta})}^p\right).
\]
Moreover, using \eqref{inter-energy-0317} in Lemma \ref{improved-Caccio-1}, we have
\begin{equation}\label{gradient-conclu}
\| \nabla u\|_{L^p({Q_{1,\beta}})}\leq N\left(|Q_{1,\beta}|^{\frac{1}{p}-\frac{1}{2}}\|u\|_{L^2(Q_{2\Lambda,\beta})}+\|F\|_{L^p(Q_{2\Lambda,\beta})}\right).
\end{equation}
From this, and the PDE \eqref{eqn-Q_R}, we have
\begin{align*}
\|\beta u_t\|_{L^p(\Gamma_{\beta}(1),W^{-1,p}(B_1))}
&\leq N \left(\|\nabla u\|_{L^p(Q_{1,\beta})}+\|F\|_{L^p(Q_{1,\beta})}\right)\\
&\leq N\left(|Q_{1,\beta}|^{\frac{1}{p}-\frac{1}{2}}\|u\|_{L^2(Q_{2\Lambda,\beta})}+\|F\|_{L^p(Q_{2\Lambda,\beta})}\right).
\end{align*}
The assertion of the theorem then follows from this last estimate and \eqref{gradient-conclu}. The proof of the theorem is completed. 
\end{proof}
\section{Boundary regularity estimates on flat domains} \label{bdr-section}
In this section, we study the following class of equations with homogeneous boundary conditions on the flat boundary part:
\begin{equation} \label{eqn-bdr-10-4}
\left\{
\begin{aligned}
\beta(x) u_t - \textup{div}(\bA (x,t) \nabla u) &= \textup{div}(F(x,t)) \quad &&\text{in} \quad Q_{2\Lambda, \beta}^+, \\[4pt]
u &=  0 \quad &&\text{on}\quad  T_{2\Lambda}\times \Gamma_{\beta}(2\Lambda).
\end{aligned} \right.
\end{equation}
Here, we refer the reader to the definitions of $Q_{2\Lambda, \beta}^+$ in \eqref{half cylinder} and $T_{2\Lambda}$ in \eqref{flat-bdr}, with $\Lambda$ the number defined in Lemma \ref{quasi-metric lemma}. We denote the mean of the matrix $\bA$ with respect to the spatial variable in the upper-half ball $B_r^+(x_0)$ by
\[
(\bA)_{B^+_{r}(x_0)}(t)= \fint_{B^+_{r}(x_0)} \bA(x, t) dx.
\]
As in \eqref{point-oss-def}, we also denote
\begin{equation} \label{Theta-A-def+}
\Theta_{\bA, r, x_0} (x,t) = |\bA(x,t) - (\bA)_{B_{r}^+(x_0)}(t)| \quad \text{and}\quad 
\Theta_{\bA, r} (x,t)= |\bA(x,t) - (\bA)_{B_{r}^+}(t)|.
\end{equation}
Moreover,
\begin{equation} \label{Theta-beta-def+}
\Theta_{\beta, r, x_0}(x) = |\beta(x) - (\beta)_{B_r(x_0)}|\beta(x)^{-1/2} \quad \text{and} \quad
\Theta_{\beta, r}(x) = |\beta(x) - (\beta)_{B_r}|\beta(x)^{-1/2}.
\end{equation}
We note that the space $\hW^{1,p}(Q_{1,\beta}^+, \beta)$ is defined in Definition \ref{hat-W-space}, and the weak solutions to the class of equations \eqref{eqn-bdr-10-4} are defined in  Definition \ref{boundary-weak-def}. The following result on boundary regularity estimate is the main result of this section.
\begin{theorem} \label{inter-theorem-bdry} Let $\nu \in (0,1)$, $M_0 \geq 1$, and $p \in [2,\infty)$. There exist a sufficiently small constant $\kappa_0= \kappa_0(n, \nu, p, M_0)>0$ and a constant $N=N(n, \nu, p, M_0)>0$ such that the following assertion holds. Suppose that $\beta$ is a weight satisfying \eqref{beta-cond}, $\bA$ is a matrix satisfying \eqref{ellip-cond} in $Q_{2\Lambda,\beta}^+$, and
\[
\fint_{Q_{r, \beta}^+(z_0)} \Theta_{\bA, r, x_0} (x,t)^2 dxdt+ \frac{1}{\beta(B_r(x_0))}\int_{B_r(x_0)} \Theta_{\beta, r, x_0}(x)^2  dx \leq \kappa_0^2
\]
for all $z_0 = (x_0, t_0) \in Q_{1, \beta}^+$ and all $r \in (0, 1)$. Then, for every weak solution $u \in \hW^{1,2}(Q_{2\Lambda, \beta}^+, \beta)$ of  \eqref{eqn-bdr-10-4} 
 with $F \in L^p(Q_{2\Lambda,\beta}^+)^n$, it holds that $u \in \hW^{1,p}(Q_{1,\beta}^+, \beta)$ and
\begin{align*}
&\| \nabla u\|_{L^p(Q_{1,\beta}^+)}+\|\beta u_t\|_{L^p(\Gamma_{\beta}(1),W^{-1,p}(B_1^+))}\\
&\leq N\left(|Q_{1,\beta}^+|^{\frac{1}{p}-\frac{1}{2}}\|u\|_{L^2(Q_{2\Lambda,\beta}^+)}+\|F\|_{L^p(Q_{2\Lambda,\beta}^+)}\right).
\end{align*}
\end{theorem}

\subsection{Boundary energy estimates} We study the following parabolic problem
\begin{equation} \label{eqn-bdry}
\left\{
\begin{aligned}
\beta(x) u_t - \textup{div}(\bA(x,t) \nabla u)  &= \textup{div}(F(x,t))\quad &&\text{in}\quad Q^+_{r,\beta}(z_0),\\
u&=0 \quad &&\text{on} \quad T_{r}(x_0) \times \Gamma_{\beta, z_0}(r),
\end{aligned}\right.
\end{equation}
with $z_0 = (x_0, t_0) \in \mathbb{R}^n \times \mathbb{R}$ and $r>0$ so that $T_r(x_0) \not= \emptyset$. For the definition of weak solutions $u \in \hat{\W}^{1,2}(Q^+_{r, \beta}(z_0), \beta)$ to \eqref{eqn-bdry}, see Definition \ref{boundary-weak-def}. By applying the same argument as in the proof of Lemma \ref{Caccio-1}, we obtain the following lemma on Caccioppoli type estimates for weak solutions to \eqref{eqn-bdry}.

\begin{lemma} \label{Caccio-2}  
For every $\nu \in (0,1)$ and $M_0 \geq 1$, there exists a constant $N = N(n, \nu, M_0)>0$ such that the following holds. Let $z_0 = (x_0,t_0) \in \R^{n}\times \R$ and $r>0$. Suppose that $u \in \hat{\W}^{1,2}(Q^+_{r, \beta}(z_0), \beta)$ is a weak solution of \eqref{eqn-bdry} in $Q^+_{r, \beta}(z_0)$, where $\beta$ is weight satisfying \eqref{beta-cond}, $\bA$ is a matrix satisfying \eqref{ellip-cond} in $Q^+_{r, \beta}(z_0)$, and $F \in L^2(Q^+_{r, \beta}(z_0))^n$. Then
\begin{align*}
&\sup_{t \in \Gamma_{\beta, z_0}(3r/4)} \int_{B^+_{3r/4}(x_0)} u(x,t)^2 \beta(x) dx + \int_{Q^+_{3r/4, \beta}(z_0)} |\nabla u(x,t)|^2 dx dt\\
&\quad \leq N\left\{\int_{Q^+_{r, \beta}(z_0)} u(x,t)^2\Big [1+ \frac{1}{r^2}  + \frac{\beta(x)}{r^2\Psi_{\beta,x_0}(r)}  \Big] dx dt+ \int_{Q^+_{r, \beta}(z_0)}|F|^2  dx dt\right\}.
\end{align*}
\end{lemma}

\smallskip
From Lemma \ref{Caccio-2}, and the argument as in the proof of Lemma \ref{improved-Caccio-1}, we also have the following result.
\begin{lemma} \label{bdr-improved-Caccio-1}Let $\nu \in (0,1)$ and $M_0 \geq 1$. Assume that $\beta$ is a weight satisfying \eqref{beta-cond},  and that $\bA$ is a matrix satisfying \eqref{ellip-cond} in $Q_{r, \beta}^+(z_0)$ with $r \in (0,1]$ and $z_0 = (x_0, t_0) \in \R^n \times \R$. Then, for every weak solution $u \in \hat{\W}^{1,2}(Q_{r, \beta}^+(z_0), \beta)$ of \eqref{eqn-bdry} with $F \in L^2(Q_{r, \beta}^+(z_0))^n$, it holds that
\begin{align*} \notag
& \sup_{t \in \Gamma_{\beta, z_0}(3r/4)}\frac{1}{r^2\Psi_{\beta, x_0}(r)}\fint_{B_{3r/4}^+(x_0)} u(x,t)^2 \beta(x) dx +  \fint_{Q_{3r/4, \beta}^+(z_0)} |\nabla u|^2dxdt\\ \label{inter-energy-0317-bdry}
& \leq N\fint_{Q_{r,\beta}^+(z_0)} \big( u^2/r^2 + |F|^2\big)dx dt,
\end{align*}
where $N = N(n, \nu, M_0)>0$.
\end{lemma}
\smallskip
Next, we introduce the following result, which is the boundary version of Lemma \ref{u-L2-Cac}.
\begin{lemma} \label{bdry-Poincare} 
Let $\nu\in (0,1)$ and $M_0 \geq 1$. There exists a constant $N = N(n, \nu, M_0)>0$ such that the following assertion holds. Suppose that $\beta$ is a weight satisfying \eqref{beta-cond}, and that $\bA$ satisfies \eqref{ellip-cond} in $Q^+_{r, \beta}(z_0)$ with some $r>0$. Suppose also that $u \in \hat{\W}^{1,2}(Q^+_{r, \beta}, \beta)$ is a weak solution of \eqref{eqn-bdry} in $Q^+_{r, \beta}$. Then
\begin{align*}
 \int_{Q^+_{r, \beta}} |u(x,t) |^2 dxdt  \leq N &\left[r^2\int_{Q^+_{r, \beta}}\Big[|\nabla u(x,t)|^2 + |F(x,t)|^2\Big] dxdt \right. \\
& +\left.  \left(\frac{1}{\beta(B_{r})}\int_{B_{r}} \Theta_{\beta, r}(x)^2 dx\right) \left(\int_{Q^+_{r, \beta}}|u(x,t) |^2 dxdt\right) \right],
\end{align*}
where $\Theta_{\beta, r}(x)$ is defined in \eqref{Theta-beta-def+}.
\end{lemma}
\begin{proof} The proof is similar to that of Lemma \ref{u-L2-Cac}, and we just outline some important steps. By applying the dilation \eqref{scale-1} and using Lemma \ref{scaling-lemma}, it suffices to prove the lemma under the assumption that
\[
r =1 \quad \text{and} \quad \Psi_{\beta}(1) =1.
\]
Note that in this setting, we have $Q_1^+=Q_{1,\beta}^+$. We prove the lemma by a contradiction argument. Suppose that the assertion does not hold. Then, there exist sequences $\{\beta_k\}_k$, $\{u_k\}_k$, $\{\bA_k\}_k$ and $\{F_k\}_k$ such that for each $k\in\N$,\, $u_k \in \hat\W^{1,2}(Q^+_1, \beta_k)$ is a weak solution of
\begin{equation} \label{eqn-u-k-bdry}
\beta_k(x) \partial_t u_k - \textup{div}(\bA_k \nabla u_k) = \textup{div}(F_k) \quad \text{in} \quad Q^+_1,
\end{equation}
where the weight $\beta_k$ satisfying $\Psi_{\beta_k}(1)=1$, and
\[
\beta_k^{-1} \in A_{1+\frac{2}{n_0}} \quad \text{and} \quad [\beta_k^{-1}]_{A_{1+\frac{2}{n_0}}} \leq M_0  \quad \text{for} \quad n_0 =\max\{n, 2\}.
\]
Moreover, $u_k$ satisfies
\begin{equation*}\label{bdry-neg}
\begin{aligned}
\int_{Q^+_{1}} |u_k|^2 dxdt 
\geq k &\left[ \int_{Q^+_{1}}\big[|\nabla u_k|^2 + |F_k|^2\big] dxdt \right. \\
&\ + \left. \Big(\frac{1}{\beta_k(B_1)}\int_{B_{1}} \Theta_{\beta_k, 1}(x)^2 dx\Big) \Big(\int_{Q^+_{1}} |u_k|^2 dxdt \Big) \right].
\end{aligned}
\end{equation*}
Now, by using the scaling $u_k \mapsto u_k/\lambda$ and dividing the PDE by $\lambda = \|u_k\|_{L^2(Q_{1}^+)}$, we can normalize and assume that
\begin{equation} \label{normalize-uk-bdry}
\int_{Q^+_{1}}|u_k |^2 dxdt =1, \quad \forall \ k \in \mathbb{N}.
\end{equation}
Then, for each $k \in \mathbb{N}$,
\begin{equation} \label{uF-k-beta-bdry}
\|\nabla u_k\|_{L^2(Q_1^+)}^2 + \|F_k\|_{L^2(Q_1^+)}^2+ \frac{1}{\beta_k(B_1)}\|\Theta_{\beta_k, 1}\|_{L^2(B_1)}^2 
\leq \frac{1}{k}.
\end{equation}

\smallskip
Now, from \eqref{normalize-uk-bdry}, \eqref{uF-k-beta-bdry}, and the PDE \eqref{eqn-u-k-bdry}, we can find a constant $N = N(n, \nu)>0$ so that for each $k \in \mathbb{N}$,
\begin{align*}
\|u_k\|_{\hat\W^{1,2}(Q^+_1, \beta_k)}
\leq N \big[\|u_k\|_{L^{2}(Q^+_1)}+ \|\nabla u_k\|_{L^{2}(Q^+_1)} + \|F_k\|_{L^{2}(Q^+_1)} \big]\leq 3N.
\end{align*}
Note that the condition \eqref{osc-assumption} in Proposition \ref{compactness-lemma} is also satisfied due to \eqref{uF-k-beta-bdry}, it then follows from Proposition \ref{compactness-lemma} that there exist a function $u_0 \in \W^{1,2}(Q_1^+)$ and a subsequence of $\{u_k \}_k$, which we still denote by $\{u_k\}_k$, so that
\begin{equation}\label{conv-lemma-ener-bdry}
\begin{aligned}
u_k  &\rightarrow u_0 \ &&\text{in} \quad L^2((-1,0),L^{l}(B_1^+)),\\ 
\nabla u_k  &\rightharpoonup \nabla u_0 \ &&\text{in} \quad L^2(Q_1^+), 
\end{aligned} \quad \quad \text{as}\quad k \rightarrow \infty,
\end{equation}
for every $l \in [1, 2^*)$. Due to \eqref{normalize-uk-bdry}, \eqref{uF-k-beta-bdry} and \eqref{conv-lemma-ener-bdry}, it follows that
\begin{equation} \label{u-zero-proper-bdry}
\nabla u_0 =0 \quad \text{and} \quad \int_{Q_1} |u_0|^2 dxdt =1.
\end{equation}
From this and as $u_0=0$ on $T_{1}\times (-1,0)$, it follows from the first equality of \eqref{u-zero-proper-bdry} that  $u_0=0$ in $Q^+_1$. However, this contradicts the second equality of \eqref{u-zero-proper-bdry}. Therefore, the lemma is proved.
\end{proof}

\smallskip
Now, following the proof of Lemma \ref{boundedness-u-sol}, we can apply Lemma \ref{bdry-Poincare} to derive the following result for weak solutions of \eqref{eqn-bdry} in $Q^+_{1, \beta}$.
\begin{lemma}\label{bdry-corollary} 
For every $\nu \in (0,1)$ and $M_0 \geq 1$, there exist a sufficiently small constant $\tilde \kappa = \tilde \kappa(n, \nu, M_0)>0$ and a constant $N = N(n, \nu, M_0)>0$ such that the following assertion holds. Suppose that $\beta$ is a weight satisfying \eqref{beta-cond}, $\bA$ satisfies \eqref{ellip-cond} in $Q^+_{1, \beta}$, and
\[
\frac{1}{\beta(B_1)} \int_{B_1} \Theta_{\beta, 1}(x)^2dx  \leq  {\tilde \kappa}^2.
\]
Then, for every weak solution $u \in \hat{\W}^{1,2}(Q^+_{1, \beta})$ of \eqref{eqn-bdry} in $Q^+_{1, \beta}$, it holds that
\begin{align*}
\|u \|_{\hat\W^{1,2}(Q^+_{1, \beta}, \beta)} & \leq N \Big[\|\nabla u\|_{L^2(Q^+_{1, \beta})} + \|F\|_{L^2(Q^+_{1,\beta})}\Big].
\end{align*}
\end{lemma}
\smallskip
\subsection{Boundary approximation estimates} To prove Theorem \ref{inter-theorem-bdry}, we use a perturbation technique by freezing the coefficients.  Hence, as in \eqref{v-Qr-sol}, we introduce the following equations
\begin{equation} \label{bdry-approx}
\left\{
\begin{aligned}
(\beta)_{B_r(x_0)} v_t - \textup{div} ((\bA)_{B^+_{r}(x_0)}(t) \nabla v)  &= 0 \quad &&\text{in}\quad Q^+_{r,\beta}(z_0),\\[4pt]
v&=0 \quad &&\text{on} \quad T_{r}(x_0) \times \Gamma_{\beta, z_0}(r)
\end{aligned}\right.
\end{equation}
with $r>0$ and $z_0 = (x_0, t_0)\in \mathbb{R}^{n}\times \mathbb{R}$ satisfying $T_{r}(x_0) \not=\emptyset$.  Similar to Lemma \ref{compare-lemma-1}, we have the following important lemma on boundary Caccioppoli type estimates for the difference of two solutions of \eqref{eqn-bdry} and \eqref{bdry-approx}.
\begin{lemma} \label{bdry-Caccioppoli} 
Suppose that $u \in \hat\W^{1,2}(Q^+_{r,\beta}(z_0), \beta)$ is a weak solution of \eqref{eqn-bdry} in $Q^+_{r, \beta}(z_0)$, where the weight $\beta$ satisfies \eqref{beta-cond} and the matrix $\bA$ satisfies \eqref{ellip-cond} in  $Q^+_{r,\beta}(z_0)$. Suppose also that  $v \in \hat\W^{1,2}(Q^+_{r,\beta}(z_0))$ is a weak solution of \eqref{bdry-approx} in $Q^+_{r, \beta}(z_0)$. Then, for 
\[
w = u -v,
\] 
it holds that
\begin{align*}
& \fint_{Q^+_{r, \beta}(x_0)} |\nabla w|^2 \varphi^2 dx dt \leq   N \fint_{Q^+_{r, \beta}(x_0)} |F|^2 \varphi^2 dx dt \\
& \quad + N\fint_{Q^+_{r, \beta}(z_0)} w^2 \Big[\varphi^2 +|\nabla \varphi|^2 + \beta(x)\Big(\frac{\varphi^2}{r^2(\beta)_{B_r(x_0)}} + |\partial_t \varphi| \Big)\Big]dxdt \\
& \quad +   N \|\varphi \nabla v\|_{L^\infty(Q^+_{r, \beta}(z_0))}^2 \fint_{Q^+_{r, \beta}(z_0)} \Theta_{\bA, r, x_0}(x,t)^2 dx dt \\
& \quad +N\|r(\beta)_{B_{r}(x_0)} v_t \varphi\|_{L^\infty(Q^+_{r, \beta}(z_0))}^2  \Big( \frac{1}{\beta(B_r(x_0))} \int_{B_{r}(x_0)} \Theta_{\beta, r, x_0}(x)^2 dx\Big),
\end{align*}
for every $\varphi  \in C^\infty_0(Q_{r, \beta}(z_0))$, where $N = N(n, \nu)>0$.
\end{lemma}

\smallskip
\begin{proof} To begin with, it follows from Lemma \ref{Bdr-Lipschitz-constant} that $v_t, \nabla v  \in L^{\infty}_{\text{loc}}(Q^+_{r,\beta}(z_0))$. Consequently, if $u \in \hat\W^{1,2}(Q^+_{r, \beta}(z_0), \beta)$ is a solution to \eqref{eqn-bdry} in $Q^+_{r, \beta}(z_0)$, then 
\[
w = u - v \in  \hat\W^{1,2}_{\text{loc}}(Q^+_{r, \beta}(z_0), \beta)
\]
is a weak solution of
\begin{align*}
\beta(x) w_t  - \textup{div}( \bA \nabla w) & = \textup{div}(F)- (\beta(x) - (\beta)_{B_r(x_0)}) v_t \\
& \quad + \textup{div}[(\bA - (\bA)_{B^+_{r}(x_0)}(t))\nabla v] \quad \text{in} \quad Q^+_{r,\beta}(z_0)
\end{align*}
with the boundary condition $w=0$ on $T_r(x_0)\times \Gamma_{\beta,z_0}(r)$. Then, by using Steklov's average if needed (see \cite[p. 18]{DiB} for example), we can test the equation of $w$ with $w\varphi^2$ for every $\varphi  \in C^\infty_0(Q_{r, \beta}(z_0))$. As the proof is similar to that of Lemma \ref{compare-lemma-1}, we skip the details.
\end{proof}

\smallskip
The following lemma provides the boundary version of  Lemma \ref{L2-comparision}  on the closeness of solutions of the two PDEs \eqref{eqn-bdry} and \eqref{bdry-approx}. 
\begin{lemma} \label{L2-comparision-bdry} Let $\nu \in (0,1)$ and $M_0 \geq 1$ be fixed. Then, for every $\epsilon \in (0, 1)$, there exists a sufficiently small constant $\bar\kappa = \bar \kappa(n, \nu, M_0, \epsilon)$  such that the following assertions hold. Suppose that $\beta$ is a weight satisfying \eqref{beta-cond}, $\bA$ is a matrix satisfying \eqref{ellip-cond} in $Q^+_{4, \beta}$, and
\begin{align*}
& \fint_{Q^+_{4,\beta}}\Theta_{\bA, 4}(x,t)^2 dxdt  + \frac{1}{\beta(B_4)} \int_{B_4} \Theta_{\beta, 4}(x)^2 dx  + \fint_{Q^+_{4,\beta}}  |F(x,t)|^2 dx dt \leq \bar \kappa^2.
\end{align*}
Then, for every weak solution $u \in \hat\W^{1,2}(Q^+_{4,\beta},\beta)$ of \eqref{eqn-bdry} in $Q^+_{4,\beta}$ satisfying
\[
\fint_{Q^+_{4,\beta}} |\nabla u(x,t)|^2dxdt \leq 1,
\]
there exists a weak solution $v \in \hat\W^{1,2}(Q^+_{4,\beta})$ of \eqref{bdry-approx} in $Q^+_{4,\beta}$ such that
\begin{equation} \label{L-2-w-small-bdry}
\left(\fint_{Q^+_{4, \beta}} |u-v|^2 dx dt\right)^{1/2}  \leq \epsilon.
\end{equation}
Moreover, there exist constants $N = N(n, \nu, M_0)>0$ and $\theta = \theta(n, M_0) \in (0,1)$ such that
\begin{equation} \label{L-2-mu-bdry}
\begin{split}
& \fint_{Q^+_{7/2,\beta}} |\nabla v|^2 dx dt \leq N, \quad \text{and} \\
& \frac{1}{(\beta)_{B_{7/2}}} \fint_{Q^+_{7/2, \beta}} |u-v|^2 \beta(x) dx dt \leq N \epsilon^{2(1-\theta)}.
\end{split} 
\end{equation}
\end{lemma}
\begin{proof}  It is sufficient to prove the lemma for $\bar \kappa \in (0, \tilde \kappa]$, where $\tilde \kappa = \tilde \kappa(n, \nu, M_0)>0$ is the small constant defined in Lemma \ref{bdry-corollary}. By applying the dilation in the time variable as in \eqref{scale-2}, we may assume without loss of generality that
 $\Psi_{\beta}(4) =1$ so that $Q^+_{4,\beta}=Q^+_4$. To prove the assertion \eqref{L-2-w-small-bdry}, we use the same contradiction argument as in the proof of Lemma \ref{L2-comparision}.  However, instead of applying Lemma \ref{boundedness-u-sol}, we use Lemma \ref{bdry-corollary} to find a constant $N = N(n, \nu, M_0) >0$ such that
\[
\|u_k \|_{\hat\W^{1,2}(Q_4^+, \beta_k)} \leq N, \quad \text{for all $k \in \mathbb{N}$}. 
\]
From this, \eqref{beta-cond}, Proposition \ref{compactness-lemma}, and by passing a subsequence, we can find $u_0 \in \hat\W^{1,2}(Q^+_4)$ such that
\begin{equation} \label{uk-convergence-bdry}
\begin{split}
& u_k  \rightarrow u_0 \quad \text{in} \quad L^2((-16, 0), L^{l}(B^+_4)), \quad \text{and}\\
& \nabla u_k \rightharpoonup \nabla u_0  \quad \text{in} \quad  L^2(Q^+_4) \quad \text{as} \quad k \rightarrow \infty,
\end{split}
\end{equation}
for any $l \in [1, 2^*)$. 

\smallskip
Now, for a fixed $\varphi \in C_0^\infty(Q_4^+)$, and for every $k \in \mathbb{N}$, we use $\varphi$ as a test function for the equation of $u_k$ to obtain
\begin{equation*} 
 -\int_{Q^+_4} \beta_k u_k  \partial_t \varphi  dxdt +  \int_{Q^+_4} \wei{\bA_k \nabla u_k, \nabla \varphi} dxdt  = -\int_{Q^+_4} \wei{F_k, \nabla \varphi} dxdt.
\end{equation*}
By following the same procedure as in the proof of Lemma \ref{L2-comparision}, we can show that $u_0 \in \hat\W^{1,2}(Q^+_4)$ is a weak solution of 
\begin{equation} \label{u-0-eqn-bdry}
\left\{
\begin{aligned}
\beta_0\partial_t u_0 - \textup{div} (\bA_0(t) \nabla u_0) &= 0  \quad &&\text{in} \quad Q^+_4,\\
u_0 &=  0  \quad &&\text{on} \quad T_4 \times \Gamma_{\beta}(4),
\end{aligned} \right.
\end{equation}
where $\beta_0\in [\frac{1}{M_0},1]$, and the matrix $\bA_0(t)$ is defined as in \eqref{weak star convergence}.
Due to the convergences in \eqref{uk-convergence-bdry}, we have
\begin{equation} \label{L2-u_0-bdry}
\fint_{Q^+_4} |\nabla u_0|^2 dx dt \leq 1.
\end{equation}

\smallskip
Now, let us denote
\[
g_k = [(\beta_k)_{B_4}-\beta_0]\partial_t u_0 -\textup{div} ([(\bA_k)_{B_4^+} (t)-\bA_0(t)] \nabla u_0).
\]
From the triangle inequality, the PDE \eqref{u-0-eqn-bdry}, \eqref{L2-u_0-bdry}, the ellipticity and boundedness conditions of $\bA_k$ in \eqref{ellip-cond}, the definition of $\bA_0$ as in \eqref{weak star convergence}, and the boundedness of $\beta_0$, we infer that there is a constant $N = N(n, \nu, M_0)>0$ such that
\begin{equation}  \label{g-k-convergence-bdry}
\begin{split} 
& \|g_k \|_{L^2((-16, 0), W^{-1, 2}(B_4^+))} \leq N, \quad \forall \ k \in \mathbb{N}, \\ 
& g_k  \rightharpoonup  0 \quad \text{in} \quad L^2((-16, 0), W^{-1, 2}(B_4^+)) \quad \text{as} \quad  k \rightarrow \infty.
\end{split}
\end{equation}
Moreover, let $h_k \in \W^{1,2}_*(Q_4^+)$ be the weak solution of the equation
\begin{equation} \label{h-k-sol-bdry}
\left\{
\begin{aligned}
(\beta_k)_{B_4}\partial_t h_k -\textup{div}[(\bA_k)_{B_4^+}(t) \nabla h_k] & =g_k \quad &&\text{in} \quad Q_4^+, \\
h_k &= 0 \quad &&\text{on}  \quad \partial_p Q_4^+.
\end{aligned} \right.
\end{equation}
Note that the estimate of $g_k$ in \eqref{g-k-convergence-bdry}, the ellipticity and boundedness conditions \eqref{ellip-cond}, and the definitions of $(\bA_k)_{B_4^+}(t)$ and $(\beta_k)_{B_4}$ allow us to obtain the existence of $h_k$ by the Galerkin method. In addition, from the standard energy estimate for $h_k$, we infer that there is $N = N(n, \nu, M_0)>0$ such that
\[
\|h_k\|_{\W^{1,2}_*(Q_4^+)} \leq N, \quad \forall \ k \in \mathbb{N}.
\]
From this, we apply the classical Aubin-Lions theorem (see \cite[Theorem 1]{Simon}), and by passing through a subsequence if needed, we can find $h \in \W^{1,2}_*(Q_4)$ such that
\begin{equation*} 
   \left\{
\begin{aligned}
h_k \rightarrow h \quad &\text{in} \quad L^2((-16, 0), L^{2}(B_4^+)),\\
\nabla h_k  \rightharpoonup \nabla h  \quad & \text{in} \quad L^2((-16, 0), L^{2}(B_4^+)), \\
\partial_t h_k \rightharpoonup \partial_t h  \quad &\text{in} \quad    L^2((-16, 0), W^{-1,2}(B_4^+)),
\end{aligned} \quad \quad \text{as} \quad k \rightarrow \infty,
   \right.
\end{equation*}
From this, \eqref{beta_0}, \eqref{weak star convergence}, and the convergences of $g_k$ in \eqref{g-k-convergence-bdry}, we can use the weak form of \eqref{h-k-sol-bdry}  to pass the limit as $k\rightarrow \infty$ in the same way as the proof of \eqref{u-0-eqn-bdry} to conclude that $h \in \W^{1,2}_*(Q_4^+)$ is a weak solution of 
\begin{equation*} 
\left\{
\begin{aligned}
\beta_0 \partial_t h  -\textup{div}(\bA_0(t) \nabla h) & = 0 \quad &&\text{in} \quad Q_4^+, \\
h &= 0 \quad &&\text{on}  \quad \partial_p Q_4^+.
\end{aligned} \right.
\end{equation*}
This implies that $h =0$ and consequently
\begin{equation}\label{h-k-unweight-bdry}
\begin{aligned}
&\|h_k\|_{L^2(Q_4^+)} \rightarrow 0 \quad \text{as} \quad k \rightarrow \infty.
\end{aligned}
\end{equation}
Then, set $v_k = u_0 - h_k$, we see that $v_k \in \hat\W^{1,2}(Q_4^+)$ is a weak solution of 
\[
\left\{
\begin{array}{rcll}
(\beta_k)_{B_4}\partial_t v_k -\textup{div}((\bA_k)_{B_4^+}(t) \nabla v_k)  & = & 0 & \quad \text{in}\quad Q_4^+\\
v_k & = & 0 & \quad \text{on} \quad T_4 \times \Gamma_{\beta}(4).
\end{array} \right.
\]
However, with $w_k = u_k -v_k=u_k -u_0+h_k$, it follows from \eqref{uk-convergence-bdry} and \eqref{h-k-unweight-bdry} that
\begin{align*}
\| w_k\|_{L^2(Q_4^+)} \leq   \|u_k -u_0\|_{L^2(Q_4^+)} + \|h_k\|_{L^2(Q_4^+)}  \rightarrow 0 \quad \text{as} \quad k \rightarrow \infty.
\end{align*}
This contradicts the conclusion of the contrapositive statement as in \eqref{w-k-contra}, when $v_k$ is in place of $v$ with sufficiently large $k$. The assertion \eqref{L-2-w-small-bdry} is then proved.

\smallskip
Lastly, we prove \eqref{L-2-mu-bdry}. Note that as $\bar\kappa \leq \tilde \kappa$ with $\tilde \kappa$  the constant in Lemma \ref{bdry-corollary},  it follows from Lemma \ref{bdry-corollary} and the assumptions in the statement of the lemma that
\begin{align*}
\fint_{Q^+_{4}} |u(x,t)|^2 dx dt & \leq N \fint_{Q^+_{4}} \Big[ |\nabla u(x,t)|^2 + |F(x,t)|^2\Big] dxdt \\
& \leq N\big(1 + \bar \kappa^2) \leq 2N,
\end{align*}
where $N = N(n, \nu, M_0)>0$. From this, \eqref{L-2-w-small-bdry}, and the triangle inequality, we see that
\begin{align*}
\fint_{Q^+_{4}}|v(x,t)|^2 dx dt 
&\leq 2\fint_{Q^+_{4}} |w(x,t)|^2 dxdt + 2\fint_{Q^+_{4}} |u(x,t)|^2 dx dt\\
& \leq 2\epsilon^2 + 4N \leq  2+ 4N,
\end{align*}
for $w = u(x,t) -v(x,t)$. Here, we also used the fact that $\epsilon \in (0,1)$. From this, and by the standard energy estimate for the equation of $v$ and the doubling properties stated in Lemma \ref{property}, we obtain
\[
\fint_{Q^+_{7/2, \beta}} |\nabla v|^2 dx dt \leq N \fint_{Q^+_{4}} |v(x,t)|^2 dx dt \leq N.
\]
This proves the first assertion in \eqref{L-2-mu-bdry}. It now remains to prove the second assertion in  \eqref{L-2-mu-bdry}. Due to the assumption in the statement of the lemma and the doubling properties, we can apply Lemma \ref{embedd-lemma} to infer that 
\[
\begin{split}
& \frac{1}{(\beta)_{B_{7/2}}}\fint_{Q^+_{7/2,\beta}} |w(x,t)|^2 \beta(x)dx dt  \\
& \leq N  \Big(\fint_{Q^+_{7/2,\beta}} |w|^2 dxdt \Big)^{1-\theta} \left[ \Big(\fint_{Q^+_{7/2, \beta}} |w|^{2} dxdt \Big)^{\theta} +  \Big(\fint_{Q^+_{7/2,\beta}} |\nabla w|^{2} dxdt \Big)^{\theta} \right]\\
& \leq N \epsilon^{2(1-\theta)} \left[ \epsilon^{2\theta} +  \Big(\fint_{Q^+_{4}} |\nabla u(x,t)|^{2} dxdt \Big)^{\theta}  + \Big(\fint_{Q^+_{7/2,\beta}} |\nabla v(x,t)|^{2} dxdt \Big)^{\theta}\right] \\
& \leq N \epsilon^{2(1-\theta)},
\end{split}
\]
where $N = N(n, \nu, M_0)>0$. Hence, the second assertion of \eqref{L-2-mu-bdry} is proved, and the proof of the lemma is completed.
\end{proof}

Following the proof of Proposition \ref{L2-gradient-comparision}, and by applying Lemma \ref{L2-comparision-bdry} and Lemma \ref{bdry-Caccioppoli} in place of Lemma \ref{L2-comparision}
and Lemma \ref{compare-lemma-1}, respectively,  we obtain the following proposition for the boundary case.

\smallskip
\begin{proposition} \label{L2-gradient-comparision-bdry} Let $\nu\in (0,1)$ and $M_0 \geq 1$ be fixed. Then, for every $\epsilon \in (0,1)$, there exists a sufficiently small constant $\kappa = \kappa(n, \nu, M_0, \epsilon) \in (0,1)$  such that the following assertions hold. Suppose that $\beta$ is a weight satisfying \eqref{beta-cond}, and that $\bA$ is a matrix satisfying \eqref{ellip-cond} in $Q^+_{4, \beta}$. Moreover, assume that
\begin{align} \label{small-data-interior-bdry}
 \fint_{Q^+_{4,\beta}}\Theta_{\bA, 4}^2 dxdt  + \frac{1}{\beta(B_4)} \int_{B_4} \Theta_{\beta, 4}(x)^2 dx  + \fint_{Q^+_{4,\beta}}|F|^2 dx dt \leq \kappa^2.
\end{align}
Then, for every weak solution $u \in \hat\W^{1,2}(Q^+_{4,\beta}, \beta)$ of \eqref{eqn-bdry} in $Q^+_{4,\beta}$ satisfying
\[
\fint_{Q^+_{4,\beta}} |\nabla u(x,t)|^2dxdt \leq 1,
\]
there exists a weak solution $v \in \hat\W^{1,2}(Q^+_{4,\beta})$ of \eqref{bdry-approx} in $Q^+_{4,\beta}$ such that
\begin{equation} \label{L-2-gradient-w-small-bdry}
\left(\fint_{Q^+_{2, \beta}} |\nabla u-\nabla v|^2 dx dt\right)^{1/2}  \leq \epsilon \quad \text{and} \quad \|\nabla v\|_{L^\infty(Q^+_{3, \beta})}  \leq N,
\end{equation}
where $N = N(n, \nu, M_0)>0$.
\end{proposition}

\smallskip
\begin{proof} 
For the given $\epsilon$, let $\bar{\epsilon},\, \kappa'\in (0,1)$ be  sufficiently small numbers such that 
\[
\bar{\epsilon}^{2(1-\theta)} \leq \frac{\epsilon^2}{2N_0} \quad \text{and}\quad (\kappa')^2<\frac{\epsilon^2}{2N_0},
\]
where $\theta = \theta (n, M_0) \in (0,1)$ is defined in Lemma \ref{L2-comparision-bdry}, and $N_0 = N_0(n, \nu, M_0)$ is the positive constant defined in \eqref{last-interior-prop-bdry} below.  Then, let 
\[
\kappa = \min\{\tilde \kappa (n, \nu, M_0),\ \bar \kappa(n, \nu, M_0, \bar{\epsilon}),\ \kappa'\},
\]
where $\tilde \kappa (n, \nu, M_0)$ is the number defined in Lemma \ref{bdry-corollary}, $\bar \kappa(n, \nu, M_0, \bar{\epsilon})$ is the number defined in Lemma \ref{L2-comparision-bdry}. We prove the proposition with this choice of $\kappa$.

\smallskip

First, by \eqref{small-data-interior-bdry} and the choice of $\kappa$, it follows from Lemma \ref{L2-comparision-bdry} that
there is a weak solution $v \in \hat\W^{1,2}(Q^+_{4,\beta})$ of \eqref{bdry-approx} in $Q^+_{4,\beta}$ such that
\begin{equation} \label{L-2-w-small-bar-bdry}
\left(\fint_{Q^+_{4, \beta}} |u-v|^2 dx dt\right)^{1/2}  \leq \bar{\epsilon},
\end{equation}
and
\begin{equation} \label{L-2-mu-bar-bdry}
\begin{split}
& \fint_{Q^+_{7/2,\beta}} |\nabla v|^2 dx dt \leq N, \quad \text{and}\\
& \frac{1}{(\beta)_{B_{7/2}}} \fint_{Q^+_{7/2, \beta}} |u-v|^2 \beta(x) dx dt \leq N \bar{\epsilon}^{2(1-\theta)},
\end{split} 
\end{equation}
where $N = N(n, \nu, M_0)>0$ and $\theta = \theta(n, M_0) \in (0,1)$.  In addition, by applying Lemma \ref{Bdr-Lipschitz-constant} and the doubling properties of $\beta^{\frac{n}{2}}$ (when $n\geq 3$) or of $\beta$ (when $n=1,\, 2$), we obtain
\begin{equation} \label{L-infty-bar-bdry}
\begin{split}
& \|\nabla v\|_{L^\infty(Q^+_{3,\beta})} \leq N \left(\fint_{Q^+_{7/2,\beta}} |\nabla v|^2 dx dt \right)^{1/2} \leq N \quad \text{and} \\
& \|(\beta)_{B_{3}} v_t\|_{L^\infty(Q_{3,\beta})} \leq N \left(\fint_{Q^+_{7/2,\beta}} |\nabla v|^2 dx dt  \right)^{1/2}\leq N.
\end{split}
\end{equation}
On the other hand, let us denote $w = u-v$. By applying Lemma \ref{bdry-Caccioppoli} to $w$ with a suitable cut-off function $\varphi$, and the doubling properties, we have
\begin{align*}
 & \fint_{Q^+_{2,\beta}} |\nabla w(x,t)|^2 dx dt   \leq N  \fint_{Q^+_{4,\beta}} \Big[  |F(x,t)|^2 + |w(x,t)|^2 \Big] dxdt   \\
& \qquad + \frac{N}{(\beta)_{B_{7/2}}}  \fint_{Q^+_{7/2,\beta}} w(x,t)^2 \beta(x)dx dt +  N \|\nabla v\|_{L^\infty(Q^+_{3, \beta})}^2 \fint_{Q^+_{4,\beta}} \Theta_{\bA, 4}(x,t)^2 dx dt \\
& \qquad + N \| (\beta)_{B_3} v_t\|_{L^\infty(Q^+_{3, \beta})}^2 \fint_{Q^+_{4, \beta}}\Theta_{\beta, 4}(x)^2 dx.
\end{align*}
From this, \eqref{small-data-interior-bdry}, \eqref{L-2-w-small-bar-bdry}, \eqref{L-2-mu-bar-bdry},\eqref{L-infty-bar-bdry}, and the fact that $\bar{\epsilon} \in (0,1)$, we infer that
\begin{equation} \label{last-interior-prop-bdry}
\fint_{Q^+_{2,\beta}} |\nabla w(x,t)|^2 dx dt   \leq   N_0 \Big[ \kappa^2  + \bar{\epsilon}^{2(1-\theta)} \Big], \quad \text{with}\quad N_0 = N_0(n, \nu, M_0)>0.
\end{equation}
By the choices of $\kappa$ and $\bar{\epsilon}$, it follows that
\[
\fint_{Q^+_{2,\beta}} |\nabla w(x,t)|^2 dx dt \leq \epsilon^2.
\]
Then, \eqref{L-2-gradient-w-small-bdry} follows from this estimate and \eqref{L-infty-bar-bdry}. The proposition is proved.
\end{proof}
\subsection{Level set estimates} 
Since cylinders $Q^+_{r,\beta}(z_0)$ defined in \eqref{half cylinder} have the same heights as $Q_{r,\beta}(z_0)$, by a modification of the proof for Lemma \ref{measure density}, we have the following density estimates for the boundary cases.

\begin{lemma}\label{measure density-bdry}
Let $\nu\in (0,1)$ and $M_0\geq 1$. For any $q_0\in(0,1)$, there exist $K=K(n,\nu, M_0)>1$ and a sufficiently small number $\hat \kappa=\hat \kappa(n,\nu,M_0, q_0)\in (0,1)$ such that the following assertion holds. Suppose that $\beta$ satisfies \eqref{beta-cond}, $\bA$ is a matrix satisfying \eqref{ellip-cond} in $Q^+_{2\Lambda, \beta}$ with $\Lambda$  defined in Lemma \ref{quasi-metric lemma}. Moreover, assume that
\begin{equation*}
\begin{aligned} 
 \sup_{r\in (0,1)}\ \sup_{\substack{z=(x,t)\\ \in Q^+_{1,\beta}}}\left(\fint_{Q^+_{r,\beta}(z)}\Theta_{\bA, r, x}^2(y,s) dyds + \frac{1}{\beta(B_{r}(x))} \int_{B_{r}(x)} \Theta_{\beta, r, x}(y)^2 dy\right) \leq \hat \kappa^2,
\end{aligned}
\end{equation*}
and $u\in \hat\W^{1,2}(Q^+_{2\Lambda,\beta},\beta)$ is a weak solution  of \eqref{eqn-bdr-10-4} in $Q^+_{2\Lambda,\beta}$. Then, for any $C_{\rho,\beta}(z_0)$ with $\rho\in(0,\frac{1}{4})$ and $z_0\in Q^+_{1,\beta}$, if 
\begin{equation}\label{nonempty-bdry}
C_{\rho,\beta}(z_0)\cap\left\{Q^+_{1,\beta}: \M_{Q^+_{2\Lambda,\beta}}(|\nabla u|^2)\leq 1\right\}\cap \left\{\M_{Q^+_{2\Lambda,\beta}}(|F|^2)\leq \hat \kappa^2\right\}\neq \emptyset,
\end{equation}
we have
\begin{equation}\label{density-2}
|C_{\rho,\beta}(z_0)\cap\left\{Q^+_{1,\beta}: \M_{Q^+_{2\Lambda,\beta}}(|\nabla u|^2)> K\right\}|\leq q_0 |C_{\rho,\beta}(z_0)|.
\end{equation}
\end{lemma}

\smallskip
\begin{proof}

\smallskip
To begin with, we consider a special case of $C_{\rho,\beta}(\bar z_0)$ with $\bar z_0\in T_1\times \Gamma_{\beta}(1)$. Following a similar way as the proof of Lemma \ref{measure density}, we can show that there exist constants $K=K(n,\nu, M_0)>1$ and $\hat \kappa=
\hat\kappa(n,\nu,M_0, q_0)\in (0,1)$ small such that if \eqref{nonempty-bdry} holds for $C_{\rho,\beta}(\bar z_0)$ in place of $C_{\rho,\beta}(z_0)$, then
\begin{equation}\label{special claim}
|C_{\rho,\beta}(\bar z_0)\cap\left\{Q^+_{1,\beta}: \M_{Q^+_{2\Lambda,\beta}}(|\nabla u|^2)> K\right\}|\leq 3^{-2n-2}M_0^{-1}q_0 |C_{\rho,\beta}(\bar z_0)|.
\end{equation}
\smallskip
Based on \eqref{special claim}, we now prove \eqref{density-2} for a general $C_{\rho,\beta}(z_0)$ satisfying \eqref{nonempty-bdry} and with $z_0\in Q_{1,\beta}^+$. Let us denote 
\[
z_0=(x_0, t_0)\in \R^{n}\times \R \quad \text{with}\quad x_0=(x_0',x_0^n)\in \R^{n-1}\times (0,\infty).
\]
We consider two cases when $x_0^n\geq 2\rho$ and when $x_0^n < 2\rho$.

\smallskip \noindent
{\bf Case 1}: $x_0^n\geq 2\rho$. Since $\rho\in (0,\frac{1}{4})$ and by the triangle inequality, we have
\[
B_{2\rho}(x_0)\subset B_{2\Lambda}^+\quad \text{and}\quad Q_{2\rho,\beta}(z_0)\subset Q_{2\Lambda}^+.
\]
Then, we restrict the equation to $Q_{2\rho,\beta}(z_0)$. After taking a linear translation and scaling, we verify that the conditions of Lemma \ref{measure density} are satisfied. Thus, \eqref{density-2} follows from Lemma \ref{measure density} and its proof.

\smallskip \noindent
{\bf Case 2}: $x_0^n\leq 2\rho$. We can find a point $\bar z_0=(\bar x_0, t_0)\in T_1\times \Gamma_{\beta}(1)$ with $\bar x_0=(x_0',0)$ so that $B_{\rho}(x_0)\subset B_{3\rho}(\bar x_0)$. Due to this, the condition \eqref{nonempty-bdry} is also satisfied for $C_{3\rho,\beta}(\bar z_0)$ in place of $C_{\rho,\beta}(z_0)$. 
Moreover, by $\textup{(iii)}$ of Lemma \ref{property}, we note that  
\begin{equation}\label{cylinders  measure control}
|C_{3\rho,\beta}(\bar z_0)|=3^n(1-A)^{-\frac{2}{n}}|C_{\rho,\beta}(z_0)|
\leq 3^{2n+2}M_0|C_{\rho,\beta}(z_0)|,
\end{equation}
where
\[
A=\left\{
   \begin{aligned}
\frac{\beta^{\frac{n}{2}}\big(B_{3\rho}(\bar x_0)\setminus B_{\rho}(x_0)\big)}{\beta^{\frac{n}{2}}\big(B_{3\rho}(\bar x_0)\big)}&\leq 1-3^{-n-\frac{n^2}{2}}M_0^{-\frac{n}{2}}\quad &&\text{for}\quad n\geq 3,\\
\frac{\beta\big(B_{3\rho}(\bar x_0)\setminus B_{\rho}(x_0)\big)}{\beta\big(B_{3\rho}(\bar x_0)\big)}&\leq 1-3^{-n-\frac{n^2}{2}}M_0^{-\frac{n}{2}}\quad &&\text{for}\quad n=1,\, 2.
   \end{aligned}\right.
\]
Thus, \eqref{density-2} follows from \eqref{special claim} and \eqref{cylinders  measure control}, as
\begin{align*}
|C_{\rho,\beta}(z_0)\cap\big\{Q^+_{1,\beta}: \M_{Q^+_{2,\beta}}(|\nabla u|^2)> K\big\}|
&\leq |C_{3\rho,\beta}(\bar z_0)\cap\big\{Q^+_{1,\beta}: \M_{Q^+_{2,\beta}}(|\nabla u|^2)> K\big\}| \\
&\leq 3^{-2n-2}M_0^{-1}q_0 |C_{3\rho,\beta}(\bar z_0)|\\
&\leq q_0 |C_{\rho,\beta}(z_0)|.
\end{align*}
\smallskip
Therefore, \eqref{density-2} holds. The proof of the lemma is completed.
\end{proof}

\smallskip
As a consequence of Lemma \ref{measure density-bdry}, we obtain the following corollary.
\begin{corollary}\label{corollary-1-bdry}
Assume that the assumptions of Lemma \ref{measure density-bdry} hold. If 
\[
|C_{\rho,\beta}(z_0)\cap\left\{Q^+_{1,\beta}:\ \M_{Q^+_{2\Lambda,\beta}}(|\nabla u|^2)> K\right\}|> q_0 |C_{\rho,\beta}(z_0)|,
\]
then
\[
C_{\rho,\beta}(z_0)\cap Q^+_{1,\beta}\subset\left\{Q^+_{1,\beta}: \M_{Q^+_{2\Lambda,\beta}}(|\nabla u|^2)> 1\right\}\cup\left\{\M_{Q^+_{2\Lambda,\beta}}(|F|^2)> \hat \kappa^2\right\}.
\]
\end{corollary}
By adapting the proof in Lemma \ref{covering}, we have the covering lemma on the boundary.
\begin{lemma}[Krylov-Safonov covering lemma]\label{covering-bdry} Let $M_0 \geq 1$ and $\beta$ satisfy  \eqref{beta-cond}. Then, for every $q_0\in (0,1)$, there is $\gamma_2 = \gamma_2(n, M_0)>0$ such that the following assertions hold. Suppose that $S\subset Q^+_{1,\beta}$ is an open measurable set with $|S|\leq q_0|Q^+_{1,\beta}|$. In addition, we denote by the family of weighted cylinders
\[
\A=\big\{C_{\rho,\beta}(z):\ z\in Q_{1,\beta}^+,\ \rho>0,\ |C_{\rho,\beta}(z)\cap S|> q_0|C_{\rho,\beta}(z)|\ \big\}.
\]
Then, the open set 
\begin{equation}\label{def-E-bdry}
E^+ = \bigcup_{C_{\rho,\beta}(z) \in \mathcal{A}} C_{\rho,\beta}(z) \cap Q^+_{1,\beta}
\end{equation}
satisfies
\begin{equation*}
|S\setminus E^+|=0 \quad \text{and}\quad |S|\leq \gamma_2 q_0|E^+|.
\end{equation*}
\end{lemma}

Based on this covering lemma and Corollary \ref{corollary-1-bdry}, we derived the following lemma, which provides decay estimates of the measure of the level-sets of $\M_{Q^+_{2,\beta}}(|\nabla u|^2)$.

\begin{lemma}\label{measure decay lemma-bdry} Suppose that the assumptions in Lemma \ref{measure density-bdry} hold. 
Let $K$ be the constant from Lemma \ref{measure density-bdry}, and define
\[
S=\left\{\, Q_{1,\beta}^+ : \ \M_{Q_{2\Lambda,\beta}^+}(|\nabla u|^2)> K \,\right\}.
\]
If $|S|\leq q_0|Q_{1,\beta}^+|$, then for all $m\in \N$,  
\begin{equation}
\begin{aligned}\label{measure decay-bdry}
|\big\{Q^+_{1,\beta}: \M_{Q^+_{2,\beta}}(|\nabla u|^2)  > K^m\big\}| \leq  &l_0^{m}|\big\{Q^+_{1,\beta}: \M_{Q^+_{2,\beta}}(|\nabla u|^2)> 1\big\}|\\
&+\sum_{i=1}^ml_0^{i}|\big\{Q^+_{1,\beta}: \M_{Q^+_{2,\beta}}(|F|^2)> K^{m-i}\hat \kappa^2\big\}|,
\end{aligned}
\end{equation}
where $l_0=\gamma_2q_0<1$ with $\gamma_2 = \gamma_2(n, M_0)$ defined in Lemma \ref{covering-bdry}.
\end{lemma}

\begin{proof}   
For $m=1$, let
$E^+$ be defined by \eqref{def-E-bdry}.
As the condition $|S|\leq q_0|Q_{1,\beta}^+|$ in the statement of Lemma \ref{covering-bdry} is satisfied, it then follows from Corollary \ref{corollary-1-bdry} that  
\[
E^+\subset \left\{Q^+_{1,\beta}:\ \M_{Q^+_{2\Lambda,\beta}}(|\nabla u|^2)> 1\right\}\cup\left\{Q^+_{1,\beta}:\ \M_{Q^+_{2\Lambda,\beta}}(|F|^2)> \hat \kappa^2\right\},
\]
and thus, \eqref{measure decay-bdry} follows from Lemma \ref{covering-bdry}. 

\smallskip
For general $m\in \N$, the assertion \eqref{measure decay-bdry} is proved by an induction argument with the scaling $u \mapsto u/K$ and $F \mapsto F/K$.
\end{proof}

\subsection{Proof of Theorem \ref{inter-theorem-bdry}.} Theorem \ref{inter-theorem-bdry} follows from a similar argument as in the proof of Theorem \ref{inter-theorem}, by using Lemma \ref{bdr-improved-Caccio-1}, the PDE \eqref{eqn-bdry}, Lemma \ref{measure density-bdry}, Lemma \ref{covering-bdry}, and Lemma \ref{measure decay lemma-bdry} in place of the corresponding ones in Section \ref{interior section}. Thus, we omit the proof.

\section{Global regularity estimates and proof of Theorem \ref{main-theorem}} \label{global-section} For $p \in (1, \infty)$, we denote by $\W^{1,p}_0(\Omega_T, \beta)$ the weighted Sobolev space consisting of functions $u \in L^p((0, T), W^{1,p}_0(\Omega))$ such that
\[
\beta(x) u_t \in L^p((0, T), W^{-1, p}(\Omega)).
\]
The space $\W^{1,p}_{0} (\Omega_T, \beta)$ is endowed with the norm
\begin{equation}\label{norm-W1p}
\|u\|_{\W^{1,p}_{0} (\Omega_T, \beta)} = \|u\|_{L^p((0, T), W^{1,p}_0(\Omega))} + \| \beta u_t\|_{L^p((0, T), W^{-1, p}(\Omega))}, 
\end{equation}
for $u \in \W^{1,p}_{0} (\Omega_T, \beta)$. Note that, unlike $\W^{1,p}_{*} (\Omega_T, \beta)$, the space $\W^{1,p}_{0} (\Omega_T, \beta)$ does not require its functions to be $0$ when $t=0$. A function $u \in \W^{1,p}_{0} (\Omega_T, \beta)$ is said to be a weak solution to the equation
\begin{equation} \label{main-eqn-2}
\left\{
\begin{aligned}
\beta(x) u_t - \textup{div}(\bA(x,t) \nabla u) &= \textup{div}(F) \quad  &&\text{in} \quad \Omega_T,\\
u &= 0  \quad &&\text{on} \quad \partial\Omega\times (0, T],
\end{aligned} \right.
\end{equation}
if
\begin{align*}
& -\int_{\Omega_T} \beta(x) u(x,t) \partial_t \varphi(x,t) dx dt + \int_{\Omega_T} \wei{\bA(x,t) \nabla u(x,t), \nabla \varphi (x,t)} dx dt \\
& = -\int_{\Omega_T} \wei{F(x,t), \nabla \varphi(x,t)} dx dt, \quad \forall \varphi \in C^\infty_0(\Omega_T).
\end{align*}
\smallskip
We note that there is no initial condition imposed on \eqref{main-eqn-2} at $t=0$. 

\smallskip
In this section, we first prove Theorem \ref{main-thm-2} below, which constitutes the main step in proving Theorem \ref{main-theorem}. This theorem provides spatial global regularity estimates for weak solutions to \eqref{main-eqn-2} without imposing any condition on the initial data.

\begin{theorem}\label{main-thm-2}Let $\nu \in (0,1)$, $M_0 \geq 1$, and $p \in [2, \infty)$. Then, there exists a constant $\delta = \delta (n, \nu, p, M_0) \in (0,1)$ sufficiently small such that the following assertions hold. Suppose that  \eqref{ellip-cond}, \eqref{beta-cond} hold, $\partial \Omega \in C^1$, and
\begin{equation}\label{small delta}
\Theta_{\bA}(R_0) + \Theta_{\beta} (R_0)< \delta
\end{equation}
for some $R_0 \in (0,1)$. Suppose also that $u \in \W^{1,2}_{0} (\Omega_T, \beta)$ is a weak solution to \eqref{main-eqn-2} with $F \in L^p(\Omega_T)^n$. Then, for any $T_0 \in (0, T)$, we have $u \in \W^{1,p}_0(\Omega\times (T_0, T),\beta)$,  and there exists a constant $N = N(n, \nu, p, M_0, R_0, \Omega, T_0, T) >0$ such that
\begin{equation} \label{global-time-cut-est}
\|u\|_{\W^{1,p}_0(\Omega\times (T_0, T),\beta)}\leq N\Big[ \|u\|_{L^2(\Omega_T)} + \|F\|_{L^p(\Omega_T)} \Big].
\end{equation}
\end{theorem}

\begin{proof} For the given $\nu \in (0,1)$, $M_0 \geq 1$, and $p \in [2, \infty)$, let 
\begin{equation}\label{choice of delta}
\delta=\min\left\{\delta_0 (n, \nu, p, M_0),\, \frac{\kappa_0(n, \nu, p, M_0)}{N_1} \right\}\in (0,1),
\end{equation}
where $\delta_0$ and $\kappa_0$ are the constants defined in Theorem \ref{inter-theorem} and Theorem \ref{inter-theorem-bdry}, respectively,  and $N_1=N_1(n, \nu, M_0)>0$ is the constant given in Lemma \ref{matrix flattening} below. We prove the theorem under this choice of $\delta$ in \eqref{choice of delta}.

\smallskip
As $\partial\Omega\in C^1$, there exists $R=R(\delta) \in (0, R_0)$ sufficiently small
such that the following assertion holds: for each $x_0=(x'_0, x_0^n)\in \partial \Omega$ with $ x_0' \in \R^{n-1}$, there exists a Lipschitz continuous function $\phi: \R^{n-1} \rightarrow \R$ such that, upon relabeling and reorienting the coordinate axes,
\[
\Omega\cap B_{R}(x_0) = \{(x', x^n) \in B_R(x_0):\ x^n \geq \phi(x')\}, \quad \phi(x_0') =0, \quad \nabla \phi(x_0') =0,
\]
and 
\begin{equation}\label{h-lip}
\|\nabla\phi\|_{L^\infty(\R^{n-1})}\leq \delta.
\end{equation}

\smallskip
Taking a linear translation, we can assume without loss of generality that $x_0=0$. Next, we define $\Phi: \R^{n}\rightarrow \R^n$ as 
\[
\Phi(x',x^n)=(x', x^n-\phi(x')), \quad (x', x^n) \in \R^{n-1} \times \R.
\]
It is straightforward that $\Phi$ is invertible and its inverse map $\Phi^{-1}$ is given by
\[
\Phi^{-1}(y',y^n)=(y', y^n+\phi(y')), \quad (y', y^n) \in \R^{n-1} \times \R.
\] 
Note that
\begin{equation*}
\nabla \Phi(x)= \begin{bmatrix}
I_{n-1}& 0 \\
-(\nabla \phi)^* &1
\end{bmatrix}(x),
\quad \text{and}\quad 
\nabla \Phi^{-1}(y)= \begin{bmatrix}
I_{n-1}& 0 \\
(\nabla \phi)^* &1
\end{bmatrix}(y),
\end{equation*}
where $(\nabla \phi)^*$ denotes the transpose matrix of the matrix $\nabla \phi$. As a consequence, due to $\delta\in (0,1)$, 
for each $y=(y', y^n)\in B_{r}(y_0)$ with any $y_0\in \R^n$ and any $r>0$, we have
\begin{align*}
|\Phi^{-1}(y)-\Phi^{-1}(y_0)|
&=\sqrt{|y'-y'_0|^2+|y^n+\phi(y')-y^n_0-\phi(y'_0)|^2}\\
&\leq \sqrt{|y'-y'_0|^2+(|y^n-y^n_0|+\delta r)^2}\\
&\leq \sqrt{r^2+2\delta r^2+\delta^2r^2}=(1+\delta)r<2r.
\end{align*}
Similarly, for each $x=(x', x^n)\in B_{r/2}(\Phi^{-1}(y_0))$, we have
\begin{align*}
|\Phi(x)-y_0|
&=|(x', x^n-\phi(x'))-(y'_0, y^n_0)|\\
&\leq \sqrt{|x'-y'_0|^2+|x^n-\phi(x')-(y^n_0+\phi(y'_0))+\phi(y'_0)|^2}\\
&\leq \sqrt{|x'-y'_0|^2+\Big(|x^n-(y^n_0+\phi(y'_0))|+\frac{\delta r}{2}\Big)^2}\leq r.
\end{align*}
Thus, from the last two estimates, we have the following inclusion property: for any $r>0$ and $y_0\in \R^n$, 
\begin{equation}\label{phi-lip-1}
 B_{r/2}(\Phi^{-1}(y_0))\subset  \Phi^{-1}(B_{r}(y_0))\subset B_{2r}(\Phi^{-1}(y_0)).
\end{equation}

\smallskip
Next, we set
\begin{equation} \label{flattening variables}
\begin{aligned}
&\tilde u(y,t)=u(\Phi^{-1}(y), t),\quad  \tilde \beta(y)=\beta(\Phi^{-1}(y)),\\
&\tilde \bA(y,t)=\nabla\Phi(\Phi^{-1}(y))\bA(\Phi^{-1}(y),t)\nabla \Phi(\Phi^{-1}(y))^*,\\ 
& \tilde F(y,t)=\nabla\Phi(\Phi^{-1}(y))F(\Phi^{-1}(y),t).
\end{aligned}
\end{equation}
Let $t_0\in (T_0, T]$, we write $z_0 = (x_0, t_0) \in \R^n \times \R$. Now, we choose $\rho=\rho(T_0, M_0)\in (0,\frac{R}{4})$ sufficiently small so that 
\begin{equation} \label{rho-def-global}
B_{2\Lambda\rho}^+\subset\Phi(\Omega\cap B_{R}), \quad Q^+_{2\Lambda\rho,\tilde \beta}(z_0)\subset B_{2\Lambda\rho}^+\times (0,t_0],
\end{equation}
and
\begin{equation}\label{map subset}
\Phi^{-1}(Q^+_{2\Lambda\rho,\tilde \beta}(z_0))\subset (\Omega\cap B_{R})\times (0,t_0].
\end{equation}
Then, $\tilde u \in \hat{\W}^{1,2}(Q^+_{2\Lambda\rho,\tilde \beta}(z_0))$ is a weak solution to the equation
\begin{equation*}
\left\{
\begin{aligned}
\tilde\beta(y) \tilde u_t - \textup{div}(\tilde\bA(y,t) \nabla \tilde u)  &= \textup{div}(\tilde F) \quad &&\text{in}\quad Q^+_{2\Lambda\rho,\tilde \beta}(z_0),\\
\tilde u&=0 \quad &&\text{on} \quad T_{\rho}\times \Gamma_{2\Lambda\tilde \beta, z_0}(2\Lambda\rho).
\end{aligned}\right.
\end{equation*}

Now, by  Lemma \ref{matrix flattening} below and due to the choice of $\delta$ in \eqref{choice of delta}, and with a linear translation and a dilation, we can apply Theorem \ref{inter-theorem-bdry} to yield
\begin{align*}
&\|\nabla \tilde u\|_{L^p(Q_{\rho, \tilde \beta}^+(z_0))}+\|\tilde \beta \tilde u_t\|_{L^p(\Gamma_{\tilde\beta, z_0}(\rho),\, W^{-1,p}(B_{\rho}^+))}\\
&\leq N\left(|Q_{2\Lambda\rho,\tilde\beta}^+(z_0)|^{\frac{1}{p}-\frac{1}{2}}\|\tilde u\|_{L^2(Q_{2\Lambda\rho,\tilde\beta}^+(z_0))}+\|\tilde F\|_{L^p(Q_{2\Lambda\rho,\tilde\beta}^+(z_0))}\right),
\end{align*}
where $N=N(n, \nu, p, M_0)$. Moreover, it follows from \eqref{phi-lip-1} that 
\[
\Omega\cap B_{\rho/2}\subset \Phi^{-1}(B^+_{\rho}),\quad \text{and}\quad \Omega_T\cap Q_{\rho/2,\beta}(z_0)\subset \Phi^{-1}(B^+_{\rho})\times \Gamma_{\tilde \beta, z_0}(\rho).
\]
By changing the variables from $y$ back to $x$, and noting that $\delta\in (0,1)$, we have
\[
\nabla u(x,t)=\nabla \Phi(x)^*\nabla \tilde u(y,t)\quad  \text{so that}\quad |\nabla u(x,t)|\leq (1+\delta)|\nabla \tilde u(y,t)|< 2|\nabla \tilde u(y,t)|. 
\]
Thus, combining all these  and \eqref{map subset}, we obtain
\begin{equation}\label{bdry esti}
\begin{aligned}
&\|\nabla u\|_{L^{p}(\Omega_T\cap [Q_{\rho/2, \beta}(z_0)])}+\|\beta u_t\|_{L^p(\Gamma_{\beta, z_0}(\rho/2),\, W^{-1,p}(\Omega\cap B_{\rho/2}))}\\
&\leq N\left(|\{\Omega\cap B_R\}\times (0,T)|^{\frac{1}{p}-\frac{1}{2}}\|u\|_{L^2(\{\Omega\cap B_{R}\}\times (0,T))}+\|F\|_{L^p(\{\Omega\cap B_{R}\}\times (0,T))}\right)
\end{aligned}
\end{equation}
for all $z_0=(x_0,t_0)$ with $x_0\in \partial \Omega$ and $t_0\in (T_0, T)$, where $N=N(n, \nu, p, M_0)$.

\smallskip
Lastly, due to the compactness of $\overline{\Omega} \times [T_0, T]$, we cover it by finitely many cylinders $\{Q_{\rho_i/2, \beta}(z_i)\}_{i=1}^m$ with some $m \in \N$ and $\rho_i \in (0, \rho)$. For each $i = 1, 2,\ldots, m$, we have either
\[
Q_{2\Lambda \rho_i,\beta}(z_i) \subset \Omega \times (0, T]
\]
or
\[
z_i = (x_i,t_i)\in \partial \Omega\times (T_0, T] \quad \text{with} \quad \Gamma_{\beta, z_i}(2\Lambda \rho_i) \subset (0, T].
\]
For the former, after taking a linear translation and a dilation, we apply Theorem \ref{inter-theorem}. For the latter, we can use \eqref{bdry esti}. Then, by adding the resulting estimates, we obtain \eqref{global-time-cut-est}. The proof of the theorem is completed.
\end{proof}
\smallskip
Now, we state the following lemma used in the proof of Theorem \ref{main-thm-2}. Although it is standard, it is not easy to locate suitable references for our setting here; we provide the proofs of the lemma in Appendix \ref{proof-beta flattening} for completeness.
\begin{lemma}\label{matrix flattening} For $\nu \in (0,1)$ and $M_0 \geq 1$, there exists a constant $N_1=N_1(n, \nu, M_0)>0$ such that the following assertions hold. Suppose that $\bA(x,t)$ is a measurable matrix-valued function satisfying \eqref{ellip-cond} and \eqref{small-ness},  and  $\beta(x)$ satisfies \eqref{beta-cond} and \eqref{small-ness} with some $\delta, R_0 \in (0,1)$. Let $\tilde \beta(y)$ be defined as in \eqref{flattening variables}, 
$\tilde \bA(y,t)$ be defined as in \eqref{flattening variables}.
Then, $\tilde \beta(y)$ satisfies
\begin{equation}\label{newbeta-cond}
\tilde\beta^{-1} \in A_{1+\frac{2}{n_0}} \quad \text{and} \quad [\tilde\beta^{-1}]_{A_{1+\frac{2}{n_0}}} \leq 2^{n+2}M_0,  \quad \text{for} \quad n_0 = \max\{n, 2\},
\end{equation}
and
\begin{equation}\label{condition-1}
\frac{1}{\tilde \beta(B_{r}(y_0))}\int_{B_{r}(y_0)} \Theta_{\tilde \beta, r}(y)^2  dy \leq N_1^2\delta^{2}\quad \text{for all}\quad y_0\in \R^{n},\ r\in (0, R_0/2),
\end{equation}
In addition, $\tilde \bA(y,t)$ satisfies \eqref{ellip-cond}, and
\begin{equation}\label{condition-2}
 \fint_{Q_{r, \tilde\beta}^+(\tilde z_0)} \Theta_{\tilde \bA, r, y_0} (y,t)^2 dydt \leq N_1^2 \delta^{2},
\end{equation}
for all $r\in (0, R_0/2)$ and all $\tilde z_0=(y_0, t_0) \in \overline{Q}_{\rho,\, \tilde\beta}^+$, where $\rho>0$ is the number such that \eqref{rho-def-global} and \eqref{map subset} hold.
\end{lemma}

\smallskip
Finally, we provide the proof of Theorem \ref{main-theorem}.
\begin{proof}[Proof of Theorem \ref{main-theorem}] We first consider the case $p =2$. By Lemma \ref{embedd-lemma}, we see that $u \in L^2(\Omega_T, \beta)$ if $u \in \W^{1,2}_*(\Omega_T, \beta)$. Hence, if $u \in \W^{1,2}_*(\Omega_T, \beta)$ is a weak solution to \eqref{main-eqn}, by using Steklov's average as in \cite[p. 18]{DiB}, we can formally take $u$ as the test function to equation \eqref{main-eqn} to obtain
\[
\frac{1}{2}\frac{d}{dt}\int_{\Omega} u(x,t)^2 \beta(x) dx + \int_{\Omega}\wei{\bA(x,t) \nabla u, \nabla u} dx = -\int_{\Omega} \wei{F, \nabla u} dx.
\]
From this, \eqref{ellip-cond}, and the standard energy estimate using Young's inequality, we obtain
\[
\sup_{t\in (0,T)}\int_{\Omega} u(x,t)^2 \beta(x) dx + \int_{\Omega_T} |\nabla u(x,t)|^2 dxdt \leq N \int_{\Omega_T} |F(x,t)|^2 dxdt
\]
for $N = N(n, \nu)>0$. From this, and due to the homogeneous boundary condition, we can apply the Poincar\'e inequality to get
\[
\|u\|_{L^2((0, T), W^{1,2}(\Omega))} \leq N \|F\|_{L^2(\Omega_T)}.
\]
Then, using the PDE in \eqref{main-eqn}, we infer that
\begin{equation} \label{main-p=2}
\|u\|_{\W^{1,2}_{*}(\Omega_T, \beta)} \leq N \|F\|_{L^2(\Omega_T)}.
\end{equation}
Hence, \eqref{main-thm-est} is proved. Note also that the existence of a solution $u \in \W^{1,2}_{*}(\Omega_T, \beta)$ can be derived from the Galerkin method.  The uniqueness of the solution also follows directly from \eqref{main-thm-est}.  Theorem \ref{main-theorem} is then proved when $p=2$.

\smallskip
Next, we consider the case $p >2$. In this case, note that as $F \in L^p(\Omega_T)^n$, we see that $F \in L^2(\Omega_T)^n$. Then, it follows from the previous case that there is a unique weak solution $u \in \W^{1,2}_{*}(\Omega_T, \beta)$ to \eqref{main-eqn}, and \eqref{main-p=2} holds. It then remains to prove that $u \in \W^{1,p}_{*}(\Omega_T, \beta)$ and \eqref{main-thm-est} holds. To this end,  we trivially extend $u$, $F$ on $\Omega\times (-1,0)$, and extend $\bA(x,t)$ to be an identity matrix function on $\Omega\times (-1,0)$. We note that $u \in \W^{1,2}_{*}(\Omega \times (-1, T), \beta)$, and it is a weak solution to the equation
\begin{equation} \label{main-eqn-extend}
\left\{
\begin{array}{cccl}
\beta(x) u_t - \textup{div}(\bA(x,t) \nabla u) & =& \textup{div}(F) & \quad  \text{in} \quad \Omega \times (-1, T), \\[4pt]
u & = & 0 & \quad \text{on} \quad \partial_p ( \Omega \times (-1, T)).
\end{array} \right.
\end{equation}
Now, let $\delta = \delta(n, \nu, p, M_0) \in (0,1)$ be the number defined in Theorem \ref{main-thm-2}. Note also that the conditions \eqref{ellip-cond}, \eqref{beta-cond}, and \eqref{small-ness} hold for the coefficients in \eqref{main-eqn-extend}. We then apply Theorem \ref{main-thm-2} to infer that $u \in \W^{1,p}_{*}(\Omega_T)$  
\[
\|u\|_{\W^{1,p}_{*}(\Omega_T)} \leq N \Big[ \| u\|_{L^2(\Omega_T)} + \|F\|_{L^p(\Omega_T)} \Big].
\]
Then, \eqref{main-thm-est} follows from this last estimate and \eqref{main-p=2}. The proof of the theorem when $p > 2$ is completed.

\smallskip
Lastly, we consider the case $p \in (1, 2)$. In this case, the existence, uniqueness, and the estimate \eqref{main-thm-est} can be obtained by the duality argument using the case $p \in (2, \infty)$. As the argument is standard, we skip the details. The proof of the theorem is completed.
\end{proof}
\appendix 
\section{Proof of Lemma \ref{quasi-metric lemma}.}\label{Appendix A}
\begin{proof}
It follows directly from the definition of $\rho_{\beta}$  that $\rho_{\beta}$ satisfies the non-negative, identity, and symmetry properties. 
Thus, we only need to show \eqref{quasi-tri}, that is, there exists $\Lambda \geq 1$ such that 
\begin{equation*}
\rho_{\beta}(z_0, z)\leq \Lambda \big[\rho_{\beta}(z_0, \bar{z})+\rho_{\beta}(\bar{z} ,z)\big], \quad \forall\, z_0,\, z ,\, \bar{z} \in \R^{n}\times \R.
\end{equation*}
To this end, let $z_0 = (x_0, t_0),\, z=(x,t)$, and $\bar{z} = (\bar{x}, \bar{t})$ be three distinct points in $\R^n \times \R$, and we denote 
\[
r= \rho_{\beta}(z_0, \bar{z}) + \rho_{\beta}(\bar{z},  z).
\]
We assume without loss of generality that
\begin{equation}\label{distance assumption}
\rho_{\beta}(z_0, z)> r \quad \text{and} \quad t_0\geq t.
\end{equation}
It then follows from \eqref{rho def} that
\begin{equation} \label{rho-zz-0-app}
\rho_\beta(z_0, z) = \max\{|x_0 - x|, h_{x_0}^{-1}(t_0 -t)\}.
\end{equation}
We split the proof into the following two cases depending on the relationship between $\bar{t}$ and $t_0$.

\smallskip
\noindent
\textbf{Case I}. We consider the case that $\bar{t} \geq t_0$. In this case, we have $\bar{t} >t_0\geq  t$ and therefore
\[
\rho_{\beta}(z_0, \bar{z}) = \max\{|x_0 - \bar{x}|, h^{-1}_{\bar{x}} (\bar{t} - t_0) \} \quad \text{and} \quad \rho_{\beta}(\bar{z}, z) =\max\{|\bar{x} - x|, h_{\bar{x}}^{-1}(\bar{t} -t) \}.
\]
By this and the definition of $r$, it follows that
\begin{equation}\label{space-bdd}
 |x_0-x| \leq |x_0 -\bar{x}| + |\bar{x} -x|\leq r \quad \text{and}  \quad \bar{t} -t \leq h_{\bar{x}} (\rho_{\beta}(\bar{z}, z)) \leq h_{\bar{x}}(r).
\end{equation}
As $|x_0-\bar x|\leq r$ from the first assertion of \eqref{space-bdd}, we note that $B_{r} (\bar{x}) \subset B_{2r} (x_0)$. Due to this and the nonnegativity of  $\beta$, we infer that
\begin{equation*}
\left\{
   \begin{aligned}
\frac{1}{2} r \beta(B_{r}(\bar{x}))
&\leq \frac{1}{2}(2r)\beta(B_{2r}(x_0))\quad &&\text{if}\quad n=1,\\
\sigma_n^{-\frac{2}{n}}[\beta^{\frac{n}{2}}(B_{r}(\bar{x}))]^{\frac{2}{n}}
&\leq \sigma_n^{-\frac{2}{n}}[\beta^{\frac{n}{2}}(B_{2 r}(x_0))]^{\frac{2}{n}}\quad &&\text{if}\quad n\geq 2.
\end{aligned}\right.
\end{equation*}
Consequently, $h_{\bar{x}} (r) \leq h_{x_0}(2r)$. From this and as $t_0 - t < \bar{t} - t$, we infer from the second assertion in \eqref{space-bdd} that $t_0 - t \leq  h_{x_0}(2r)$ and therefore
\begin{equation}\label{time-bdd}
h_{x_0}^{-1}(t_0- t)\leq  2r.
\end{equation}
Combing \eqref{rho-zz-0-app}, the first assertion in \eqref{space-bdd}, and \eqref{time-bdd}, we infer that 
\[
\rho_{\beta}(z_0, z)\leq  2r  = 2 \Big[ \rho_{\beta}(z_0, \bar{z})+\rho_{\beta}(z, \bar{z})\Big].
\]
Therefore, the assertion \eqref{quasi-tri} holds with $\Lambda =2$.

\smallskip
\noindent
\textbf{Case II}. We consider the case that $\bar{t} < t_0$. In this case, we have
\begin{equation} \label{d-1-0924}
\rho_{\beta}(z_0, \bar{z}) = \max\{|x_0 - \bar{x}|, h^{-1}_{x_0} (t_0 -\bar{t} ) \}.
\end{equation}
Note that if $|x_0- x|=\rho_{\beta}(z_0, z)$, then by using $|x_0-\bar{x}|\leq \rho_{\beta}(z_0, \bar{z})$, and the triangle inequality, we see that
\[
\rho_{\beta}(\bar{z}, z)\geq |x-\bar{x}| \geq |x- x_0| - |x_0 -\bar{x}| \geq \rho_{\beta}(z_0, z) - \rho_{\beta}(z_0, \bar{z}),
\]
which contradicts \eqref{distance assumption}. Therefore, $|x_0-x| < \rho_{\beta}(z_0, z)$. This means that
\begin{equation} \label{d-0-0924}
\rho_{\beta}(z_0, z) = h_{x_0}^{-1}(t_0 - t) \quad \text{or} \quad t_0 -t = h_{x_0}(\rho_{\beta}(z_0, z)).
\end{equation}
Geometrically, this means that the point $z$ is on the bottom of $\overline{Q}_{\rho_{\beta}(z_0, z)}(z_0)$. Since $h_{x_0}$ is strictly increasing, it follows from \eqref{d-1-0924} and \eqref{distance assumption} that
\[
t_0 -\bar{t} \leq h_{x_0} (\rho_{\beta}(z_0, \bar{z}) )< h_{x_0} (r) < h_{x_0}(\rho_{\beta}(z_0, z)) = t_0 -t.
\]
Hence, $t < \bar{t} \leq t_0$ and
\begin{equation} \label{d-2-0924}
\rho_{\beta}(\bar{z}, z) =\max\{|\bar{x} - x|, h_{\bar{x}}^{-1}(\bar{t} -t) \}.
\end{equation}

\smallskip
Now, let us denote $\bar{\Lambda}= 2^{\frac{1}{2\zeta_0}}N_2^{\frac{1}{n\zeta_0}}$, where $N_2=N_2(n,M_0)\geq 1$ and $\zeta_0=\zeta_0(n,M_0)\in (0,1)$ are given in $\textup{(ii)}$ of Lemma \ref{property}. We claim that 
\begin{equation}\label{quasi-claim}
\rho_{\beta}(z_0, z) \leq \bar{\Lambda} \max\big\{\rho_{\beta}(z_0, \bar{z}),\ \rho_{\beta}(z,\bar{z})\big\}.
\end{equation}
Indeed, suppose that \eqref{quasi-claim} does not hold, then 
\begin{equation} \label{Lambda-bar}
\frac{\rho_{\beta}(z_0,\bar{z})}{\rho_{\beta}(z_0, z)} <  \frac{1}{\bar{\Lambda}} \quad \text{and} \quad  \frac{\rho_{\beta}(z,\bar{z})}{ \rho_{\beta}(z_0, z) }< \frac{1}{\bar{\Lambda}}.
\end{equation}
Due to \eqref{distance assumption}, we observe that
\[
B_{\rho_{\beta}(z_0,\bar{z})}(x_0) \subset B_{\rho_{\beta}(z_0, z)}(x_0) \quad \text{and}\quad B_{\rho_{\beta}(\bar{z}, z)}(\bar{x})\subset B_{\rho_{\beta}(z_0, z)}(x_0).
\]
Thus, by $\textup{(ii)}$ of Lemma \ref{property} and the choice of $\bar{\Lambda}$ in \eqref{quasi-claim}, and \eqref{Lambda-bar}, we see that 
\[
\big[\beta^{\frac{n}{2}}\big(B_{\rho_{\beta}(z_0,\bar{z})}(x_0)\big)\big]^{\frac{2}{n}}+\big[\beta^{\frac{n}{2}}\big(B_{\rho_{\beta}(\bar{z}, z)}(\bar{x})\big)\big]^{\frac{2}{n}}< \big[\beta^{\frac{n}{2}}\big(B_{\rho_{\beta}(z_0,z)}(x_0)\big)\big]^{\frac{2}{n}} \quad \text{for}\quad n\geq 2.
\]
For $n=1$, by a similar argument and with \eqref{distance assumption}, we also obtain  
\[
\rho_{\beta}(z_0, \bar{z})\beta\big(B_{\rho_{\beta}(z_0, \bar{z})}(x_0)\big)+\rho_{\beta}(\bar{z}, z)\beta\big(B_{\rho_{\beta}(\bar{z}, z)}(\bar{x})\big)<\rho_{\beta}(z_0,z)\beta\big(B_{\rho_{\beta}(z_0,z)}(x_0)\big).
\]
Then, it follows from \eqref{heights} that  
\[
h_{x_0}(\rho_{\beta}(z_0, \bar{z}))+ h_{\bar{x}}(\rho_{\beta}(z,\bar{z})) <h_{x_0}(\rho_{\beta}(z_0,z)).
\]
This estimate, \eqref{d-1-0924}, \eqref{d-0-0924}, and \eqref{d-2-0924} imply that
\[
(t_0 -\bar{t}) + (\bar{t} - t) < t_0 -t,
\]
which is impossible. Hence,
\[
\rho_{\beta}(z_0, z)\leq \bar{\Lambda} \big[\rho_{\beta}(z_0,\bar{z})+\rho_{\beta}(\bar{z}, z)\big].
\]

\smallskip
In conclusion, by combining {\bf Case I} and {\bf Case II} and taking 
\[
\Lambda=\max\{2,\,\bar{\Lambda}\}=\max\{2,\ 2^{\frac{1}{\zeta_0}}N_2^{\frac{1}{n\zeta_0}}\},
\] 
we obtain \eqref{quasi-tri}. The proof of the lemma is completed.
\end{proof}
\section{Proof of Lemma \ref{covering}.} \label{Appendix B}
\begin{proof} We adapt the proof in \cite{Krylov-Safonov} and modify the argument in \cite[Lemma 3.8]{CJP}. For each $z_0\in S\subset Q_{1,\beta}$, it follows from Lemma \ref{quasi-metric lemma} and \eqref{cylinder-quasi} that
\[
\rho_{\beta}(z_0, z)\leq \Lambda [\rho_{\beta}(z_0, 0)+\rho_{\beta}(z, 0)]\leq 2\Lambda \qquad\text{for all}\quad z\in Q_{1,\beta},
\]
where $\Lambda=\Lambda(n, M_0)\geq 1$ defined in Lemma \ref{quasi-metric lemma}. This, together with \eqref{inclusions}, implies that
\[
Q_{1,\beta}\subset C_{4\Lambda, \beta}(z_0) \qquad \text{for all}\quad z_0\in Q_{1,\beta}.
\]
As $|S| \leq q_0 |Q_{1,\beta}|$, we infer that
\begin{equation}\label{large radius}
|C_{\rho,\beta}(z_0)\cap S|\leq |S|\leq q_0|Q_{1,\beta}|\leq q_0|C_{\rho, \beta}(z_0)| \qquad \text{for all}\quad \rho\geq 4\Lambda.
\end{equation}
Therefore, for almost every $z_0 = (x_0, t_0) \in S$, the Lebesgue point theorem together with \eqref{large radius} yields the existence of a radius
\begin{equation}\label{radius}
0 < r(z_0) < 4\Lambda
\end{equation}
such that
\[
|C_{r(z_0),\beta}(z_0) \cap S| = q_0 |C_{r(z_0),\beta}(z_0)|
\]
and
\begin{equation}\label{density-1}
|C_{\rho,\beta}(z_0) \cap S| < q_0 |C_{\rho,\beta}(z_0)| \quad \text{for all } \rho > r(z_0).
\end{equation}
\smallskip
\noindent
\textbf{Step 1}. We claim that there exists a constant $N = N(n, M_0) > 0$ such that
\begin{equation}\label{covering-claim}
|C_{r(z_0),\beta}(z_0)| \leq N \, |C_{r(z_0),\beta}(z_0) \cap Q_{1,\beta}| \qquad \text{for all}\quad z_0 = (x_0, t_0) \in S.
\end{equation}
\smallskip
To prove the claim, we note that as $x_0 \in B_1$, we infer from \eqref{radius} that
\begin{equation}\label{spatial}
|B_{r(z_0)}(x_0)| \leq 8^n\Lambda^n \, |B_{r(z_0)}(x_0) \cap B_1|.
\end{equation}
Now, let us denote
\[
h = \frac{1}{2} r(z_0)^2 \Psi_{\beta, x_0}(r(z_0)),
\]
which is the half height of the cylinder $C_{r(z_0), \beta}(z_0)$. We then consider the following two cases.

\smallskip
\noindent
\textbf{Case 1: $h \leq \Psi_{\beta, 0}(1)$.}  
Since $t_0 \in (-\Psi_{\beta, 0}(1), 0)$, it follows that
\begin{equation}\label{time-1}
\big|\,(t_0 - h,\, t_0 + h) \cap (-\Psi_{\beta, 0}(1),\, 0)\,\big| 
\geq \frac{1}{2} |(t_0 - h,\, t_0 + h)| = h.
\end{equation}

\smallskip
\noindent
\textbf{Case 2: $h > \Psi_{\beta, 0}(1)$.}  
Since $t_0 \in (-\Psi_{\beta, 0}(1), 0)$, we have
\[
(-\Psi_{\beta, 0}(1), 0) \subset (t_0 - h, t_0 + h),
\]
so that
\begin{equation*}
\big|\,(t_0 - h, t_0 + h) \cap (-\Psi_{\beta, 0}(1), 0)\,\big| = \Psi_{\beta, 0}(1).
\end{equation*}
Moreover, since $B_{r(z_0)}(x_0) \subset B_3$, by the definition of $h$ and the doubling properties due to Lemma \ref{property}, we obtain
\[
2h = r(z_0)^2 \Psi_{\beta, x_0}(r(z_0)) \leq 9 \Psi_{\beta, 0}(3) \leq N(n, M_0)^{-1} \Psi_{\beta, 0}(1)
\]
for some $N(n, M_0) > 0$. Combining the last two estimates gives
\begin{equation}\label{time-2}
\big|\,(t_0 - h, t_0 + h) \cap (-\Psi_{\beta, 0}(1), 0)\,\big| \geq N(n, M_0) \, 2h.
\end{equation}
\smallskip
The claim \eqref{covering-claim} now follows from \eqref{spatial}, \eqref{time-1}, and \eqref{time-2}.

\smallskip
\noindent
\textbf{Step 2}. We  define
\[
\mathcal{B} = \{ C_{r(z),\beta}(z) : z \in S \}.
\]
Then $\mathcal{B}$ is an open cover of $S$. By the Vitali covering lemma, there exists a finite or countable, pairwise disjoint subcollection $\{C_{r(z_i),\beta}(z_i)\}_{i \in I} \subset \mathcal{B}$ such that
\[
S \subset \bigcup_{i \in I} C_{5 r(z_i),\beta}(z_i).
\]
\smallskip
\noindent
\textbf{Step 3.}  
Using \eqref{density-1}, the doubling property of $\beta$, the disjointness of $\{C_{r(z_i),\beta}(z_i)\}_{i\in I}$, and \eqref{covering-claim}, we obtain
\begin{equation*} \begin{aligned} |S| &\leq |\bigcup_{i\in I}C_{5r(z_i),\beta}(z_i)\cap S|\leq \sum_{i\in I}|C_{5r(z_i),\beta}(z_i)\cap S|\\ &\leq 5^nN(n,M_0)q_0\sum_{i\in I}|C_{r(z_i),\beta}(z_i)| &&\text{by \eqref{density-1} and doubling}\\ &\leq 5^nN(n,M_0)q_0\sum_{i\in I}|C_{r(z_i),\beta}(z_i)\cap Q_{1,\beta}| &&\text{by \eqref{covering-claim}}\\ &=5^nN(n,M_0)q_0 |\bigcup_{i\in I}C_{r(z_i),\beta}(z_i)\cap Q_{1,\beta}| &&\text{by disjoint property}\\ &\leq N(n, M_0)q_0|E|. 
\end{aligned} 
\end{equation*}
The lemma is proved.
\end{proof}
\section{Proof of Lemma \ref{matrix flattening}} \label{proof-beta flattening}
\begin{proof}
We first prove that $\tilde \beta(y)$ satisfies \eqref{newbeta-cond}, which is equivalent to show
\begin{equation}\label{locally A-p}
\left\{
\begin{aligned}
&(\tilde \beta^{-1})_{B_{r}(y_0)}(\tilde \beta)_{B_r(y_0)}\leq 4M_0\quad &&\text{for}\quad n=1,\\
&(\tilde \beta^{-1})_{B_{r}(y_0)}(\tilde \beta^{\frac{n}{2}})_{B_{r}(y_0)}^{\frac{2}{n}}\leq 2^{n+2}M_0\quad &&\text{for} \quad n\geq 2,
\end{aligned}\right.
\end{equation}
for all $B_{r}(y_0)\subset \R^n$. For any fixed $B_{r}(y_0)$, by taking a change of variables and using the fact that $\textup{det}(\nabla \Phi^{-1})=1$, and \eqref{phi-lip-1}, we get
\[
\tilde\beta^{-1}\left(B_r(y_0)\right)=\beta^{-1}\left(\Phi^{-1}(B_{r}(y_0))\right)\leq \beta^{-1}\left(B_{2r}(\Phi^{-1}(y_0))\right),
\]
which implies
\begin{align*}
(\tilde \beta^{-1})_{B_{r}(y_0)}
\leq\frac{|B_{2r}|}{|B_{r}|}(\beta^{-1})_{B_{2r}(\Phi^{-1}(y_0))}=2^{n}(\beta^{-1})_{B_{2r}(\Phi^{-1}(y_0))}.
\end{align*}
Similarly, 
\[
\left\{
\begin{aligned}
&(\tilde \beta)_{B_r(y_0)}\leq 2(\beta)_{B_{2r}(\Phi^{-1}(y_0))}\quad &&\text{for}\quad n=1,\\
&(\tilde \beta^{\frac{n}{2}})_{B_{r}(y_0)}^{\frac{2}{n}}\leq 2^{2}(\beta^{\frac{n}{2}})_{B_{2r}(\Phi^{-1}(y_0))}^{\frac{2}{n}}\quad &&\text{for} \quad n\geq 2.
\end{aligned}\right.
\]
Thus, \eqref{locally A-p} follows by combining the last two estimates. Therefore, \eqref{newbeta-cond} is proved.

\smallskip
Now, we prove \eqref{condition-1}. We start with the claim that there exists a constant $N_0=N_0(n, M_0)\geq 1$ such that
\begin{equation}\label{claim-c}
\frac{1}{\tilde\beta(B)}\int_{B}|\tilde\beta(y)-(\tilde\beta)_{B}|^2\tilde\beta^{-1}(y)\, dy\leq \frac{N_0}{\tilde\beta(B)}\int_{B}|\tilde\beta(y)-c|^2\tilde\beta^{-1}(y)\, dy
\end{equation}
holds for all $c\in\R$, and for all ball $B\subset \R^n$. To see this, by the triangle inequalities,
\begin{equation}\label{tri ineq}
\begin{aligned}
&\frac{1}{\tilde\beta(B)}\int_{B}|\tilde\beta(y)-(\tilde\beta)_{B}|^2\tilde\beta^{-1}(y)\, dy\\
&\leq \frac{2}{\tilde\beta(B)}\int_{B}|\tilde\beta(y)-c|^2\tilde\beta^{-1}(y)\, dy+\frac{2}{\tilde\beta(B)}\int_{B}|(\tilde\beta)_{B}-c|^2\tilde\beta^{-1}(y)\, dy
\end{aligned}
\end{equation}
for all $c\in \R$ and all $B\subset \R^n$. On the other hand, by H\"{o}lder's inequality, we notice that
\begin{align*}
|(\tilde\beta)_{B}-c|^2 & =\left|\frac{1}{|B|}\int_B\tilde \beta(y)-c\, dy\right|^2 \leq \left(\frac{1}{|B|}\int_B|\tilde\beta(y)-c|\, dy\right)^2\\
&\leq \left(\frac{1}{|B|}\int_B\tilde\beta(y)\, dy\right)\left(\frac{1}{|B|}\int_B|\tilde\beta(y)-c|^2\tilde \beta^{-1}(y)\, dy\right)\\
&=(\tilde\beta)_B\left(\frac{1}{|B|}\int_B|\tilde\beta(y)-c|^2\tilde \beta^{-1}(y)\, dy\right).
\end{align*}
Thus,
\begin{align*}
\frac{2}{\tilde\beta(B)}\int_{B}|(\tilde\beta)_{B}-c|^2\tilde\beta^{-1}(y)\, dy
&=\frac{2(\tilde\beta)_B\tilde\beta^{-1}(B)}{\tilde\beta(B)}\left(\frac{1}{|B|}\int_B|\tilde\beta(y)-c|^2\tilde \beta^{-1}(y)\, dy\right)\\
&=(\tilde \beta)_B(\tilde \beta^{-1})_B\left(\frac{2}{|\tilde \beta(B)|}\int_B|\tilde\beta(y)-c|^2\tilde \beta^{-1}(y)\, dy\right)\\
&\leq [\tilde \beta]_{A_2}\left(\frac{2}{|\tilde \beta(B)|}\int_B|\tilde\beta(y)-c|^2\tilde \beta^{-1}(y)\, dy\right)
\end{align*}
Plugging this into \eqref{tri ineq} and using Remark \ref{A-p-remark}-\textup{(ii)}, we achieve \eqref{claim-c} with $N_0=N_0(n, M_0)\geq 2$.

\smallskip
Now, from the claim \eqref{claim-c}, it is sufficient to prove that there is $N = N(n, M_0)>0$ such that for each $r \in (0, R_0/2)$ and any $y_0 \in \R^n$, there is $c \in \R$ such that
\begin{equation} \label{beta-os-change-var}
\frac{1}{\tilde \beta(B_{r}(y_0))}\int_{B_{r}(y_0)}|\tilde\beta(y)-c|^2\tilde\beta^{-1}(y)\, dy \leq N \delta^2.
\end{equation}
To this end, we note that
\begin{equation*}\label{equiva measure}
\tilde \beta(B_{r}(y_0))=\int_{B_{r}(y_0)}\tilde \beta(y)dy=\int_{\Phi^{-1}(B_{r}(y_0))}\beta(x)dx=\beta\big(\Phi^{-1}(B_{r}(y_0))\big).
\end{equation*}
Then, by using a change of variables, and  \eqref{flattening variables} and \eqref{phi-lip-1}, we see that 
\begin{align*}  
&\frac{1}{\tilde \beta(B_{r}(y_0))}\int_{B_{r}(y_0)}|\tilde\beta(y)-c|^2\tilde\beta^{-1}(y)\, dy\\  
&=\frac{1}{\beta\big(\Phi^{-1}(B_{r}(y_0))\big)}\int_{\Phi^{-1}(B_{r}(y_0))}|\beta(x)-c|^2\beta^{-1}(x)dx\\  
&\leq\frac{1}{\beta\big(\Phi^{-1}(B_{r}(y_0))\big)}\int_{B_{2r}(\Phi^{-1}(y_0))}|\beta(x)-c|^2\beta^{-1}(x)dx.
\end{align*}
From this, and  by taking $c=(\beta)_{B_{2r}(\Phi^{-1}(y_0))}$, we infer from \eqref{phi-lip-1},  the doubling property of $\beta$, and \eqref{small delta} that
\begin{align*}  
\frac{1}{\tilde \beta(B_{r}(y_0))}\int_{B_{r}(y_0)}|\tilde\beta(y)-c|^2\tilde\beta^{-1}(y)dy  
&\leq \frac{1}{\beta\big(\Phi^{-1}(B_{r}(y_0))\big)}\int_{B_{2r}(\Phi^{-1}(y_0))}|\beta(x)-c|^2\beta^{-1}(x)dx\\  
&\leq \frac{\beta\big(B_{2r}(\Phi^{-1}(y_0))\big)}{\beta\big(\Phi^{-1}(B_{r}(y_0))\big)}\delta^2\leq N(n, M_0)\delta^2.
\end{align*}
This implies \eqref{beta-os-change-var}, and the proof of \eqref{condition-1} is completed. 

\smallskip
We now move our attention to the assertions on  $\tilde \bA(y,t)$. By the definition of $\tilde \bA(y,t)$ in \eqref{flattening variables}, it is not difficult to find that $\tilde \bA(y,t)$ satisfies \eqref{ellip-cond}. It then remains to prove \eqref{condition-2}. By a direct computation, we find
\[
\tilde \bA(y,t)=\bA(\Phi^{-1}(y),t)+\mathbf{B}(\Phi^{-1}(y),t),
\]
where
\begin{align*}
\mathbf{B}=\begin{bmatrix}
0&\cdots& 0& -\sum_{j=1}^{n-1}a_{1j}\phi_{x_j} \\
\vdots&\ddots&\vdots&\vdots\\
0&\cdots& 0& -\sum_{j=1}^{n-1}a_{(n-1)j}\phi_{x_j}\\
-\sum_{i=1}^{n-1}a_{i1}\phi_{x_i}&\cdots&-\sum_{i=1}^{n-1}a_{i(n-1)}\phi_{x_i}&b_{nn}\\
\end{bmatrix},
\end{align*}
in which $a_{ij}$ are the entries of matrix $\bA(x,t)$, and \[
b_{nn}=\sum_{i,j=1}^{n-1}a_{ij}\phi_{x_i}\phi_{x_j}-2\sum_{j=1}^{n-1}a_{nj}\phi_{x_j}.
\]
Then, it follows from \eqref{ellip-cond} and \eqref{h-lip} that
\begin{equation} \label{B-est-10-16}
\|\bB\|_{L^\infty(\R^{n+1})} \leq N \delta \quad \text{with} \ N = N(n, \nu).
\end{equation}

Now, let $\tilde z_0=(y_0,t_0) \in Q_{\rho, \tilde \beta}^+$ and $r\in (0, R_0/2)$ be fixed. It follows from \eqref{phi-lip-1} that
\[
\Phi^{-1}(B^+_r(y_0))\subset B_{2r}(\Phi^{-1}(y_0))\cap \Omega.
\]
This, the definitions of $\Gamma_{\tilde \beta, y_0}$, and $\Gamma_{\beta, \Phi^{-1}(y_0)}$ in \eqref{cylinder-def} imply
\begin{equation}\label{phi-inclusion-2}  
\Phi^{-1}(Q_{r, \tilde\beta}^+(\tilde z_0))\subset Q_{2r,\beta}(\Phi^{-1}(\tilde z_0))\cap \Omega_T \quad \text{with}\quad \Phi^{-1}(\tilde z_0)=(\Phi^{-1}(y_0), t_0),
\end{equation}
and by the doubling property,
\begin{equation}\label{quotient mea}
\frac{|Q_{2r,\beta}(\Phi^{-1}(\tilde z_0))\cap \Omega_T|}{|Q^+_{r, \tilde\beta}(\tilde z_0)|}\leq N_0(n,M_0)
\end{equation}
for some $N_0=N_0(n, M_0)>1$.

\smallskip
Now, by the change of variables, and the triangle inequality, \eqref{phi-inclusion-2}, and \eqref{quotient mea}, we obtain
\begin{equation*}
\begin{aligned}
&\fint_{Q_{r, \tilde\beta}^+(\tilde z_0)}|\tilde \bA(y,t)-(\bA)_{B_{2r}(\Phi^{-1}(y_0))\cap \Omega}(t)|^2dydt\\
&\leq\frac{1}{|Q^+_{r, \tilde\beta}(\tilde z_0)|} \int_{Q_{2r,\beta}(\Phi^{-1}(\tilde z_0))\cap \Omega_T}|\bA(x,t)+\mathbf{B}(x,t)-(\bA)_{B_{2r}(\Phi^{-1}(y_0))\cap \Omega}(t)|^2dxdt\\
&\leq 2N_0(n, M_0)\fint_{Q_{2r,\beta}(\Phi^{-1}(\tilde z_0))\cap \Omega_T}|\bA(x,t)-(\bA)_{B_{2r}(\Phi^{-1}(y_0))\cap \Omega}(t)|^2dxdt\\
&\quad +2N_0(n, M_0)\fint_{Q_{2r,\beta}(\Phi^{-1}(\tilde z_0))\cap \Omega_T}|\bB(x,t)|^2 dxdt\\
&\leq 2N_0(n,M_0)\delta^2+2N_0(n, M_0)N(n,\nu)\delta^2,
\end{aligned}
\end{equation*}
where in the last step, we used \eqref{small delta}, and \eqref{B-est-10-16}. Now, note that
\begin{align*}
\fint_{Q_{r, \tilde\beta}^+(\tilde z_0)} \Theta_{\tilde \bA, r, y_0} (y,t)^2 dydt
&= \fint_{Q_{r, \tilde\beta}^+(\tilde z_0)}|\tilde \bA(y,t)-(\tilde \bA)_{B^+_r(y_0)}(t)|^2dydt\\
&\leq 4\fint_{Q_{r, \tilde\beta}^+(\tilde z_0)}|\tilde \bA(y,t)-(\bA)_{B_{2r}(\Phi^{-1}(y_0))\cap \Omega}(t)|^2dydt.
\end{align*}
Then \eqref{condition-2} follows from the last two estimates. The proof of the lemma is completed.
\end{proof}
\section*{Acknowledgement}
The research of T. Phan was partially supported by Simons Foundation, grant \# 769369. 

\end{document}